\documentclass{amsart}
\usepackage{amsthm,amsmath,amssymb,amscd}

\usepackage{tikz}
\usetikzlibrary{matrix,arrows,calc,snakes,patterns,decorations.markings}

\usepackage{stmaryrd}
\usepackage{mathrsfs}
\usepackage{eucal}
\usepackage{hyperref}
\numberwithin{equation}{subsection}

\DeclareMathOperator{\locM}{S}
\DeclareMathOperator{\diff}{d}
\DeclareMathOperator{\bfR}{R}
\DeclareMathOperator{\bfL}{L}
\DeclareMathOperator{\bfD}{D}
\DeclareMathOperator{\bfK}{K}
\DeclareMathOperator{\Ch}{Ch}

\DeclareMathOperator{\coH}{H}
\DeclareMathOperator{\Ho}{Ho}
\DeclareMathOperator{\inc}{inc}
\DeclareMathOperator{\fold}{fold}
\DeclareMathOperator{\even}{even}
\DeclareMathOperator{\odd}{odd}

\DeclareMathOperator{\ac}{ac}
\DeclareMathOperator{\cf}{cf}
\DeclareMathOperator{\sg}{sg}
\DeclareMathOperator{\Acyc}{Acyc}
\DeclareMathOperator{\dgProj}{dg-Proj}
\DeclareMathOperator{\dgHom}{dg-Hom}
\DeclareMathOperator{\dgInj}{dg-Inj}
\DeclareMathOperator{\co}{co}
\DeclareMathOperator{\sing}{sing}
\DeclareMathOperator{\ctr}{ctr}
\DeclareMathOperator{\opproj}{proj}
\DeclareMathOperator{\gproj}{G-proj}
\DeclareMathOperator{\ginj}{G-inj}
\DeclareMathOperator{\opinj}{inj}
\DeclareMathOperator{\Noeth}{Noeth}
\DeclareMathOperator{\Inj}{Inj}
\DeclareMathOperator{\compact}{c}
\DeclareMathOperator{\id}{id}

\DeclareMathOperator{\Proj}{Proj}
\DeclareMathOperator{\cone}{Cone}
\DeclareMathOperator{\gldim}{gl.dim}

\DeclareMathOperator{\Alg}{-Alg}
\DeclareMathOperator{\gen}{span}
\DeclareMathOperator{\subobj}{Subobj}
\DeclareMathOperator{\Perf}{Perf}
\DeclareMathOperator{\QCoh}{QCoh}
\DeclareMathOperator{\Mod}{-Mod}
\DeclareMathOperator{\smod}{-mod}
\DeclareMathOperator{\Mor}{Mor}

\DeclareMathOperator{\Hom}{Hom}

\DeclareMathOperator{\cof}{Cof}
\DeclareMathOperator{\fib}{Fib}
\DeclareMathOperator{\weak}{W}
\DeclareMathOperator{\bbar}{bar}

\DeclareMathOperator{\coker}{coker}
\DeclareMathOperator{\image}{im}
\DeclareMathOperator{\coimage}{coim}
\DeclareMathOperator{\filt}{filt-}
\DeclareMathOperator{\tot}{Tot}
\DeclareMathOperator{\ext}{Ext}

\newcommand {\ueven} {\ensuremath{^{\even}}}
\newcommand  {\uodd} {\ensuremath{^{\odd}}}

\newcommand   {\lac} {\ensuremath{_{\ac}}}
\newcommand   {\lcf} {\ensuremath{_{\cf}}}
\newcommand   {\lsg} {\ensuremath{_{\sg}}}
\newcommand   {\uco} {\ensuremath{^{\co}}}
\newcommand {\lsing} {\ensuremath{_{\sing}}}

\newcommand  {\uctr} {\ensuremath{^{\ctr}}}

\newcommand {\lproj} {\ensuremath{_{\opproj}}}
\newcommand {\uproj} {\ensuremath{^{\opproj}}}
\newcommand{\ugproj} {\ensuremath{^{\gproj}}}
\newcommand {\uginj} {\ensuremath{^{\ginj}}}
\newcommand  {\linj} {\ensuremath{_{\opinj}}}
\newcommand  {\uinj} {\ensuremath{^{\opinj}}}
\newcommand   {\inj} {\ensuremath{\calI}}
\newcommand  {\proj} {\ensuremath{\calP}}

\newcommand    {\ZZ} {\ensuremath{{\mathbb Z}}}
\newcommand  {\calO} {\ensuremath{{\mathcal O}}}
\newcommand  {\calX} {\ensuremath{{\mathcal X}}}
\newcommand  {\calY} {\ensuremath{{\mathcal Y}}}
\newcommand  {\calM} {\ensuremath{{\mathcal M}}}
\newcommand  {\calP} {\ensuremath{{\mathcal P}}}
\newcommand  {\calK} {\ensuremath{{\mathcal K}}}
\newcommand  {\calI} {\ensuremath{{\mathcal I}}}

\newcommand  {\calC} {\ensuremath{{\mathcal C}}}
\newcommand  {\calD} {\ensuremath{{\mathcal D}}}
\newcommand  {\calE} {\ensuremath{{\mathcal E}}}
\newcommand  {\calF} {\ensuremath{{\mathcal F}}}
\newcommand  {\calG} {\ensuremath{{\mathcal G}}}
\newcommand  {\calS} {\ensuremath{{\mathcal S}}}
\newcommand  {\calL} {\ensuremath{{\mathcal L}}}
\newcommand  {\calT} {\ensuremath{{\mathcal T}}}
\newcommand  {\calW} {\ensuremath{{\mathcal W}}}

\newcommand   {\invlim} {\ensuremath{\varprojlim}}
\renewcommand{\projlim} {\ensuremath{\varinjlim}}

\newcommand {\scA} {\ensuremath{{\mathscr A}}}
\newcommand {\scB} {\ensuremath{{\mathscr B}}}
\newcommand {\scE} {\ensuremath{{\mathscr E}}}
\newcommand {\scC} {\ensuremath{{\mathscr C}}}

\newcommand   {\p}    {\ensuremath{^{\prime}}}
\newcommand  {\pp}    {\ensuremath{^{\prime\prime}}}
\newcommand  {\ua}    {\ensuremath{^{\ast}}}
\newcommand  {\la}    {\ensuremath{_{\ast}}}
\newcommand  {\us}    {\ensuremath{^{\sharp}}}
\newcommand  {\ui}    {\ensuremath{^{-1}}}
\newcommand  {\uc}    {\ensuremath{^{\compact}}}
\newcommand  {\ol} [1]{\ensuremath{\overline{#1}}}

\newcommand   {\sbar} {\ensuremath{\underline{\bbar}}}

\newcommand  {\uprod} {\ensuremath{^{\Pi}}}      
\newcommand {\uoplus} {\ensuremath{^{\oplus}}}  
\newcommand  {\monad} {\ensuremath{{\perp}}}

\newcommand{\rec}{\ensuremath{
    \begin{tikzpicture}[baseline]
      \useasboundingbox (0,0) rectangle (4.75mm,2mm);
      \draw[->](0.5mm,1.1mm)--(4.25mm,1.1mm);
      \draw[->](4.25mm,0.2mm)--(0.5mm,0.2mm);
      \draw[->](4.25mm,2mm)--(0.5mm,2mm);
    \end{tikzpicture}
  }
}
\newcommand{\bigrec}{\ensuremath{
    \begin{tikzpicture}[baseline]
      \useasboundingbox (0,0) rectangle (9.75mm,2.3mm);
      \draw[->](0.5mm,1.2mm)--(9.25mm,1.2mm);
      \draw[->](9.25mm,0.1mm)--(0.5mm,0.1mm);
      \draw[->](9.25mm,2.3mm)--(0.5mm,2.3mm);
    \end{tikzpicture}
  }
}

\theoremstyle{plain}
\newtheorem{prop}{Proposition}[subsection]
\newtheorem{lem}[prop]{Lemma}
\newtheorem{theorem}[prop]{Theorem}
\newtheorem{cor}[prop]{Corollary}
\newtheorem{fact}[prop]{Fact}

\theoremstyle{definition}
\newtheorem{definition}[prop]{Definition}
\newtheorem{notation}[prop]{Notation}

\theoremstyle{remark}
\newtheorem{rem}[prop]{Remark}
\newtheorem{ex}[prop]{Example}

\title{Models for singularity categories}
\author{Hanno Becker}
\address{
Mathematisches Institut Universit\"at Bonn\\
Endenicher Allee 60\\
53115 Bonn\\
}
\email{habecker@math.uni-bonn.de}
\date{\today}

\begin{document}

\begin{abstract}
In this article we construct various models for singularity categories of modules over differential graded
  rings. The main technique is the connection between abelian model structures, cotorsion pairs and deconstructible
  classes, and our constructions are based on more general results about localization and transfer of abelian model
  structures. We indicate how recollements of triangulated categories can be obtained model categorically,
  discussing in detail Krause's recollement $\bfK\lac(\Inj(R))\to\bfK(\Inj(R))\to\bfD(R)$. In the special 
  case of curved mixed $\ZZ$-graded complexes, we show that one of our singular models is Quillen equivalent to
  Positselski's contraderived model for the homotopy category of matrix factorizations. 
\end{abstract}
\maketitle

\section*{Introduction}
\normalfont
Let $R$ be a Noetherian ring and $\bfD\lsg(R)= \bfD^b(R\smod)/\Perf(R)$ its singularity category. We
ask if it is possible to realize $\bfD\lsg(R)$ as the homotopy category of a stable model category attached to
$R$. Firstly, the singularity category is essentially small, whereas the homotopy category of a model category in the sense of
\cite{Hovey_ModelCategories} always has arbitrary small coproducts \cite[Example 1.3.11]{Hovey_ModelCategories}. This forces us to think first about
how to define a ``large'' singularity category for $R$ (admitting arbitrary small coproducts) in which $\bfD\lsg(R)$
naturally embeds. Secondly, if this is done, we can try to find a model for this large singularity category. 

Given a locally Noetherian Grothendieck category $\scA$ with compactly generated derived category $\bfD(\scA)$, Krause
\cite{Krause_StableDerived} proved that the singularity category $\bfD^b(\Noeth(\scA))/\bfD(\scA)\uc$ of $\scA$ (the
Verdier quotient of the bounded derived category of Noetherian objects of $\scA$ by the subcategory of compact 
objects of $\bfD(\scA)$) is up to direct summands equivalent to the subcategory of compact objects in the homotopy
category $\bfK\lac(\Inj(\scA))$ of acyclic complexes of injectives, and that there is even a recollement
$\bfK\lac(\Inj(\scA))\rec\bfK(\Inj(\scA))\rec\bfD(\scA)$. This suggests firstly that we should attempt 
to construct a model for $\bfK\lac(\Inj(\scA))$ and secondly that such a model might be obtained by localizing a
suitable model for $\bfK(\Inj(\scA))$ with respect to $\bfD(\scA)$, whatever this should mean precisely. 

If $\scA=R\Mod$ for a Noetherian ring $R$, Positselski \cite[Theorem 3.7]{Positselski_TwoKindsOfDerivedCategories}
showed that $\bfK(\Inj(\scA))$ is equivalent to what he calls the \textit{coderived category}
$\bfD\uco(R)$ of $R$, defined as the Verdier quotient $\bfK(R)/\Acyc\uco(R)$, where
$\Acyc\uco(R)$ is the localizing subcategory of $\bfK(R)$ generated by the total complexes of short exact sequences of complexes of
$R$-modules; objects of $\Acyc\uco(R)$ are called \textit{coacyclic complexes}. In particular, Krause's ``large''
singularity category $\bfK\lac(\Inj(R))$ is equivalent to a Verdier quotient $\bfD\uco(R)/\bfD(R)$. 

All in all, the last paragraphs suggest that a model for the singularity category could be obtained by lifting the quotient
$\bfD\uco(R)/\bfD(R)$ to the world of model categories. For $\bfD(R)$ there are the 
well-known projective and injective models, and for $\bfD\uco(R)$ a model has been constructed by Positselski
\cite{Positselski_TwoKindsOfDerivedCategories}. Moreover, these models are \textit{abelian},
i.e. they are compatible with the abelian structure of $\Ch(R\Mod)$ in the sense of \cite[Definition
2.1]{Hovey_Cotorsion}. By \cite[Theorem 2.2]{Hovey_Cotorsion} an abelian model structure is completely determined by the
classes $\calC$, $\calW$, $\calF$ of cofibrant, weakly trivial and fibrant objects, respectively, and the triples
$(\calC,\calW,\calF)$ arising in this way are precisely those for which $\calW$ is thick and both
$(\calC,\calW\cap\calF)$ and $(\calC\cap\calW,\calF)$ are complete cotorsion pairs (see Definitions
\ref{def_thick} and \ref{def_cotorsionpair} for the definition of thickness and cotorsion pairs, respectively). For
example, in the injective model $\calM\uinj(R)$ for $\bfD(R)$, everything is 
cofibrant, the weakly trivial objects $\calW\uinj$ are the acyclic complexes and the fibrant objects $\calF\uinj$ are the
dg-injectives. In Positselski's coderived model $\calM\uco(R)$ for $\bfD\uco(R)$, again everything is cofibrant, but the
weakly trivial objects $\calW\uco$ are the coacyclic complexes (see Proposition \ref{prop_cocontraderivedmodel}) and the fibrant objects
$\calF\uco$ are the componentwise injective complexes of $R$-modules. In particular, we see that both model structures
are \textit{injective} in the sense that everything is cofibrant, and that $\calW\uco(R)\subset\calW\uinj(R)$ and
$\calF\uinj(R)\subset\calF\uco(R)$. 

In order to construct the desired localization, we show (Proposition \ref{prop_localization}) that given an abelian
category $\scA$ with two injective abelian model structures $\calM_i = (\scA,\calW_i,\calF_i)$, $i=1,2$, satisfying
$\calF_2\subset\calF_1$ (hence $\calW_1\subset\calW_2$), there is another new abelian model structure $\calM_1/\calM_2$ on $\scA$ with $\calC = \calW_2$ 
and $\calF = \calF_1$ (the class $\calW$ of weakly trivials is determined by this and described explicitly in the
Proposition), called the \textit{right localization} of $\calM_1$ with respect to $\calM_2$. Moreover, we show
(Proposition \ref{prop_bousfield}) that $\calM_1/\calM_2$ is a right Bousfield localization of $\calM_1$ with respect to
$\{0\to X\ |\ X\in\calF_2\}$ in the sense of \cite[Definition 3.3.1(2)]{Hirschhorn_LocalizationOfModelCategories}, and
that on the level of homotopy categories we get a colocalization sequence \cite[Definition 3.1]{Krause_StableDerived} of
triangulated categories $\Ho(\calM_2)\to\Ho(\calM_1)\to\Ho(\calM_1/\calM_2)$. 

Applied to the injective model $\calM\uinj(R)$ for the ordinary derived category $\bfD(R)$ and Positselski's coderived
model $\calM\uco(R)$ for the contraderived category $\bfD\uco(R)$, we get 
another abelian model structure $\calM\uco\lsing(R)=\calM\uco(R)/\calM\uinj(R)$ on $\Ch(R\Mod)$, called the
\textit{(absolute) singular coderived model}, where the cofibrant objects are the acyclic complexes of 
$R$-modules and the fibrant objects are the componentwise injective complexes of $R$-modules. In particular,
$\Ho(\calM\uco\lsing(R))\cong\bfK\lac(\Inj(R))$ and there is a colocalization sequence
$\bfD(R)\to\bfD\uco(R)\cong\bfK(\Inj(R))\to\bfK\lac(\Inj(R))$.

More generally, we construct a \textit{relative singular coderived model} $\calM\uco\lsing(A/R)$ for any morphism of dg rings
$\varphi: R\to A$ as follows: first we show that the coderived model structure $\calM\uco(R)$ on $R\Mod$ pulls back to a
model structure $\varphi\ua\calM\uco(R)$ on $A\Mod$ (Proposition \ref{prop_pullback}), and then (Definition
\ref{def_relativesingmodel}) we define $\calM\uco\lsing(A/R)$ as the right localization
$\calM\uco(A)/\varphi\ua\calM\uco(R)$. In case $R$ is an ordinary ring of finite left-global dimension, this 
will be seen to be equal to the absolute singular coderived model $\calM\uco\lsing(A)$ as defined above (Proposition
\ref{prop_finglobdim}).   

At this point we have succeeded in constructing models for singularity categories, but we cannot yet explain from the
model categorical perspective why the sequence $\bfK\lac(\Inj(A))\to\bfK(\Inj(A))\to\bfD(A)$ is not only a localization
sequence but in fact a recollement, as is known at least in the case $A$ is an ordinary Noetherian ring by 
\cite[Proposition 3.6]{Krause_StableDerived}. For this, we show that the absolute (it is important to restrict to the
absolute case) singular model structure $\calM\uco\lsing(A)$, which is a ``mixed'' model
structure in the sense that usually neither everything is fibrant nor everything is cofibrant, admits a certain (Quillen equivalent)
injective variant ${^{i}}\calM\uco\lsing(A)$. The construction of this model structure is presented in Proposition
 \ref{prop_singmodelcoctr}. The point is that while the the identity on $A\Mod$ is right Quillen
$\calM\uco(A)\to\calM\uco\lsing(A)$ and provides a right adjoint of $\bfK\lac(\Inj(A))\to\bfK(\Inj(A))$, it is
\textit{left} Quillen $\calM\uco(A)\to{^{i}}\calM\uco\lsing(A)$, providing a left adjoint of
$\bfK\lac(\Inj(A))\to\bfK(\Inj(A))$ and proving that $\bfK\lac(\Inj(A))\to\bfK(\Inj(A))\to\bfD(A)$ is a recollement
(Corollary \ref{cor_injrec}).  

Moreover, we can now right-localize $\calM\uinj(A)$ at ${^{i}}\calM\uco\lsing(A)$ to obtain another ``mixed'' model
structure ${^{m}}\calM\uinj(A)$, which turns out to be another model for $\bfD(A)$ Quillen equivalent to the injective
model $\calM\uinj(A)$, explaining the existence of the left adjoint of $\bfK(\Inj(A))\to\bfD(A)$. We see that the
recollement $\bfK\lac(\Inj(A))\to\bfK(\Inj(A))\to\bfD(A)$ unfolds to a butterfly of model structures and Quillen
functors as follows (L denotes left Quillen functors and R denotes right Quillen functors). For more details on the
properties of the butterfly, see Proposition \ref{prop_butterfly}. 
\begin{align*}\label{eq:butterfly}
\begin{tikzpicture}[description/.style={fill=white,inner sep=2pt}]
    \matrix (m) [matrix of math nodes, row sep=2em,
                 column sep=4em, text height=1.5ex, text depth=0.25ex,
                 inner sep=0pt, nodes={inner xsep=0.3333em, inner ysep=0.3333em}]
    {
       \calM\uco\lsing(A) \pgfmatrixnextcell\pgfmatrixnextcell \calM\uinj(A)\\
       \pgfmatrixnextcell \calM\uco(A) \pgfmatrixnextcell \\
       {^{i}}\calM\uco\lsing(A) \pgfmatrixnextcell\pgfmatrixnextcell {^{m}}\calM\uinj(A)\\
    };
    \draw[->] ($(m-1-1.south) + (1mm,0)$) -- node[scale=0.75,right]{L} ($(m-3-1.north) + (1mm,0)$);
    \draw[->] ($(m-3-1.north) - (+1mm,0)$) -- node[scale=0.75,left]{R} ($(m-1-1.south) - (1mm,0)$);
    \draw[->] ($(m-1-3.south) + (1mm,0)$) -- node[scale=0.75,right]{R} ($(m-3-3.north) + (1mm,0)$);
    \draw[->] ($(m-3-3.north) + (-1mm,0)$) -- node[scale=0.75,left]{L} ($(m-1-3.south) - (1mm,0)$);
    \draw[->] ($(m-1-1.south east) - (4mm,0)$) -- node[scale=0.75,below]{L} ($(m-2-2.north west) + (0mm,0)$);
    \draw[->] ($(m-2-2.north west) + (4mm,0)$) -- node[scale=0.75,above]{R} ($(m-1-1.south east) - (0mm,0)$);
    \draw[->] ($(m-3-1.north east) + (0mm,0)$) -- node[scale=0.75,below]{R} ($(m-2-2.south west) + (4mm,0)$);
    \draw[->] ($(m-2-2.south west) + (0mm,0)$) -- node[scale=0.75,above]{L} ($(m-3-1.north east) - (4mm,0)$);
    \draw[->] ($(m-2-2.north east) + (0mm,0)$) -- node[scale=0.75,below]{L} ($(m-1-3.south west) + (4mm,0)$);
    \draw[->] ($(m-1-3.south west) + (0mm,0)$) -- node[scale=0.75,above]{R} ($(m-2-2.north east) - (4mm,0)$);
    \draw[->] ($(m-2-2.south east) - (4mm,0)$) -- node[scale=0.75,below]{R} ($(m-3-3.north west) + (0mm,0)$);
    \draw[->] ($(m-3-3.north west) + (4mm,0)$) -- node[scale=0.75,above]{L} ($(m-2-2.south east) + (0mm,0)$);
\end{tikzpicture}
\end{align*}

All the constructions mentioned so far also work in the projective/contraderived setting, yielding absolute and relative
singular contraderived model structures on categories of modules over a dg ring, as well as a projective variant and a
butterfly unfolding the recollement $\bfK\lac(\Proj(A))\to\bfK(\Proj(A))\to\bfD(A)$. 

We discuss two examples. Firstly, let $R$ be a Gorenstein ring in the sense of
\cite{Buchweitz}, i.e. $R$ is Noetherian and of finite injective dimension both as a left and as a right module over
itself. Then the $0$-th cosyzygy functor $\Ch(R\Mod)\to R\Mod$ is a (left) Quillen equivalence between the
absolute singular contraderived model $\calM\uctr\lsing(R)$ on $\Ch(R\Mod)$ and Hovey's Gorenstein projective
model structure on $R\Mod$ \protect{\cite[Theorem 8.6]{Hovey_Cotorsion}}. Similarly, the $0$-th syzygy functor is a
(right) Quillen equivalence between the absolute singular coderived model $\calM\uco\lsing(R)$ and Hovey's Gorenstein
injective model on $R\Mod$. These two results are proved in Section \ref{subsection_gorenstein}.

Secondly, we consider matrix factorizations. Fix any ring $S$ with a central element $w\in Z(S)$ and let $K_{S,w} =
S[s]/(s^2)$ be the \textit{Koszul algebra} of $(S,w)$, i.e. $\deg(s)=-1$ and $\diff(s)=w$. Modules over $K_{S,w}$ can be
identified with complexes of $S$-modules $X$ equipped with a square-zero nullhomotopy $s: X\to \Sigma^{-1} X$ for
$X\stackrel{\cdot w}{\longrightarrow} X$, i.e. they can be thought of as ``curved'' mixed complexes with curvature
$w$. For any such curved mixed complex $(X,\diff,s)$ we can form the sequences $\prod
X\ueven\stackrel{\diff+s}{\longrightarrow}\prod X\uodd\stackrel{\diff+s}{\longrightarrow}\prod X\ueven$ and $\bigoplus
X\ueven\stackrel{\diff+s}{\longrightarrow}\bigoplus X\uodd\stackrel{\diff+s}{\longrightarrow}\bigoplus X\ueven$, called the
\textit{folding with products} and \textit{folding with sums} of $(X,\diff,s)$ and denoted $\fold\uprod X$ and
$\fold\uoplus X$, respectively. Since $\diff s + s\diff = w$ we see that $(\diff+s)^2=w$, and hence $\fold\uoplus(X)$ and
$\fold\uprod(X)$ are $(S,w)$-duplexes, i.e. matrix factorizations of type $(S,w)$ with possibly non-free
components. The category of $(S,w)$-duplexes is the same as the category of curved dg modules over the $\ZZ/2\ZZ$-graded
curved dg ring $S_w$ with $(S_w)^{\ol{0}} = S$, $(S_w)^{\ol{1}}=0$ and curvature $w\in Z(S)$, and in particular it
carries Positselski's contraderived model structure $\calM\uctr(S_w)$. We then prove that $\fold\uoplus$ and
$\fold\uprod$ are left resp. right Quillen equivalences $\calM\uctr\lsing(K_{S,w}/S)\to\calM\uctr(S_w)$. 
\vskip1mm

\textbf{Structure:} In Sections \ref{subsection_basicdefs} and \ref{subsection_cotorsion}
we recall the definition of abelian model categories as well as their relation to complete cotorsion pairs and deconstructible classes. In Section
\ref{subsection_fourmodels} we use this relation to give self-contained constructions of the injective, projective,
contraderived and coderived model structures on the category of modules over a dg ring. Next, in Section
\ref{subsection_localization} we prove Proposition \ref{prop_localization} providing a method for the construction of
localizations of abelian model structures. In the intermediate Section \ref{subsection_bousfield}, which is not needed
anywhere else in this article, we show that these new model structures can be described as Bousfield
localizations in the classical sense (Proposition \ref{prop_bousfield}). Then, in Section \ref{subsection_gendefs} we
turn to the construction of the relative and absolute singular contraderived and coderived model structures as well as
their projective and injective variants. In Section \ref{subsection_butterfly} we construct the
butterfly of Quillen functors lifting Krause's recollement to the level of model categories. Sections
\ref{subsection_gorenstein} and \ref{subsection_mf} contain the discussion of the examples of Gorenstein rings and
matrix factorizations. In Appendix \ref{appendix_pullback} we prove that pullbacks of deconstructible classes along
cocontinuous, monadic functors between Grothendieck categories are deconstructible (Proposition
\ref{prop_pullbackmonadic}), a fact which is used several times in Section \ref{subsection_fourmodels}.
\vskip1mm

\textbf{Acknowledgments:} This work is part of my PhD studies under supervision of Prof. Dr. Catharina Stroppel at the
University of Bonn. I thank her cordially for her support and help. I also thank Joanna Meinel, Olaf Schn\"urer and Jan Weidner for helpful discussions and numerous corrections.

\section{Abelian model categories}\label{section_abelianmodelcategories}

\subsection{Basic definitions} We begin by recalling the definition of (abelian) model structures and their homotopy
categories, focusing on the abelian case.

\begin{definition}
A \textit{model structure} $\calM$ on a category $\scC$ is a triple $(\cof,\weak,\fib)$ of classes of morphisms, called
\textit{cofibrations}, \textit{weak equivalences} and \textit{fibrations}, respectively, 
such that the following axioms are satisfied:
\begin{enumerate}
\item $\weak$ satisfies the $2$-out-of-$3$ axiom, i.e. given two composable morphisms $f,g$ in $\calM$, if
  two of $f,g,gf$ belong to $\weak$, then so does the third.
\item $\cof,\weak$ and $\fib$ are closed under retracts.
\item In any commutative square
\begin{equation*}\begin{tikzpicture}[description/.style={fill=white,inner sep=2pt}]
    \matrix (m) [matrix of math nodes, row sep=1.5em,
                 column sep=1.5em, text height=1.5ex, text depth=0.25ex,
                 inner sep=0pt, nodes={inner xsep=0.3333em, inner ysep=0.3333em}]
    {
       A & X \\
       B & Y \\
    };
    \draw[->] (m-1-1) -- (m-1-2);
    \draw[->] (m-2-1) -- (m-2-2);
    \draw[->] (m-1-1) -- node[left,scale=0.75]{$f$} (m-2-1);
    \draw[->] (m-1-2) -- node[right,scale=0.75]{$g$} (m-2-2);
    \draw[dashed,->] (m-2-1) -- (m-1-2);
\end{tikzpicture}\end{equation*}
the dashed arrow exists, making everything commutative, provided that either $f\in\cof$ and $g\in\weak\cap\fib$ or
$f\in\cof\cap\weak$ and $g\in\fib$. 
\item Any morphism $f$ factors as $f = \beta\circ\alpha$ with $\alpha\in\cof$,
  $\beta\in\weak\cap\fib$.
\item Any morphism $f$ factors as $f = \beta\circ\alpha$ with $\alpha\in\cof\cap\weak$,
  $\beta\in\fib$.
\end{enumerate}
A \textit{model category} is a bicomplete category (i.e. a category possessing arbitrary small limits and colimits)
equipped with a model structure. Given a model category, we will sometimes drop the classes $\cof,\weak,\fib$ from the
notation. 
\end{definition}

\begin{notation}
Given a model category $(\scC,\calM)$, an object $X\in\scC$ is called \textit{weakly trivial} if $0\to X\in\weak$
(equivalently, $X\to 0\in\weak$). Similarly, it is called \textit{cofibrant} if $0\to X\in\cof$, and it is called 
\textit{fibrant} if $X\to 0\in\fib$. The classes of cofibrant, weakly trivial, and fibrant objects will be denoted
$\calC$, $\calW$ and $\calF$, respectively. The \textit{homotopy category} is the localization $\scC[\weak\ui]$ and is
denoted $\Ho(\calM)$. 
\end{notation}

In this article we will mainly be concerned with model structures on abelian categories ``compatible'' with the abelian
structure in the following way:

\begin{definition}\label{def_defabelianmodelstructure}
A model structure on an abelian category is called \textit{abelian} if cofibrations equal monomorphism with cofibrant
kernel and fibrations equal epimorphisms with fibrant kernel. An \textit{abelian model 
  category} is a bicomplete abelian category equipped with an abelian model structure.
\end{definition}

\begin{rem} There are other definitions of abelian model structures which seem different at first. In
  \cite{Hovey_Cotorsion} a model structure on an abelian category is said to be compatible with the abelian structure if
  every cofibration is a monomorphism and a 
morphism is a (trivial) fibration if and only if it is an epimorphism with (trivially) fibrant kernel. In
\cite{Gillespie_ModelExact}, Gillespie requires in addition that a morphism is a (trivial) cofibration if and only if it
is a monomorphism with (trivially) cofibrant cokernel. The connection between these definitions is drawn in
\cite[Proposition 4.2]{Hovey_Cotorsion}: Assuming that every cofibration is a monomorphism and every fibration is an
epimorphism, the four possible conditions on the characterization (trivial) (co)fibration in terms of their (co)kernels
come in two pairs: Assuming that cofibrations equal monomorphisms with cofibrant cokernel is equivalent to assuming that
trivial fibrations are epimorphisms with trivially fibrant kernel, and assuming that trivial cofibrations equal
monomorphisms with trivially cofibrant cokernel is equivalent to assuming that fibrations are epimorphisms with fibrant
kernel. In particular, our Definition \ref{def_defabelianmodelstructure} is equivalent to \cite{Hovey_Cotorsion} is
equivalent to \cite{Gillespie_ModelExact}. \end{rem}

Requiring that any cofibration (resp. fibration) should be a monomorphism (resp. epimorphism) is not as automatic as
it might appear at first: for example, given a ring $R$ the standard projective model structure on $\Ch_{\geq 0}(R\Mod)$
\cite{Quillen_HomotopicalAlgebra} is \textit{not} abelian since fibrations are required to be epimorphisms only in
positive degrees. As a positive example, the standard injective and projective model structures on the category
$\Ch(R\Mod)$ of \textit{unbounded} chain complexes of $R$-modules are abelian:

\begin{prop}[\protect{\cite{Hovey_ModelCategories}}]\label{prop_chinjproj}
Let $R$ be a ring.
\begin{enumerate}
\item There exists a cofibrantly generated abelian model structure on $\Ch(R\Mod)$ with $\calC =
  \Ch(R\Mod)$, $\calW = \Acyc(R\Mod)$ and $\calF = \dgInj(R)$, called the \textit{standard injective model structure} on $\Ch(R\Mod)$.
\item There exists a cofibrantly generated abelian model structure on $\Ch(R\Mod)$ with $\calF =
  \Ch(R\Mod)$, $\calW = \Acyc(R\Mod)$ and $\calC = \dgProj(R)$, called the \textit{standard projective model structure} on $\Ch(R\Mod)$.
\end{enumerate}
The standard projective and injective model structures on $\Ch(R\Mod)$ are denoted $\calM\uproj(R)$ and $\calM\uinj(R)$,
respectively. 
\end{prop}
\begin{proof}
The existence and cofibrant generation of injective and projective model structures on $\Ch(R\Mod)$ is proved in
\cite[Theorems 2.3.11 and 2.3.13]{Hovey_ModelCategories}, and \cite[Propositions 2.3.9 and
2.3.20]{Hovey_ModelCategories} show that they are abelian. 
\end{proof}

Another example of an abelian model structure is Hovey's model for the singularity category of a Gorenstein ring. Recall
that a ring $R$ is \textit{Gorenstein} \cite{Buchweitz} if $R$ is Noetherian and of finite injective dimension both as a left and
as a right module over itself. An $R$-module is called \textit{Gorenstein projective} if it arises as the 
$0$-th syzygy of an acyclic complex of projective $R$-modules, and it is called \textit{Gorenstein injective} if it arises as the
$0$-th syzygy of an acyclic complex of injective $R$-modules. The classes of Gorenstein projective and Gorenstein
injective $R$-modules are denoted $\gproj(R)$ and $\ginj(R)$, respectively.
\begin{prop}[\protect{\cite[Theorem 8.6]{Hovey_Cotorsion}}]\label{prop_gorensteinmodels}
Let $R$ be a Gorenstein ring.
\begin{enumerate}
\item There exists an abelian model structure on $R\Mod$, called the \textit{Gorenstein projective} model
  structure and denoted $\calM\ugproj(R)$, with $\calC = \gproj(R)$, $\calW = \proj^{<\infty}(R)$ (the modules of finite
  projective dimension) and $\calF=R\Mod$.  
\item There exists an abelian model structure on $R\Mod$, called the \textit{Gorenstein injective} model
  structure and denoted $\calM\uginj(R)$, with $\calC=R\Mod$, $\calW = \proj^{<\infty}(R)$ and $\calF = \ginj(R)$.
\end{enumerate}
Moreover, both $\calM\ugproj(R)$ and $\calM\uginj(R)$ are cofibrantly generated.
\end{prop}

Right from the definition we know that an abelian model structure is determined by the triple of cofibrant,
weakly trivial and fibrant objects. The question which such triples actually give rise to abelian model structures was
solved in \cite{Hovey_Cotorsion} in terms of complete cotorsion pairs:

\begin{definition}[\protect{\cite[Definition 2.3]{Hovey_Cotorsion}}]\label{def_cotorsionpair}
For an abelian category $\scA$, a \textit{cotorsion pair} in $\scA$ is a pair $(\calD,\calE)$ of classes of objects such
that the following hold: 
\begin{enumerate}
\item $\calD = {^\perp\calE} := \{X\in\scA\ |\ \ext^1_\scA(X,\calE)=0\}$.
\item $\calE = \calD^\perp := \{Y\in\scA\ |\ \ext^1_\scA(\calD,Y)=0\}$.
\end{enumerate}
In this case, we call $\calD$ the \textit{cotorsion class} and $\calE$ the \textit{cotorsionfree class}. A cotorsion
pair $(\calD,\calE)$ is called \textit{complete} if the following two conditions are satisfied: 
\begin{enumerate}
\item[(3)] $(\calD,\calE)$ has \textit{enough projectives}, i.e. for each $Z\in\scA$ there exists an exact sequence $0\to Y\to
  X\to Z\to 0$ such that $X\in\calD$ and $Y\in\calE$.
\item[(4)] $(\calD,\calE)$ has \textit{enough injectives}, i.e. for each $Z\in\scA$ there exists an exact sequence
$0\to Z\to Y\to X\to 0$ such that $Y\in\calE$ and $X\in\calD$.
\end{enumerate}
A cotorsion pair $(\calD,\calE)$ is called \textit{resolving} if $\calD$ is closed under taking kernels of
epimorphisms, and it is called \textit{coresolving} if $\calE$ is closed under taking cokernels of monomorphisms. It is
called \textit{hereditary} if it is both resolving and coresolving.
\end{definition}

For example, denoting $\calI$ the class of injectives, the pair $(\scA,\inj)$ is a hereditary cotorsion pair with enough
projectives. It is complete if and only if $\scA$ has enough injectives in the usual sense. Similarly, denoting $\proj$
the class of projectives, the pair $(\proj,\scA)$ is a hereditary cotorsion pair with enough injectives, and it  is
complete if and only if $\scA$ has enough projectives.  

\begin{definition}\label{def_thick}
A subcategory $\calW$ of an abelian category $\scA$ is called \textit{thick} if it is closed under summands and
if it satisfies the \textit{$2$-out-of-$3$ property}, i.e. whenever two out of three terms in a short exact sequence lie
in $\calW$, then so does the third. 
\end{definition}

\begin{theorem}[\protect{\cite[Theorem 2.2]{Hovey_Cotorsion}}]\label{thm_hovey}
Let $\scA$ be a bicomplete abelian category and $\calC$, $\calW$ and $\calF$ classes of objects in $\scA$. Then the
following are equivalent: 
\begin{enumerate}
\item[(i)] There exists an abelian model structure on $\scA$ where $\calC$ is the class of cofibrant, $\calF$ is the class
  of fibrant, and $\calW$ is the class of weakly trivial objects.
\item[(ii)] $\calW$ is thick and both $(\calC,\calF\cap\calW)$ and $(\calC\cap\calW,\calF)$ are complete cotorsion pairs.
\end{enumerate}
Slightly abusing the notation, given a triple $(\calC,\calW,\calF)$ as above we will often denote
its induced abelian model structure $(\calC,\calW,\calF)$ as well.
\end{theorem}

We call an abelian model structure $\calM=(\calC,\calW,\calF)$ \textit{hereditary} if their associated cotorsion pairs
$(\calC,\calW\cap\calF)$ and $(\calC\cap\calW,\calF)$ are hereditary. In view of the $2$-out-of-$3$ property of $\calW$,
this is equivalent to saying that $\calC$ is closed under taking kernels of epimorphisms and $\calF$ is closed under
taking cokernels of monomorphisms. Note that Gillespie \cite{Gillespie_ModelExact} even obtained a version of Theorem
\ref{thm_hovey} for exact categories endowed with model structures compatible with the exact structure. Moreover, he
does \textit{not} assume the existence of arbitrary small colimits and limits, as is done here and in
\cite{Hovey_ModelCategories}, for example.  

Let us consider the extreme cases of \textit{projective} (resp. \textit{injective}) abelian model structures, i.e. model structures where
everything is fibrant (resp. cofibrant). 

\begin{cor}\label{cor_injprojmodel}
Let $\scA$ be a bicomplete abelian category and $\calC,\calW\subset\scA$ classes of objects in $\scA$. Then the
following are equivalent:
\begin{enumerate}
\item[(i)] $(\calC,\calW,\scA)$ gives rise to an abelian model structure on $\scA$.
\item[(ii)] $\scA$ has enough projectives, $(\calC,\calW)$ is a complete cotorsion pair with $\calC\cap\calW = \proj(\scA)$
  and $\calW$ satisfies the $2$-out-of-$3$ property. 
\end{enumerate}
Dually, for classes of objects $\calW,\calF\subseteq\scA$ the following are equivalent:
\begin{enumerate}
\item[(i)] $(\scA,\calW,\calF)$ gives rise to an abelian model structure on $\scA$.
\item[(ii)] $\scA$ has enough injectives, $(\calW,\calF)$ is a complete cotorsion pair with $\calW\cap\calF =
  \inj(\scA)$ and $\calW$ satisfies the $2$-out-of-$3$ property.
\end{enumerate}
\end{cor} 
\begin{proof}
By Theorem \ref{thm_hovey}, $(\calC,\calW,\scA)$ giving rise to an abelian model structure on $\scA$ is equivalent to
$\calW$ satisfying the $2$-out-of-$3$ property and $(\calC,\calW\cap\calF)=(\calC,\calW)$,
$(\calC\cap\calW,\calF)=(\calC\cap\calW,\scA)$ being complete cotorsion pairs. The latter means that $\scA$ has enough
projectives and $\calC\cap\calW=\proj(\scA)$. The second part is dual.
\end{proof}

We will see how complete cotorsion pairs can be constructed in the next
section. Concerning the $2$-out-of-$3$ property, the next lemma will be useful.

\begin{lem}\label{lem_twooutofthree}
Let $(\calW,\calF)$ be a cotorsion pair in an abelian category $\scA$ with enough injectives. Consider the following statements:
\begin{enumerate}
\item $(\calW,\calF)$ is coresolving.
\item $\ext^k_\scA(W,F)=0$ for all $W\in\calW$, $F\in\calF$ and $k\geq 1$.
\item $\calW$ satisfies the $2$-out-of-$3$ property.
\end{enumerate}
Then (1)$\Leftrightarrow$(2). If $(\calW,\calF)$ is complete with $\calW\cap\calF=\inj(\scA)$, then also (2)$\Rightarrow$(3).
\end{lem}
\begin{proof}
(2)$\Rightarrow$(1) follows from the long exact $\ext$-sequence. Now assume (1) holds. For $F\in\calF$, pick an embedding $i:
F\hookrightarrow I$ with $I\in\inj(\scA)\subset\calF$. Then $\Sigma F := \coker(i)\in\calF$ by assumption, and 
$\ext^k_\scA(-,F)\cong\ext^{k-1}_\scA(-,\Sigma F)$ for all $k\geq 2$. Inductively, we deduce (2). This shows
(1)$\Leftrightarrow$(2), so it remains to show (2)$\Rightarrow$(3) in case $(\calW,\calF)$ is complete and
$\calW\cap\calF=\inj(\scA)$. If $0\to W_1\to W_2\to W_3\to 0$ is a short exact sequence with at least two of the $W_i$ belonging to $\calW$, we have
$\ext^2_\scA(W_i,\calF)=0$ for all $i=1,2,3$. It is therefore sufficient to show that any $X\in\scA$
satisfying $\ext_\scA^2(X,\calF)=0$ actually satisfies $\ext^1_\scA(X,\calF)=0$, i.e. $X\in\calW$. For this, pick
$F\in\calF$ arbitrary and choose an exact sequence $0\to F\p\to I\to F\to 0$ with $F\p\in\calF$ and
$I\in\inj(\scA)$. Such a sequence exists since $(\calW,\calF)$ has enough projectives, $\calF$ is closed under
extensions and $\calW\cap\calF=\inj(\scA)$ by assumption. Then $\ext^1_\scA(X,F)\cong\ext^2_\scA(X,F\p)=0$, and hence $X\in\calW$.  
\end{proof}

Combining Lemma \ref{lem_twooutofthree} with its dual (note that $(2)\Rightarrow (1)$ did only use the existence of
$\ext\ua$ and the long exact $\ext\ua$-sequence) shows that in case $\scA$ has enough injectives, then $(\calW,\calF)$ being coresolving implies
$(\calW,\calF)$ being resolving. Dually, if $\scA$ has enough projectives, then $(\calW,\calF)$ being resolving implies
$(\calW,\calF)$ being coresolving. Restricting to complete cotorsion pairs, the existence of enough projectives or
injectives is not necessary: 

\begin{prop}\label{prop_cotorsiontorsion}
Let $\scA$ be an abelian category, $(\calX,\calY)$ be a complete, coresolving cotorsion pair and $\omega :=
\calX\cap\calY$. Then $\calX/\omega={^{\ddagger}}(\calY/\omega)$, $\calY/\omega = (\calX/\omega)^{\ddagger}$ in
$\scA/\omega$. Here $\scA/\omega$, $\calX/\omega$ and $\calY/\omega$ denote the stable categories and $\ddagger$ denotes
the $\Hom$-orthogonal (because $\perp$ is already occupied). Moreover, $(\calX,\calY)$ is resolving. 
\end{prop}
\begin{proof}
Given $Y\in\calY$, in a sequence $0\to Y\p\to X\to Y\to 0$ with $Y\p\in\calY$ and $X\in\calX$ we have
$X\in\calX\cap\calY=\omega$ since $\calY$ is extension-closed. As $X\to Y$ is an $\calX$-approximation, it follows that
any map $X\p\to Y$ for some other $X\p\in\calX$ factors through $\omega$, hence vanishes in $\scA/\omega$. 

Next, let $A\in\scA$ and pick exact sequences $0\to Y\to X\to A\to 0$ and $0\to X\to I\to X\p\to 0$ with
$X,X\p\in\calX$, $I\in\omega$ and $Y\in\calY$. Taking pushout yields a commutative diagram with exact rows and columns,
and a bicartesian upper right square:  
\begin{equation*}\begin{tikzpicture}[description/.style={fill=white,inner sep=2pt}]
    \matrix (m) [matrix of math nodes, row sep=1.5em,
                 column sep=1.5em, text height=1.5ex, text depth=0.25ex,
                 inner sep=0pt, nodes={inner xsep=0.3333em, inner ysep=0.3333em}]
    {
      && 0 & 0 \\
       0 & Y & X & A & 0 \\
       0 & Y & I & Y\p & 0 \\
       && X\p & X\p \\
       && 0 & 0 \\
    };
    \draw[->] (m-1-3) -- (m-2-3);
    \draw[->] (m-1-4) -- (m-2-4);
    \draw[->] (m-2-1) -- (m-2-2);
    \draw[->] (m-2-2) -- (m-2-3);
    \draw[->>] (m-2-3) -- (m-2-4);
    \draw[->] (m-2-4) -- (m-2-5);
    \draw[->] (m-3-1) -- (m-3-2);
    \draw[->] (m-3-2) -- (m-3-3);
    \draw[->] (m-3-3) -- (m-3-4);
    \draw[->] (m-3-4) -- (m-3-5);
    \draw[double,double distance=0.5mm] (m-2-2) -- (m-3-2);
    \draw[double,double distance=0.5mm] (m-4-3) -- (m-4-4);
    \draw[>->] (m-2-3) -- (m-3-3);
    \draw[->] (m-2-4) -- (m-3-4);
    \draw[->] (m-3-3) -- (m-4-3);
    \draw[->] (m-4-3) -- (m-5-3);
    \draw[->] (m-3-4) -- (m-4-4);
    \draw[->] (m-4-4) -- (m-5-4);
\end{tikzpicture}\end{equation*}
Moreover, since $\calY$ is closed under taking cokernels of monomorphisms by assumption, we also have
$Y\p\in\calY$. Now, in case $A\in {^{\ddagger}}(\calY/\omega)$ the map $A\to Y\p$ factors through an object in $\omega$, 
hence through $I\to Y\p$ as $Y = \ker(I\to Y\p)\in\calY\subset\omega^{\perp}$. Since the upper right square is
cartesian, any such factorization $A\to I$ gives rise to a splitting of $X\to A$, and hence
$A\in\calX$. Similarly, if $A\in(\calX/\omega)^{\ddagger}$, the map $X\to A$ factors through an object in $\omega$, hence
through $X\to I$, and since the upper right square is cocartesian, such a factorization yields a splitting of $A\to Y$,
so $A\in\calY$. 

For the last part, suppose $0\to Z\to X\to X\p\to 0$ is an exact sequence with $X,X\p\in\calX$. We want to show that
$Z\in\calX$, and by the above it is sufficient to show that any morphism $f: Z\to Y$ factors through $\omega$. But $f$
extends to a morphism $g: X\to Y$ (since $X\p\in\calX)$ which then factors through $\omega$ (since $X\in\calX$). 
\end{proof}
\begin{cor}\label{cor_coresres}
A complete cotorsion pair is coresolving if and only if it is resolving. In particular, any injective/projective
abelian model structure is hereditary.
\end{cor}
\begin{proof}
The first statement follows from Proposition \ref{prop_cotorsiontorsion} combined with its dual. For the second, note
that if $(\scA,\calW,\calF)$ is an injective abelian model structure, then $(\calW,\calF)$ is a resolving cotorsion pair
(since $\calW$ satisfies the $2$-out-of-$3$ property), hence hereditary by the first part. The projective case is similar.
\end{proof}

We now describe the homotopy category of an abelian model category.

\begin{prop}\label{prop_descriptionhomotopycategory}
Let $\scA$ be a bicomplete abelian category and $\calM = (\calC,\calW,\calF)$ be an abelian model structure on
$\scA$. Then the composition $\calC\cap\calF\hookrightarrow\scA\to\Ho(\calM)$ induces an equivalence
of categories $\calC\cap\calF/\omega\cong\Ho(\calM)$, where $\omega=\calC\cap\calW\cap\calF$. 
\end{prop}
\begin{proof}
This is known -- see for example \cite[Proposition 4.3,4.7]{Gillespie_ModelExact} or \cite[Theorem
VIII.4.2]{BeligiannisReiten} -- but for completeness we give a proof here. For a  
general model category $\calM$ and objects $X,Y$, the set $\calM(X,Y)$ admits two natural relations $\sim_{l/r}$ of left
and right homotopy, defined via cylinder and path objects, respectively. If $X$ is cofibrant and $Y$ is fibrant, these
two relations coincide and are equivalence relations, and $\calM(X,Y)\to\Ho(\calM)(X,Y)$ induces a bijection
$\calM(X,Y)/_{\sim}\cong\Ho(\calM)(X,Y)$. In particular, there is a fully-faithful functor
$\calM\lcf/_{\sim}\to\Ho(\calM)$, where $\calM\lcf$ is the class of simultaneously cofibrant and fibrant objects of
$\calM$, and by the existence of fibrant and cofibrant resolutions this is even an equivalence of categories. See
\cite[Theorem 1.2.10]{Hovey_ModelCategories} for details.

To prove the claim, it is therefore sufficient to show that for $X\in\calC$ and $Y\in\calF$, two morphisms $f,g: X\to Y$
are right homotopic in the above sense if and only if $f-g$ factors through $\calC\cap\calW\cap\calF$. For this, we
construct a path object $PY$ for $Y$ as follows: first choose a short exact sequence $0\to\Omega Y\to I\to Y\to 0$ with
$I\in\calC\cap\calW$ and $\Omega Y\in\calF$. Such a sequence exists by the completeness of the cotorsion pair
$(\calC\cap\calW,\calF)$. Since $\calF$ is closed under extensions, we even have
$I\in\calC\cap\calW\cap\calF=\omega$. Taking the pullback of $Y\oplus Y\xrightarrow{(1,-1)}Y\leftarrow
I$, we get the following commutative diagram with exact rows and columns:

\begin{align}\tag{$\ast$}\label{diag:pathobject}\begin{tikzpicture}[baseline,description/.style={fill=white,inner sep=2pt}]
    \matrix (m) [matrix of math nodes, row sep=1.5em,
                 column sep=2.25em, text height=1.5ex, text depth=0.25ex,
                 inner sep=0pt, nodes={inner xsep=0.3333em, inner ysep=0.3333em}]
    {
      \pgfmatrixnextcell\pgfmatrixnextcell 0 \pgfmatrixnextcell 0 \\
       0 \pgfmatrixnextcell Y \pgfmatrixnextcell Y\oplus Y \pgfmatrixnextcell Y \pgfmatrixnextcell 0 \\
       0 \pgfmatrixnextcell Y \pgfmatrixnextcell PY \pgfmatrixnextcell I \pgfmatrixnextcell 0 \\
       \pgfmatrixnextcell\pgfmatrixnextcell \Omega Y \pgfmatrixnextcell \Omega Y \\
       \pgfmatrixnextcell\pgfmatrixnextcell 0 \pgfmatrixnextcell 0 \\
    };
    \draw[->] (m-2-3) -- (m-1-3);
    \draw[->] (m-2-4) -- (m-1-4);
    \draw[->] (m-2-2) -- node[above,scale=0.75]{$\Delta$} (m-2-3);
    \draw[->] (m-2-3) -- node[above,scale=0.75]{$(1\ -1)$} (m-2-4);
    \draw[->] (m-2-1) -- (m-2-2);
    \draw[->] (m-2-4) -- (m-2-5);
    \draw[->] (m-3-1) -- (m-3-2);
    \draw[->] (m-3-2) -- (m-3-3);
    \draw[->] (m-3-3) -- (m-3-4);
    \draw[->] (m-3-4) -- (m-3-5);
    \draw[double,double distance=0.5mm] (m-2-2) -- (m-3-2);
    \draw[double,double distance=0.5mm] (m-4-3) -- (m-4-4);
    \draw[->] (m-3-3) -- (m-2-3);
    \draw[->] (m-3-4) -- (m-2-4);
    \draw[->] (m-4-3) -- (m-3-3);
    \draw[->] (m-5-3) -- (m-4-3);
    \draw[->] (m-4-4) -- (m-3-4);
    \draw[->] (m-5-4) -- (m-4-4);
\end{tikzpicture}\end{align}
The morphism $PY\to Y\oplus Y$ is a fibration because its kernel $\Omega Y$ lies in $\calF$, and $Y\to PY$
is a trivial cofibration because its cokernel $I$ belongs to $\omega\subset\calC\cap\calW$. In other words, the
factorization $Y\to PY\to Y\oplus Y$ of $\Delta: Y\to Y\oplus Y$ is a path object for $Y$ and can be used to compute the
right homotopy relation. By definition of the pullback, the morphism $(f,g)^t: X\to Y\oplus Y$ factors through $PY\to Y$
if and only if $f-g: X\to Y$ factors through $I\to Y$. Finally, since $I\to Y$ is a $\omega$-cover for $Y$ (its kernel
$\Omega Y$ is in $\calF=(\calC\cap\calW)^{\perp}\subset\omega^{\perp}$), this is in turn equivalent to $f-g: X\to Y$
factoring through \textit{some} object in $\omega$.
\end{proof}

The homotopy category of a model category $(\scA,\calM)$ whose underlying category $\scA$ is abelian carries a natural
pretriangulated structure in the sense of \cite[Definition II.1.1]{BeligiannisReiten}. This follows from \cite[Section 6.5]{Hovey_ModelCategories} together with the fact that any cogroup object in an additive category is isomorphic to one
of the form $(X,\Delta: X\to X\oplus X, 0: X\to 0)$ and that giving some object $Y$ a comodule structure over such a
cogroup is equivalent to giving a morphism $Y\to X$. See also \cite[Remark 7.1.3, Theorem
7.1.6]{Hovey_ModelCategories}. Concretely \cite[Paragraph following Definition 6.1.1]{Hovey_ModelCategories}, the
suspension functor $\Sigma:\Ho(\calM)\to\Ho(\calM)$ takes a cofibrant object $X$ to the cokernel of the inclusion
$X\oplus X\to\text{Cyl}(X)$, where $X\oplus X\to \text{Cyl}(X)\to X$ is a cylinder object for $X$, and the loop functor
$\Omega:\Ho(\calM)\to\Ho(\calM)$ takes a fibrant object $Y$ to the kernel of the projection $PY\to Y\oplus Y$, where
$Y\to PY\to Y\oplus Y$ is a path object for $Y$. If $\calM=(\calC,\calW,\calF)$ is an abelian model structure, in view of the
explicit construction \eqref{diag:pathobject} of path objects in Proposition \ref{prop_descriptionhomotopycategory} and the corresponding dual
construction of cylinder objects, we conclude that given objects $X\in\calC$ and $Y\in\calF$ their suspension and loop
objects $\Sigma X\in\calC$, $\Omega Y\in\calF$ can be defined by the property that they belong to exact sequences $0\to X\to I\to
\Sigma X\to 0$ and $0\to \Omega Y\to P\to Y\to 0$ with $I\in\calW\cap\calF$ and $P\in\calC\cap\calW$. However, for
$X,Y\in\calC\cap\calF$ it is not clear in this situation that $\Sigma X$ and $\Sigma Y$ again belong to
$\calC\cap\calW$, at least if $\calM$ is not assumed to be hereditary. Hence, in this case
we don't know how the pretriangulated structure on $\calC\cap\calF/\omega$ obtained by pulling back the pretriangulated
structure on $\Ho(\calM)$ along the equivalence $\calC\cap\calF/\omega\to\Ho(\calM)$ of Proposition
\ref{prop_descriptionhomotopycategory} can be described explicitly. Assuming that $\calM$ is hereditary, however, we
have the following \cite[Proposition 5.2]{Gillespie_ModelExact}: 

\begin{prop}\label{prop_hereditary}
Let $\calM=(\calC,\calW,\calF)$ be a hereditary abelian model structure on an abelian category $\scA$
Then $\calC\cap\calF$, endowed with the exact structure inherited from $\scA$, is Frobenius. Its class of
projective-injective objects equals $\omega := \calC\cap\calW\cap\calF$, and 
$\calC\cap\calF/\omega\to\Ho(\calM)$ is an equivalence of pretriangulated categories.
\end{prop}
\begin{cor}\label{cor_hereditarystable}
A hereditary abelian model category is stable.
\end{cor}
\begin{proof}[Proof of Proposition \ref{prop_hereditary}]
Denote $\scE$ the class of short exact sequences in $\scA$ with entries in $\calC\cap\calF$. We only check that
 $(\calC\cap\calF,\scE)$ is a Frobenius category; the remaining part involves comparing the definition of
 distinguished triangles in stable categories of Frobenius categories to the definition of fiber and cofiber sequences
 in the homotopy category of a pointed model category \cite[Definition 6.2.6]{Hovey_ModelCategories}, but we omit it. 

First, we have $\calC\cap\calF\subset\calC = {^{\perp}}(\calW\cap\calF)\subset {^{\perp}}\omega$ and similarly 
$\calC\cap\calF\subset\omega^\perp$, showing that any object in $\omega$ is projective-injective in
$(\calC\cap\calF,\scE)$. Next, given $X\in\calC\cap\calF$, the completeness of $(\calC\cap\calW,\calF)$ provides a short exact
sequence $0\to X\p\to I\to X\to 0$ in $\scA$ with $X\p\in\calF$ and $I\in\calC\cap\calW$. As $\calC$ is closed under taking kernels of
epimorphisms by assumption and $\calF$ is closed under taking extensions, we infer that $X\p\in\calC\cap\calF$ and
$I\in\omega$, proving that  $(\calC\cap\calF,\scE)$ has enough projectives, and that $\proj(\calC\cap\calF,\scE) =
\omega$. Similarly, using that $\calF$ is closed under taking cokernels of monomorphisms we get that $(\calC\cap\calF,\scE)$
has enough injectives and $\inj(\calC\cap\calF,\scE)=\omega$, finishing the proof. 
\end{proof}

\label{subsection_basicdefs}
\subsection{Small cotorsion pairs}\label{subsection_cotorsion}

In the previous section we recalled the definition and properties of abelian model structures, and in particular we
discussed Hovey's one-to-one correspondence between abelian model structures and pairs of compatible complete
cotorsion pairs. However, we did not explain so far how one can actually construct such complete cotorsion pairs,
and this is the topic of the present section. We describe how each set $\calS$ of objects in an abelian category $\scA$
yields a cotorsion pair in $\scA$, called the cotorsion pair cogenerated by $\calS$, and discuss when such cotorsion
pairs are complete, our main source being \cite{SaorinStovicek}. We then use these results to give a handy
description of classes occurring as cotorsion classes in complete cotorsion pairs cogenerated by sets in terms of
generators and deconstructibility. This prepares the ground for the construction of the projective, injective, coderived and
contraderived abelian model structures for modules over (curved) differential graded rings in the next section. We
end with a theorem of Hovey connecting complete cotorsion pairs cogenerated by sets to cofibrantly generated abelian
model categories. 

Let $\scA$ be an abelian category with small coproducts. We say that a class of objects $\calG\subseteq\scA$ is
\textit{generating} or that it \textit{generates} $\scA$ if any object in $\scA$ is the quotient of a set-indexed
coproduct of objects in $\calG$. An object $G\in\scA$ is called a \textit{generator} if $\{G\}$ is generating, i.e. if
any object in $\scA$ is a quotient of $G^{\coprod I}$ for some large enough set $I$ (for a comparison to other
definitions of generators and generating sets, see \cite[Proposition 5.2.4]{KashiwaraShapira}). We call $\scA$ an
\textit{(AB5)-category} if small colimits exist in $\scA$ and if filtered colimits are exact, and we say 
that $\scA$ is a \textit{Grothendieck category} if, in addition to being (AB5), it admits a generating set of objects
(or equivalently, a generator). Note that in a Grothendieck category a class of objects is generating if and only if it
contains a generating set. We refer to  \cite{KashiwaraShapira} for generalities on Grothendieck 
categories. For example, any Grothendieck category possesses arbitrary small limits \cite[Proposition
8.3.27(i)]{KashiwaraShapira} and has enough injectives \cite[Theorem 9.6.2]{KashiwaraShapira}. 

From now on let $\scA$ be a Grothendieck category. A cotorsion pair $(\calD,\calE)$ in $\scA$ is said to be
\textit{cogenerated by a set} if there exists a set $\calS\subset\calD$ such that $\calE = \calS^{\perp}$. Any set of
objects $\calS$ serves as the cogenerating set for a unique cotorsion pair, namely
$({^{\perp}}(\calS^{\perp}),\calS^{\perp})$. Although trivial, this is a useful method for constructing cotorsion
pairs. In order to get abelian model structures, however, a criterion is needed to check when cotorsion pairs
cogenerated by certain sets of objects are complete, which is provided by the following proposition:

\begin{prop}[\protect{\cite{SaorinStovicek}}]\label{prop_autocomplete}
Let $\scA$ be a Grothendieck category and $(\calD,\calE)$ be a cotorsion pair cogenerated by a set. Then the following hold:
\begin{enumerate}
\item $(\calD,\calE)$ has enough injectives.
\item $(\calD,\calE)$ has enough projectives if and only if $\calD$ is generating. 
\end{enumerate}
\end{prop}
\begin{proof}
Part (1) and the implication ``$\Leftarrow$'' in (2) follow from Quillen's small object argument and are explained very clearly
in \cite[Theorem 2.13]{SaorinStovicek} in the bigger generality of efficient exact categories (of which Grothendieck
categories are examples by \cite[Proposition 2.7]{SaorinStovicek}). It remains to check the implication ``$\Rightarrow$'' in (2):  Assuming
$(\calD,\calE)$ is complete, let $G\in\scA$ be a generator of $\scA$ and pick a short exact sequence $0\to E\to D\to
G\to 0$ with $E\in\calE$ and $D\in\calD$. Then $D$ is a generator for $\scA$, too, so $\calD$ is
generating.
\end{proof}

A cotorsion pair $(\calD,\calE)$ is called \textit{small} if it is cogenerated by a set and if $\calD$ is
generating. The notion of small cotorsion pairs was introduced in \cite[Definition 6.4]{Hovey_Cotorsion} in the study of completeness of
cotorsion pairs cogenerated by sets. The definition given here differs from Hovey's in that we do not assume condition
(iii) of loc.cit. However, in our situation that condition 
(iii) is automatic by \cite[Proposition 2.7]{SaorinStovicek}. In case our underlying category $\scA$ has enough
projectives (as for example in the cases of modules over dg rings we will be studying later) any cotorsion pair
cogenerated by a set is automatically small:  

\begin{cor}\label{cor_complete}
Let $\scA$ be a Grothendieck category with enough projectives. Then any cotorsion pair cogenerated by
a set is small, and in particular complete.
\end{cor}
\begin{proof}
Since $\scA$ has enough projectives it admits a projective generator. In particular, the class of
projectives is generating, and hence so is any cotorsion class. The second part follows from Proposition \ref{prop_autocomplete}.
\end{proof}

Proposition \ref{prop_autocomplete} and Corollary \ref{cor_complete} allow for proving that a certain class $\calE$
arises as the cotorsionfree part of a complete cotorsion pair. To give criteria when a class $\calD$ arises as the
cotorsion part in a complete cotorsion pair, we need a more concrete description of ${^{\perp}}(\calS^{\perp})$ for a
cogenerating set $\calS\subseteq\scA$. For this, we recall the notion of an $\calS$-filtration.

\begin{definition}[\protect{\cite[Definition 1.3]{StovicekHillGrothendieck}}]\label{def_dec}
Let $\scA$ be a Grothendieck category, $\calS$ a class of objects in $\scA$ and $X\in\scA$. An
\textit{$\calS$-filtration} on $X$ consists of an ordinal $\tau$ together with a family
$\{X_{\sigma}\}_{\sigma\leq\tau}$ of subobjects of $X$ such that the following hold:
\begin{enumerate}
\item $X_0 = 0$, $X_\tau = X$ and $X_\mu\subseteq X_\sigma$ if $\mu\leq\sigma\leq\tau$.
\item If $\sigma\leq\tau$ is a limit ordinal, $X_{\sigma} = \sum_{\mu<\sigma}
  X_\mu$.
\item $X_{\sigma+1}/X_{\sigma}$ is isomorphic to an object in $\calS$ for all
  $\sigma<\tau$.
\end{enumerate}
The \textit{size} of such an $\calS$-filtration is $|\tau|$. The class of objects admitting an $\calS$-filtration is denoted
$\filt\calS$, and its closure under taking summands is denoted ${^{\oplus}}\filt\calS$. A class $\calF\subset\scA$ of
the form $\calF=\filt\calS$ for some set $\calS\subset\scA$ is called \textit{deconstructible}.
\end{definition}

\begin{prop}\label{prop_doubleorthogonal}
Let $\scA$ be a Grothendieck category and $\calS\subseteq\scA$ be a set of objects. Assume that $\filt\calS$ is a
generating class for $\scA$. Then ${^{\perp}}(\calS^{\perp})={^{\oplus}}\filt\calS$.
\end{prop}
\begin{proof}
This is also part of \cite[Theorem 2.13]{SaorinStovicek}.
\end{proof}

\begin{prop}\label{prop_cogset}
Let $\scA$ be a Grothendieck category and let $\calD\subseteq\scA$ be some class of objects. Then the following are
equivalent:  
\begin{enumerate}
\item[(i)] $\calD$ arises as the cotorsion part in a small cotorsion pair.
\item[(ii)] $\calD$ is generating and $\calD = {^{\oplus}}\filt\calS$ for a set of objects $\calS$.
\item[(iii)] $\calD$ is generating, closed under direct summands, and deconstructible.
\end{enumerate}
\end{prop}
\begin{proof}
$(1)\Rightarrow (2)$ Suppose $(\calD,\calE)$ a small cotorsion pair cogenerated by some set
$\calS\subseteq\calD$, i.e. $\calE=\calS^{\perp}$. By definition, $\calD$ is generating and hence we may without loss of
generality assume that $\calS$ is generating, too (otherwise enlarge $\calS$ by a set of generators of $\scA$ inside
$\calD$). We then get $\calD = {^{\perp}}\calE = {^{\perp}}(\calS^{\perp})={^{\oplus}}\filt\calS$ by Proposition
\ref{prop_doubleorthogonal}. $(2)\Rightarrow (1)$: If 
$\calD={^{\oplus}}\filt\calS$ and $\calD$ is generating, then so is $\filt\calS$. Hence Propositions \ref{prop_doubleorthogonal}
and \ref{prop_autocomplete} yield the small cotorsion pair $({^{\perp}}(\calS^{\perp}),\calS^{\perp})=
({^{\oplus}}\filt\calS,\calS^{\perp})=(\calD,\calS^{\perp})$. This shows $(1)\Leftrightarrow (2)$. $(3)\Rightarrow (2)$
is clear and finally $(2)\Rightarrow (3)$ follows from
\cite[Proposition 2.9(1)]{StovicekHillGrothendieck} which says that given any deconstructible class in a Grothendieck
category, the class of direct summands of objects of this class is again deconstructible.
\end{proof}

\begin{ex}\label{ex_injproj} Let $\scA$ be a Grothendieck category.
\begin{enumerate}
\item\label{ex_injproj_inj} Suppose $G$ is generator of $\scA$ and let $\calS$ be a representative set of isomorphism classes of quotients of
$G$. Then $\scA = \filt\calS$, so $\scA$ is deconstructible. As $\scA$ itself is clearly generating, we deduce from
Proposition \ref{prop_cogset} that $(\scA,\inj(\scA))$ is a complete cotorsion pair, i.e. that $\scA$ has enough
injectives.  
\item\label{ex_injproj_proj} Assume that $\scA$ has enough projectives. Then $\proj(\scA)$ is generating, and hence the cotorsion pair
$(\proj(\scA),\scA)$ is small. Applying Proposition \ref{prop_cogset} shows that $\proj(\scA)$ is deconstructible.
\end{enumerate}
\end{ex}

We end the section by by recalling that cotorsion pairs cogenerated by sets are also relevant because of their relation to
the cofibrant generation of abelian model structures, as is shown in the following Theorem of Hovey. 

\begin{prop}\label{prop_hoveydec}
Let $\scA$ be a Grothendieck category and let $\calM = (\calC,\calW,\calF)$ be an abelian model structure on $\scA$. 
\begin{enumerate}
\item If $\calM$ is cofibrantly generated, then the cotorsion pairs $(\calC\cap\calW,\calF)$ and $(\calC,\calW\cap\calF)$ are cogenerated by sets.
\item If $(\calC\cap\calW,\calF)$ and $(\calC,\calW\cap\calF)$ are small, then $\calM$ is cofibrantly generated.
\end{enumerate}
\end{prop}
\begin{proof}
(2) is proved in \cite[Lemma 6.7]{Hovey_Cotorsion}. Part (1) is \cite[Lemma 3.1]{HoveyCotorsionArxiv}; however,
it is stated there without proof, so we give an argument for convenience of the reader. Suppose $\calM$ is cofibrantly
generated with a generating set of cofibrations $I\subseteq\cof$ and a generating set of trivial cofibrations
$J\subset\cof\cap\weak$, and put $\calS := \{\coker(f)\ |\ f\in I\}$. As cofibrations are monomorphisms with cofibrant
cokernel, we have $\calS\subseteq\calC$, and we claim that $\calS^{\perp}=\calF\cap\calW$. Indeed, if
$X\in\calS^{\perp}$, then $X\to 0$ has the right lifting property with respect to all maps $f\in I$, and hence is a
trivial fibration by assumption. In other words, $X\in\calW\cap\calF$ as claimed. Similarly one shows that $\calF =
\calT^{\perp}$ for $\calT := \{\coker(g)\ |\ g\in J\}\subseteq\calC\cap\calW$. 
\end{proof}

In particular, Proposition \ref{prop_hoveydec} shows that in case $\scA$ has enough projectives $\calM\leftrightarrow
(\calC,\calW,\calF)$ gives a one-to-one correspondence between cofibrantly generated abelian model structures on $\scA$
and triples $(\calC,\calW,\calF)$ such that both $(\calC\cap\calW,\calF)$ and $(\calC,\calW\cap\calF)$ are cotorsion
pairs cogenerated by sets. 

\subsection{Four model structures on modules over a dg ring}\label{subsection_fourmodels}

In this section we use the results of the previous section to construct four prominent abelian model structures on the
category of modules over a (curved) differential graded ring (dg rings resp. cdg rings for short): Firstly, the standard
injective and projective abelian model structures for modules over a dg ring, and secondly, Positselski's coderived and
contraderived abelian model structures for modules over a cdg ring. 

\begin{notation} A \textit{grading group} \cite[Remark preceeding Section 1.2]{Positselski_TwoKindsOfDerivedCategories} is an abelian group
$\Gamma$ together with a parity homomorphism $|\cdot|:\Gamma\to\ZZ/2\ZZ$ and a 
distinguished element $1\in\Gamma$ satisfying $|1|=\ol{1}$. A $\Gamma$-graded abelian group is a $\Gamma$-indexed family
$X\ua = \{X^k\}_{k\in\Gamma}$ of abelian groups, but we will often drop the index from the notation. We will also
sometimes drop $\Gamma$ from the notation, in which case it is implicitly assumed that a 
 grading group has been fixed. Given such a $\Gamma$-graded abelian group $X$ and some $n\in\Gamma$, we denote
 $\Sigma^n X = X$ the $\Gamma$-graded abelian group given by $(\Sigma^n X)^k := X^{k+n}$ and call it the
 \textit{$n$-fold suspension} of $X$. We also put $\Sigma := \Sigma^1$ and $\Omega := \Sigma^{-1}$. The category of
 $\Gamma$-graded abelian groups has a monoidal structure given by the tensor product $(X\otimes Y)^n :=
 \bigoplus_{p+q=n} X^p\otimes_{\mathbb Z} Y^q$; a \textit{$\Gamma$-graded ring} is an algebra object in that monoidal
 category, and a module over such an algebra object is called a \textit{$\Gamma$-graded module}. A
 \textit{$\Gamma$-graded curved differential graded ring} (cdg ring for short) is a $\Gamma$-graded ring $A$ 
 together with a map $\diff: A\to \Sigma A$ of $\Gamma$-graded abelian groups called differential and an element $w\in A^2$
 such that $\diff(w)=0$, $\diff$  satisfies the Leibniz rule and for any $x\in A$ we have $\diff^2(x) = [w,x]$. The $\Gamma$-graded ring underlying a
 $\Gamma$-graded cdg ring $A$ is denoted $A\us$. For a cdg ring $A$, a \textit{(cdg) module} over $A$ is a $\Gamma$-graded
 module $X$ over $A\us$ together with a map $\diff: X\to\Sigma X$ of $\Gamma$-graded abelian groups satisfying the Leibniz rule and
 $\diff^2(x)=wx$ for all $x\in X$. Given such an $A$-module $X$ and $n\in\Gamma$, the $n$-fold suspension $\Sigma^n X$ carries a natural
 $A$-module structure as follows: its differential $\diff_{\Sigma^n X}$ is given by $\diff_{\Sigma^n X} := (-1)^{|n|}\diff_X$, and the
action of some homogeneous $a\in A$ on some $x\in X$ given by $(-1)^{|a|\cdot |n|}ax$. The $A\us$-module underlying $X$
is denoted $X\us$. Given two $A$-modules $X,Y$, the ($\Gamma$-indexed) complex of $A\us$-linear homomorphisms $X\us\to
\Sigma\ua Y\us$ is denoted $\text{dg-Hom}_A\ua(X,Y)$: for $k\in\Gamma$, its $k$-th component is $\Hom_{A\us}(X\us,\Sigma^k Y\us)$, with
differential sending $f: X\us\to \Sigma^k Y\us$ to $\partial_Y f - (-1)^{|k|} f \partial_X$. The $k$-th cohomology
$\coH^k(\text{dg-Hom}\ua_A(X,Y))$ equals the set $[X,\Sigma^k Y]$ of homotopy classes of morphisms $X\to \Sigma^k
Y$. Finally, we denote $A\Mod\lproj$ (resp. $A\Mod\linj$) the class of $A$-modules whose underlying graded
$A\us$-modules are projective (resp. injective).    \end{notation}

Recall from \cite{Positselski_TwoKindsOfDerivedCategories} the following explicit description of the adjoints of $(-)\us$:

\begin{prop}[\protect{see \cite[Proof of Theorem 3.6]{Positselski_TwoKindsOfDerivedCategories}}]\label{prop_adjointstoforget}
Let $A$ be a cdg ring and define the functors $G^+, G^-: A\us\Mod\to A\Mod$ as follows:
\begin{enumerate}
\item $G^+(X) := X\oplus \Omega X$ as graded abelian groups. An element $(x,y)\in G^+(X)$ is 
  denoted $x+\diff(y)$. The action of some $a\in A$ on $x+\diff(y)$ is given by $ax - (-1)^{|a|} \diff(a) y + (-1)^{|a|} \diff(ay)$,
  while the differential on $G^+(X)$ sends $x+\diff(y)$ to $wy + \diff(x)$. 
\item $G^- := \Sigma\circ G^+$.
\end{enumerate}
Then there are canonical adjunctions $G^+\dashv (-)\us\dashv G^-$. 
\end{prop}

Note that if $A$ is a dg ring (so that we can talk about homology of $A$-modules) the images of $G^+$ and $G^-$ consist
of acyclic modules. This follows immediately from the explicit description of $G^{\pm}$, or alternatively by using the
adjunction property: $\coH^n(G^-(X))\cong [A,\Sigma^n
G^-(X)]\cong\ext^1_{A}(\Omega^{n-1}A,G^-(X))\cong\ext^1_{A\us}(\Omega^{n-1}A\us,X)=0$, where the latter equality 
holds because $A\us$ is projective in $A\us\Mod$; as $G^+=\Omega\circ G^-$, this also shows the acyclicity of objects in
the image of $G^+$. Here we have used that, given a cdg ring $A$ and $X\in A\Mod\lproj$, there is a canonical
isomorphism $\ext^1_{A}(X,-)\cong [\Omega X,-]$. Similarly, if $X\in A\Mod\linj$, we have $\ext^1_{A}(-,X)\cong
[-,\Sigma X]$. These isomorphisms will be used very often in what follows. We will also need the following
characterization of projective and injective objects in $A\Mod$:  

\begin{lem}\label{lem_projective}
Let $A$ be a cdg ring and $X$ an $A$-module. Then $X$ is projective in $A\Mod$ if and only if $X\us$ is projective in
$A\us\Mod$ and $X$ is contractible as an $A$-module. Similarly, $X$ is injective in $A\Mod$ if and only if $X\us$ is
injective in $A\us\Mod$ and $X$ is contractible as an $A$-module. 
\end{lem}
\begin{proof}
For any $A$-module there is a canonical epimorphism $\cone(\id_{\Omega X})\to X$ in $A\Mod$. Hence, if $X$ is projective in
$A\Mod$, it is a summand of $\cone(\id_{\Omega X})$ and hence contractible as an $A$-module. Further, as the forgetful
functor $A\Mod\to A\us\Mod$ is left adjoint to the exact functor $G^-$ (see Proposition \ref{prop_adjointstoforget}), it preserves
projective objects, and hence one direction is proved. Conversely, assume that $X\us$ is projective in $A\us\Mod$ and
$X$ is contractible as an $A$-module. Given another $A$-module $Z$, the projectiveness of $X\us$ implies that there is a
canonical isomorphism $\ext^1_{A}(X,Z)\cong [X,\Sigma Z]$, and the latter group is trivial since $X$ is
contractible. It follows that $X$ is projective in $A\Mod$, as claimed.  

The part on injective objects in $A\Mod$ is similar.
\end{proof}

\begin{lem}\label{lem_doubleorthogonal}
Let $A$ be a cdg ring and $(\calD,\calE)$ be a cotorsion pair with $\Sigma\calD\subseteq\calD$.
\begin{enumerate}
\item If $\calD\subseteq A\Mod\lproj$, then $\calD\cap\calE = \proj(A\Mod)$.
\item If $\calE\subseteq A\Mod\linj$, then $\calD\cap\calE = \inj(A\Mod)$.
\end{enumerate}
\end{lem}
\begin{proof}
We only prove (1), as the proof of (2) is similar. Assuming $\calD\subseteq A\Mod\lproj$, we claim that
$\calD\cap\calE=\proj(A\Mod)$. ``$\supseteq$'': Clearly $\proj(A\Mod)={^{\perp}}A\Mod\subseteq{^{\perp}}\calE =
\calD$. Moreover, if $X\in\proj(A\Mod)$ and $Z\in\calD\subseteq A\Mod\lproj$, we have $\ext^1_{A}(Z,X)\cong
[Z,\Sigma X]=0$ since $X$ is contractible (Lemma \ref{lem_projective}). This shows
$\proj(A\Mod)\subseteq\calD^{\perp}=\calE$, and hence $\proj(A\Mod)\subseteq\calD\cap\calE$. ``$\subseteq$'': By Lemma
\ref{lem_projective} and the assumption that $\calD\subseteq A\Mod\lproj$ it suffices to show that any
$X\in\calD\cap\calD^{\perp}$ is contractible as an $A$-module. Using that $\Sigma\calD\subseteq\calD$ by assumption,
this follows from $0=\ext^1_{A}(\Sigma X,X)\cong [\Sigma X,\Sigma X]$.
\end{proof}

\begin{prop}\label{prop_projectiveinjectivemodel}
For a dg ring $A$, the following hold:
\begin{enumerate}
\item There exists a unique projective abelian model structure on $A\Mod$, denoted
$\calM\uproj(A)$, with $\calW = \Acyc(A)$. $\calM\uproj(A)$ is called the \textit{standard projective model structure}
on $A\Mod$. The class $\calC\uproj(A)$ of cofibrant objects in $\calM\uproj(A)$ is contained in $A\Mod\lproj$.
\item There exists a unique injective abelian model structure on $A\Mod$, denoted $\calM\uinj(A)$, with
  $\calW=\Acyc(A)$. $\calM\uinj(A)$ is called the \textit{standard injective model structure} on $A\Mod$. The class
  $\calF\uinj(A)$ of fibrant objects in $\calM\uinj(A)$ is contained in $A\Mod\linj$. 
\end{enumerate}
Moreover, $\calM\uproj(A)$ and $\calM\uinj(A)$ are cofibrantly generated.
\end{prop}
\begin{proof}
(1) Let $\calS := \{\Sigma^n A\ |\ n\in\Gamma\}$. For any $n\in\Gamma$ and any $X\in A\Mod$ we have a canonical
  isomorphism $\ext^1_{A}(\Omega^n A,X)\cong [A,\Sigma^{n+1} X]\cong \coH^{n+1}(X)$, so it follows that
  $\calS^{\perp}=\Acyc(A)$. Hence, by Corollary \ref{cor_complete}, the cotorsion pair $({^{\perp}}\Acyc,\Acyc)$
  is complete. By Corollary \ref{cor_injprojmodel} and the thickness of $\Acyc(A)$ it remains to show that
  ${^{\perp}}\Acyc\cap\Acyc=\proj(A\Mod)$, so that by Lemma \ref{lem_doubleorthogonal} it suffices to show that
  ${^{\perp}}\Acyc\subseteq A\Mod\lproj$. For this, note that for any 
  $X\in{^{\perp}}\Acyc$ and any $Z\in A\us\Mod$, we have $0=\ext^1_A(X,G^-(Z))\cong\ext^1_{A\us}(X\us,Z)$, so that
  $X\us$ is projective in $A\us\Mod$ as claimed. Here we used that the image of $G^-$ consists of acyclic $A$-modules.

(2) By Corollary \ref{cor_injprojmodel} and Proposition \ref{prop_cogset} it suffices to show that $\Acyc(A)$ is
  generating and deconstructible, and that $\Acyc(A)\cap\Acyc(A)^{\perp}=\inj(A\Mod)$. By Lemma \ref{lem_projective}
  $\proj(A\Mod)\subseteq\Acyc(A)$, so $\Acyc(A)$ is generating. The deconstructibility of $\Acyc(A)$ follows
  from Proposition \ref{prop_corpullbackmonadic} applied to the monadic forgetful functor $: A\Mod\to\Ch_\Gamma(\ZZ)$ and the fact
\cite[Theorem 4.2.(2)]{StovicekHillGrothendieck} that $\Acyc(\ZZ)\subset\Ch_\Gamma(\ZZ)$ is deconstructible (in
loc.cit. the result is proved for $\Gamma=\ZZ$, but the arguments carry over to the case of a general grading
group). Finally, the equality $\Acyc(A)\cap\Acyc(A)^{\perp}=\inj(A\Mod)$ again follows from Lemma
\ref{lem_doubleorthogonal} once we've showed that for any $X\in\Acyc(A)^{\perp}$ its underlying $A\us$-module $X\us$ is
injective. Indeed, if $Z\in A\us\Mod$, we have $0=\ext^1_{A}(G^+(Z),X)\cong\ext^1_{A\us}(Z,X\us)$, where the first
equality holds because the image of $G^+$ consists of acyclic $A$-modules.  

The statement about cofibrant generation follows from Proposition \ref{prop_hoveydec}.
\end{proof}

\begin{prop}\label{prop_cocontraderivedmodel}
For a cdg ring $A$, the following hold:
\begin{enumerate}
\item There exists a unique projective abelian model structure on $A\Mod$, denoted $\calM\uctr(A)$, such that
  $\calC=A\Mod\lproj$. $\calM\uctr(A)$ is called the \textit{contraderived model structure} on $A\Mod$.
\item There exists a unique injective abelian model structure on $A\Mod$, denoted $\calM\uco(A)$, such that
  $\calF=A\Mod\linj$. $\calM\uctr(A)$ is called the \textit{coderived model structure} on $A\Mod$.
\end{enumerate}
Moreover, $\calM\uctr(A)$ and $\calM\uco(A)$ are cofibrantly generated. 
\end{prop}
\begin{proof}
(1) By Corollary \ref{cor_injprojmodel} and Proposition \ref{prop_cogset} we have to show that $A\Mod\lproj$ is
  generating and deconstructible, that $A\Mod\lproj\cap A\Mod\lproj^{\perp}=\proj(A\Mod)$ and that $A\Mod\lproj^{\perp}$
  is satisfies the $2$-out-of-$3$ property. By Lemma \ref{lem_projective}, $\proj(A\Mod)\subseteq A\Mod\lproj$, so 
  $A\Mod\lproj$ is generating. For the deconstructibility of $A\Mod\lproj$, we again apply Proposition
  \ref{prop_corpullbackmonadic}: The forgetful functor $(-)\us: A\Mod\to A\us\Mod$ is monadic, for example by the
  explicit description of its left adjoint $G^+$ in Proposition \ref{prop_adjointstoforget}, and $A\Mod\lproj$ is the
  preimage under $(-)\us$ of  $\proj(A\us\Mod)$, which is deconstructible by Example
  \ref{ex_injproj}.\ref{ex_injproj_proj}. Finally, $A\Mod\lproj\cap A\Mod\lproj^{\perp}=\proj(A\Mod)$ follows
  from Lemma \ref{lem_doubleorthogonal}, and the $2$-out-of-$3$ property of $A\Mod\lproj^{\perp}$ is ensured by the dual
  of Lemma \ref{lem_twooutofthree}, using that $A\Mod\lproj$ is closed under kernels of epimorphisms.

(2) By definition, an $A$-module $X$ belongs to $A\Mod\linj$ if and only if $X\us\in\inj(A\Mod)$, i.e.
  $0=\ext^1_{A\us}(Z,X\us)=\ext^1_A(G^+(Z),X)$ for all $Z\in A\us\Mod$. Hence $A\Mod\linj =
  G^+(A\us\Mod)^{\perp}$, so that by the deconstructibility of $A\us\Mod$ and the exactness and
  cocontinuity of $G^+$, $G^+(A\us\Mod)$ is deconstructible, too. This shows that $A\Mod\linj=\calS^{\perp}$ for
  some set $\calS\subset A\Mod$, and hence $({^{\perp}}A\linj,A\linj)$ is a complete cotorsion pair by Corollary
  \ref{cor_complete}. As above, ${^{\perp}}A\linj\cap A\linj=\inj(A\Mod)$ follows from Lemma \ref{lem_doubleorthogonal},
  and the $2$-out-of-$3$ property of ${^{\perp}}A\Mod\linj$ follows from Lemma \ref{lem_twooutofthree} together with the
  fact that $A\Mod\linj$ is closed under cokernels of monomorphisms.

The cofibrant generation follows from Proposition \ref{prop_hoveydec}.
\end{proof}
\begin{cor}\label{cor_comparison}
For a dg ring $A$, the identity on $A\Mod$ is a left Quillen functor $\calM\uproj(A)\to\calM\uctr(A)$ and a right
Quillen functor $\calM\uinj(A)\to\calM\uco(A)$.  
\end{cor}
\begin{proof}
Unraveling the definitions, this means that we have $\calC\uproj(A)\subseteq A\Mod\lproj$ and $\calF\uinj(A)\subseteq
A\Mod\linj$, which was shown in Proposition \ref{prop_projectiveinjectivemodel}.
\end{proof}

Following \cite{Positselski_TwoKindsOfDerivedCategories}, weakly trivial objects in $\calM\uco(A)$ are
called \textit{coacyclic}, while weakly trivial objects in $\calM\uctr(A)$ are called
\textit{contraacyclic}. We denote them $\calW\uco(A)$ and $\calW\uctr(A)$, respectively. Corollary
\ref{cor_comparison} implies that $\calW\uco(A)\subseteq\Acyc(A)\supseteq\calW\uctr(A)$, so coacyclic and contraacyclic
modules are in particular acyclic in the classical sense. In general, we can only give the following description:

\begin{prop}\label{prop_coctracyclic}
Let $A$ be a dg ring and $X\in A\Mod$. 
\begin{enumerate}
\item $X$ is contraacyclic if and only if for each $Z\in A\Mod\lproj$ the homomorphism complex $\dgHom\ua_A(Z,X)$
  is acyclic, if and only if $[Z,X]=0$ for all $Z\in A\Mod\lproj$ . 
\item $X$ is coacyclic if and only if for each $Z\in A\Mod\linj$ the homomorphism complex $\dgHom\ua_A(X,Z)$ is
  acyclic if and only if $[X,Z]=0$ for all $Z\in A\Mod\linj$. 
\end{enumerate}
In particular, any contractible $A$-module is both contraacyclic and coacyclic.
\end{prop}
\begin{proof}
(i) follows from $\ext^1_A(Z,-)\cong [\Omega Z,-]$ for $Z\in A\Mod\lproj$ and the isomorphism $\coH^k\left[\text{dg-Hom}\ua_A(X,Y)
\right]\cong [X,\Sigma^k Y]$, and (ii) follows using $\ext^1_A(-,Z)\cong [-,\Sigma Z]$ for $Z\in
A\Mod\linj$.
\end{proof}
\begin{lem}\label{lem_totalizations}
Let $A$ be a cdg ring and $...\stackrel{p_2}{\to} X_1\stackrel{p_1}{\to} X_0$ be an inverse system of contraacyclic
$A$-modules with all $p_n$ being epimorphisms. Then $\invlim X_n$ is $A$-contraacyclic, too. In particular, the
totalization formed by taking products of any bounded above exact sequence of $A$-modules is contraacyclic. 
\end{lem}
\begin{proof}
The first statement follows from the existence of a short exact sequence $0\to\invlim X_n\to\prod X_n\to\prod X_n\to 0$
and the fact that $\calW\uctr(A)$ satisfies the $2$-out-of-$3$ property.
It remains to show that the totalization $\tot\uprod(X\la)$ formed by taking products of an exact, bounded
above sequence of $A$-modules $...\stackrel{f_3}{\to} X_2\stackrel{f_2}{\to} X_1\stackrel{f_1}{\to} X_0\to 0\to ...$ is
contraacyclic, which is essentially a special case of the first statement: $\tot\uprod(X\la)$ is the inverse limit of the totalizations
of the soft truncations $0\to X_n/\image(f_{n+1})\to X_{n-1}\to ...\to X_1\to X_0\to 0$, which in turn are iterated
extensions of contractible $A$-modules, hence contraacyclic by Proposition \ref{prop_coctracyclic}.
\end{proof}

In case some mild conditions on $A\us$ is satisfied, Positselski gives the following description of coacyclic
and contraacyclic modules:  

\begin{prop}[\protect{\cite[Theorem 3.7, 3.8]{Positselski_TwoKindsOfDerivedCategories}}]
Let $A$ be a cdg ring.  
\begin{enumerate}
\item Suppose any countable product of projective $A\us$-modules has finite projective dimension. Then $\calW\uctr(A)$
  equals the smallest thick triangulated subcategory of $\coH^0(A\Mod)$ closed under products and containing totalizations
  of exact sequences of $A$-modules.
\item Suppose any countable sum of injective $A\us$-modules has finite injective dimension. Then $\calW\uco(A)$ equals
  the smallest thick triangulated subcategory of $\coH^0(A\Mod)$ closed under coproducts and containing totalizations of exact
  sequences of $A$-modules. 
\end{enumerate}
\end{prop}

The next proposition is contained in greater generality in \cite[Section 3.6]{Positselski_TwoKindsOfDerivedCategories}. Restricting to ordinary rings here, we give a direct proof in the setting of abelian categories.

\begin{prop}\label{prop_finglobdim}
If $R$ is an ordinary ring of finite left-global dimension (i.e. $\gldim(R\Mod)<\infty$), then
$\calM\uctr(R)=\calM\uproj(R)$ and $\calM\uco(R)=\calM\uinj(R)$. 
\end{prop}
\begin{proof}
By Corollary \ref{cor_comparison} we have $\calC\uproj(R)\subseteq\calC\uctr(R)$, so it suffices to show the reverse
inclusion, i.e. that for any $X\in\Ch_\Gamma(\Proj(R))$ we have $X\in{^{\perp}}\Acyc(R)$. Suppose first that
$X\in\Ch_\Gamma(\Proj(R))\cap\Acyc(R)$. Since $\gldim(R\Mod)<\infty$ by assumption, the syzygies $Z^n(X)$ of $X$ are projective in this
case, and hence $X$ is contractible. By Lemma \ref{lem_projective}, it follows that $X\in\proj(\Ch_\Gamma(R))
\subseteq{^{\perp}}\Acyc(R)$ as claimed. In the general case, pick a cofibrant resolution $p: P\to X$ in
$\calM\uproj(R)$, i.e. $p$ is an epimorphism with $K := \ker(p)\in\Acyc(R)$ and $P\in\calC\uproj(R)$. As the components of
$X$ are projective, $p$ is degree-wise split, so $K\in\Acyc(R)\cap\Ch_\Gamma(\Proj(R))\subseteq{^{\perp}}\Acyc(R)$ 
by the first case. Moreover, applying $\text{dg-Hom}\ua_R(-,Z)$ to $0\to K\to P\to X\to 0$ for $Z\in\Acyc(R)$ and taking
cohomology shows $[X,Z]=0$ as claimed. The proof of $\calM\uco(R)=\calM\uinj(R)$ is similar.    
\end{proof}

Morphisms of dg rings induce Quillen adjunctions between the four models:

\begin{prop}\label{prop_fouradjunctions}
Let $\varphi: R\to A$ be a morphism of dg rings and let $U_\varphi: A\Mod\to R\Mod$ be the forgetful functor.
\begin{enumerate}
\item $A\otimes_R -\dashv U_\varphi$ is a Quillen adjunction
  $\calM\uproj(R)\rightleftarrows\calM\uproj(A)$.
\item$A\otimes_R -\dashv U_\varphi$ is a Quillen adjunction
   $\calM\uctr(R)\rightleftarrows\calM\uctr(A)$.
\item $U_\varphi\dashv\Hom_R(A,-)$ is a Quillen adjunction
  $\calM\uinj(A)\rightleftarrows\calM\uinj(R)$.
\item\label{item_fouradj4} $U_\varphi\dashv\Hom_R(A,-)$ is a Quillen adjunction
  $\calM\uco(A)\rightleftarrows\calM\uco(R)$. 
\item\label{item_fouradj5} If $A\us$ is projective as an $R\us$-module, then $U_\varphi\dashv\Hom_R(A,-)$ is a Quillen
  adjunction $\calM\uctr(A)\rightleftarrows\calM\uctr(R)$. 
\end{enumerate}
\end{prop}
\begin{proof}
Given an adjunction between model categories, checking that it is a Quillen adjunction means either to check that the
left adjoint preserves (trivial) cofibrations, or, equivalently, that the right adjoint preserves (trivial)
fibrations. The point here is to check the alternative which involves the parts of the model 
structures that we know explicitly. As an example, we check that $U_\varphi\dashv\Hom_R(A,-)$ is a Quillen adjunction
$\calM\uco(A)\rightleftarrows\calM\uco(R)$ by proving that $\Hom_R(A,-)$ preserves (trivial) fibrations. A fibration in
$\calM\uco(R)$ is an epimorphism $f: Z\to X$ with 
$\ker(f)\in\calF\uco(R)=R\Mod\linj$. Since $\Hom_R(A,-)\us=\Hom_{R\us}(A\us,-)$ and $\ker(f)\us\in\inj(R\us\Mod)$, we
see that $\Hom_R(A,f): \Hom_R(A,Z)\to\Hom_R(A,X)$ is an epimorphism with kernel $\Hom_R(A,\ker(f))$. As
$\Hom_R(A,\ker(f))\us\cong\Hom_{R\us}(A\us,\ker(f)\us)$ and $\Hom_{R\us}(A\us,-)$ is right adjoint to the exact functor
$A\us\Mod\to R\us\Mod$, we get $\ker(\Hom_R(A,f))\us\in\inj(A\us\Mod)$, hence $\ker(\Hom_{R}(A,f))\in A\Mod\linj$. In other
words, $\Hom_R(A,f)$ is a fibration. Similarly, let $f$ is a trivial fibration in $\calM\uco(R)$. Then
$\ker(f)\in\inj(R\Mod)$, so $f$ is a split epimorphism with injective kernel. Since $\Hom_R(A,-)$ preserves
injectives as the right adjoint to the exact functor $A\Mod\to R\Mod$, $\Hom_R(A,f)$ is a split epimorphism with
injective kernel, too, and hence a trivial fibration in $\calM\uco(A)$.
\end{proof}

\begin{rem}\label{rem_gen}
The results of this section generalize to the case where we replaced our base category of
abelian groups by any Grothendieck category $\scA$ equipped with a closed symmetric monoidal tensor product
$-\otimes-:\scA\times\scA\to\scA$. Given a grading group $\Gamma$, the category $\scA^{\Gamma}$ of $\Gamma$-indexed
objects in $\scA$ and the category $\Ch_\Gamma(\scA)$ of $\Gamma$-indexed complexes in $\scA$ are again
Grothendieck and inherit a closed symmetric monoidal tensor product; one can then speak about algebra objects in these
categories ($\Gamma$-graded rings and $\Gamma$-graded dg rings in case $\scA={\mathbb Z}\Mod$), and form their
categories of modules, which are again Grothendieck by Lemma \ref{lem_alggroth}. The arguments of this section carry
over to this situation and show that for any $\Gamma$-graded dg ring $A$ over $(\scA,\otimes)$ its category of
modules carries the standard injective model structure, determined by injectivity and $\calW = \Acyc(A)$, and the
coderived model structure, determined by injectivity and $\calF = A\Mod\linj$. The only difference is that one has to
argue why $\Acyc(A)$ and ${^{\perp}}A\Mod\linj$ are generating; for example, this follows from the fact that both
$\Acyc(A)$ and ${^{\perp}}A\Mod\linj$ contain the class of contractible $A$-modules, and any $A$-module $X$ is the
quotient of the contractible $A$-module $\cone(\id_{\Omega X})$. If $\scA$ has enough projectives, then so do $\scA^\Gamma$,
$\Ch_\Gamma(\scA)$, $A\us\Mod$ and $A\Mod$, and we also get the standard projective and the 
contraderived model structure on $A\Mod$, determined by projectivity and $\calW=\Acyc(A)$
resp. $\calC=A\Mod\lproj$. Also see Remarks \ref{rem_gen2} and \ref{rem_gen3}.

This generalization applies for example to the case where $\scA=\QCoh(X)$ for a quasi-compact and quasi-separated scheme
$X$ (see \cite[Proposition 66]{Murfet_ModulesOverAScheme}, or to $\scA=\calO_X\Mod$ for some ringed space $(X,\calO_X)$
(see \cite[Theorem 18.1.6]{KashiwaraShapira}).
\end{rem}

\subsection{Localization of abelian model structures}

Let $\scA$ be a bicomplete abelian category and $\calM_1$, $\calM_2$ two injective abelian model structures
on $\scA$ such that $\id: \calM_2\to\calM_1$ is right Quillen. In this section we will construct from this datum another
hereditary (usually non-injective) abelian model structure, called the right localization of $\calM_1$ with respect to
$\calM_2$ and denoted $\calM_1/\calM_2$, whose homotopy category fits into a colocalization sequence with the homotopy
categories of $\calM_1$ and $\calM_2$. The arguments in the proof are elementary homological algebra only, and in
particular do not use Quillen's small object argument. Hence, we neither need to assume that the model structures we
work with are cofibrantly generated, nor that the underlying bicomplete abelian category is Grothendieck. Instead, the
assumptions are completely self-dual, and we get a dual left localization result for comparable pairs of projective
abelian model structures. We will see in the next section that what we call localizations here are indeed Bousfield
localizations in the sense of \cite{Hirschhorn_LocalizationOfModelCategories}. 

\begin{fact}\label{fact_weakequivalences}
Let $\scA$ be an abelian category equipped with an abelian model structure $\calM=(\calC,\calW,\calF)$. Then,
given a morphism $f: A\to B$ the following are equivalent:  
\begin{enumerate}
\item[(i)]\label{item:fwe} $f$ is a weak equivalence.
\item[(ii)]\label{item:ffact} $f$ factors as $A\stackrel{\iota}{\rightarrowtail} X\stackrel{p}{\twoheadrightarrow} B$ with
  $\coker(\iota)\in\calC\cap\calW$ and $\ker(p)\in\calF\cap\calW$.  
\end{enumerate}
\end{fact}
\begin{proof}
(ii)$\Rightarrow$(i) is clear, and (i)$\Rightarrow$(ii) follows from the factorization axiom.
\end{proof}

Fact \ref{fact_weakequivalences} is meant to motivate the description of $\calW$ in the following proposition.

\begin{prop}\label{prop_localization}
Let $\scA$ be a bicomplete abelian category and $\calM_1= (\calW_1,\calF_1)$ and $\calM_2= (\calW_2,\calF_2)$ be
 injective abelian model structures on $\scA$ with $\calF_2\subset\calF_1$. Then there exists a
hereditary abelian model structure on $\scA$, called the \textit{right localization} of $\calM_1$ with respect to
$\calM_2$ and denoted $\calM_1/\calM_2$, with $\calC = \calW_2$, $\calF = \calF_1$ and  
\begin{align*}
\calW & \makebox[8mm]{$:=$} \{X\in\scA\ |\ \exists \text{ ex. seq. } 0\to X\to A\to B\to 0\text{ with }A\in\calF_2, B\in\calW_1\}\\
& \makebox[8mm]{$=$} \{X\in\scA\ |\ \exists \text{ ex. seq. } 0\to A\to B\to X\to 0\text{ with }A\in\calF_2, B\in\calW_1\}.
\end{align*}
Moreover, $X\in\calW$ if and only if it belongs to the essential image of $\calF_2\to\Ho(\calM_1)$.
\end{prop}

In the course of the proof of Proposition \ref{prop_localization} we will need the following lemmata:
\begin{lem}\label{lem_stablyisomorphic}
Let $\calF$ be a Frobenius category and let $\inj$ be its class of projective-injective objects. Then the following
hold:
\begin{enumerate}
\item Assume $\calF$ weakly idempotent complete, i.e. every split monomorphism has a cokernel. Then, given
  $X,Y\in\calF$, we have $X\cong Y$ in the stable category 
  $\calF/\inj$ if and only if there exist $I,J\in\inj$ such that $X\oplus J\cong Y\oplus I$ in $\calF$.
\item Given an admissible short exact sequence $X\rightarrowtail Y\twoheadrightarrow Z$, there exists a
  canonical morphism $Z\to \Sigma X$ in the stable category $\calF/\inj$ such that $X\to Y\to Z\to \Sigma X$ is a
  distinguished triangle in $\calF/\inj$.
\end{enumerate}
\end{lem}
\begin{proof}
(1) ``$\Leftarrow$'' is clear since all object in $\calI$ are isomorphic to $0$ in $\calF/\inj$. ``$\Rightarrow$'':
  Suppose $X\cong Y$ in $\calF/\inj$. By definition, this means that we can find $f: X\to Y$, $g: Y\to X$ such that $\id_Y-fg$
  and $\id_X - gf$ respectively factor through some object in $\inj$. Pick $p: I\to X$ and $u: X\to I$ with $I\in\inj$
  such that $\id_X = gf+pu$. Then $(f,u)^t: X\to Y\oplus I$ is a split monomorphism with left inverse $(g,p): Y\oplus
  I$, so replacing $Y$ by $Y\oplus I$ we may assume $gf=\id_X$. In this case, $f$ is a split monomorphism, so by assumption
  we can choose a cokernel $k: Y\to K$ of $f$, and we have $s: K\to Y$ be such that $sk=\id-fg$. Then, picking
  morphisms $q: J\to Y$ and $v: Y\to J$ with $J\in\inj$ such that $\id_Y = fg + qv$ we get $\id_K = ks =
  k(fg+qv)s=(kq)(vs)$. Again using the assumption that $\calF$ is weakly idempotent complete, we conclude that $K$ is a
  summand of $J$, and in particular $K\in\inj$. Since $Y\cong X\oplus K$, this proves the claim. 

(2) See \cite[Lemma 2.7]{Happel}.
\end{proof}
\begin{lem}\label{lem_exact}
Let $\scA$ be an abelian category and $(\calW,\calF)$ be a coresolving cotorsion pair with enough injectives. Then for any short
exact sequence $0\to X_1\to X_2\to X_3\to 0$ in $\scA$ there exists a commutative diagram 
\begin{equation*}\begin{tikzpicture}[description/.style={fill=white,inner sep=2pt}]
    \matrix (m) [matrix of math nodes, row sep=1.5em,
                 column sep=1.5em, text height=1.5ex, text depth=0.25ex,
                 inner sep=0pt, nodes={inner xsep=0.3333em, inner ysep=0.3333em}]
    {
      & 0 & 0 & 0 \\
       0 & X_1 & X_2 & X_3 & 0 \\
       0 & A_1 & A_2 & A_3 & 0 \\
       0 & B_1 & B_2 & B_3 & 0 \\
       & 0 & 0 & 0 \\
    };
    \draw[->] (m-1-2) -- (m-2-2);
    \draw[->] (m-1-3) -- (m-2-3);
    \draw[->] (m-1-4) -- (m-2-4);
    \draw[->] (m-2-2) -- (m-3-2);
    \draw[->] (m-2-3) -- (m-3-3);
    \draw[->] (m-2-4) -- (m-3-4);
    \draw[->] (m-3-2) -- (m-4-2);
    \draw[->] (m-3-3) -- (m-4-3);
    \draw[->] (m-3-4) -- (m-4-4);
    \draw[->] (m-4-2) -- (m-5-2);
    \draw[->] (m-4-3) -- (m-5-3);
    \draw[->] (m-4-4) -- (m-5-4);
    \draw[->] (m-2-1) -- (m-2-2);
    \draw[->] (m-2-2) -- (m-2-3);
    \draw[->] (m-2-3) -- (m-2-4);
    \draw[->] (m-2-4) -- (m-2-5);
    \draw[->] (m-3-1) -- (m-3-2);
    \draw[->] (m-3-2) -- (m-3-3);
    \draw[->] (m-3-3) -- (m-3-4);
    \draw[->] (m-3-4) -- (m-3-5);
    \draw[->] (m-4-1) -- (m-4-2);
    \draw[->] (m-4-2) -- (m-4-3);
    \draw[->] (m-4-3) -- (m-4-4);
    \draw[->] (m-4-4) -- (m-4-5);
\end{tikzpicture}\end{equation*}
such that $A_i\in\calF$, $B_i\in\calW$ and all rows and columns are exact.
\end{lem}
\begin{proof}
Let $0\to X_1\to A_1\to B_1\to 0$ be short exact with $A_1\in\calF$, $B_1\in\calW$. Taking the pushout of $A_1\leftarrow
X_1\to X_2$ we get a monomorphism of exact sequences
\begin{equation*}\begin{tikzpicture}[description/.style={fill=white,inner sep=2pt}]
    \matrix (m) [matrix of math nodes, row sep=1.5em,
                 column sep=1.5em, text height=1.5ex, text depth=0.25ex,
                 inner sep=0pt, nodes={inner xsep=0.3333em, inner ysep=0.3333em}]
    {
       0 & X_1 & X_2 & X_3 & 0 \\ 
       0 & A_1 & Z & X_3 & 0 \\ 
    };
    \draw[->] (m-1-1) -- (m-1-2);
    \draw[->] (m-1-2) -- (m-1-3);
    \draw[->] (m-1-3) -- (m-1-4);
    \draw[->] (m-1-4) -- (m-1-5);
    \draw[->] (m-2-1) -- (m-2-2);
    \draw[->] (m-2-2) -- (m-2-3);
    \draw[->] (m-2-3) -- (m-2-4);
    \draw[->] (m-2-4) -- (m-2-5);
    \draw[->] (m-1-2) -- (m-2-2);
    \draw[->] (m-1-2) -- (m-2-2);
    \draw[->] (m-1-3) -- (m-2-3);
    \draw[->] (m-1-4) -- (m-2-4);
\end{tikzpicture}\end{equation*}
whose cokernel $0\to B_1\to B_1\to 0\to 0$ is an exact sequence in $\calW$. Replacing $0\to X_1\to X_2\to X_3\to 0$ by
$0\to A_1\to Z\to X_3\to 0$ we may therefore assume $A_1 = X_1\in\calF$ right from the beginning. In this case, choose an exact sequence $0\to X_2\to
A_2\to B_2\to 0$ with $A_2\in\calF$, $B_2\in\calW$. Forming the pushout of $A_2\leftarrow X_2\to X_3$ we get the
following commutative diagram:
\begin{equation*}\begin{tikzpicture}[description/.style={fill=white,inner sep=2pt}]
    \matrix (m) [matrix of math nodes, row sep=1.5em,
                 column sep=1.5em, text height=1.5ex, text depth=0.25ex,
                 inner sep=0pt, nodes={inner xsep=0.3333em, inner ysep=0.3333em}]
    {
       0 & A_1 & X_2 & X_3 & 0 \\ 
       0 & A_1 & A_2 & Z & 0 \\ 
    };
    \draw[->] (m-1-1) -- (m-1-2);
    \draw[->] (m-1-2) -- (m-1-3);
    \draw[->] (m-1-3) -- (m-1-4);
    \draw[->] (m-1-4) -- (m-1-5);
    \draw[->] (m-2-1) -- (m-2-2);
    \draw[->] (m-2-2) -- (m-2-3);
    \draw[->] (m-2-3) -- (m-2-4);
    \draw[->] (m-2-4) -- (m-2-5);
    \draw[double, double distance=0.5mm] (m-1-2) -- (m-2-2);
    \draw[->] (m-1-3) -- (m-2-3);
    \draw[->] (m-1-4) -- (m-2-4);
\end{tikzpicture}\end{equation*}
By definition, the right square is pushout, but as $X_2\to A_2$ is a monomorphism, it is also pullback, and hence the
second row is exact. Since $\calF$ is closed under cokernels of monomorphisms by assumption, we conclude
$Z\in\calF$. Hence we have constructed a monomorphism from $0\to A_1\to X_2\to X_3\to 0$ into a short exact sequence in
$\calF$ with cokernel $0\to 0\to B_2\to B_2\to 0$ lying in $\calW$, as required.
\end{proof}
\begin{proof}[Proof of Proposition \ref{prop_localization}]
Recall from Corollary \ref{cor_coresres} that $\calM_1$ and $\calM_2$ are automatically hereditary, and in particular
$\calF_1$ and $\calF_2$ are closed under taking cokernels of monomorphisms; this will be used several times in the
proof. We begin by showing that both definitions of $\calW$ agree. 

Suppose $X\in\scA$ admits a short exact sequence $0\to A\to B\to X\to 0$ with $A\in\calF_2$ and $B\in\calW_1$. Since
$(\calW_1,\calF_1)$ is a cotorsion pair with 
$\calW_1\cap\calF_1=\inj$, we can choose a short exact sequence $0\to B\to I\to B\p\to 0$ with $I\in\inj$
and $B\p\in\calW_1$. Taking pushout, we get the following commutative diagram with exact rows and columns and
bicartesian upper right square: 
\begin{equation*}\begin{tikzpicture}[description/.style={fill=white,inner sep=2pt}]
    \matrix (m) [matrix of math nodes, row sep=1.5em,
                 column sep=1.5em, text height=1.5ex, text depth=0.25ex,
                 inner sep=0pt, nodes={inner xsep=0.3333em, inner ysep=0.3333em}]
    {
      && 0 & 0 \\
       0 & A & B & X & 0 \\
       0 & A & I & A\p & 0 \\
       && B\p & B\p \\
       && 0 & 0 \\
    };
    \draw[->] (m-1-3) -- (m-2-3);
    \draw[->] (m-1-4) -- (m-2-4);
    \draw[->] (m-2-1) -- (m-2-2);
    \draw[->] (m-2-2) -- (m-2-3);
    \draw[->] (m-2-3) -- (m-2-4);
    \draw[->] (m-2-4) -- (m-2-5);
    \draw[->] (m-3-1) -- (m-3-2);
    \draw[->] (m-3-2) -- (m-3-3);
    \draw[->] (m-3-3) -- (m-3-4);
    \draw[->] (m-3-4) -- (m-3-5);
    \draw[double,double distance=0.5mm] (m-2-2) -- (m-3-2);
    \draw[double,double distance=0.5mm] (m-4-3) -- (m-4-4);
    \draw[>->] (m-2-3) -- (m-3-3);
    \draw[->] (m-2-4) -- (m-3-4);
    \draw[->] (m-3-3) -- (m-4-3);
    \draw[->] (m-4-3) -- (m-5-3);
    \draw[->] (m-3-4) -- (m-4-4);
    \draw[->] (m-4-4) -- (m-5-4);
\end{tikzpicture}\end{equation*}
As $\calF_2$ is closed under cokernels of monomorphisms, we have $A\p\in\calF_2$, and hence $0\to X\to A\p\to B\p$ is our desired sequence.

Reversing the argument (using that any $A\in\calF_2$ admits a short exact sequence $0\to A\p\to
I\to A\to 0$ with $I\in\calW_2\cap\calF_2=\inj$ and $A\p\in\calF_2$), we see that the existence of a short
exact sequence $0\to X\to A\to B\to 0$ with $A\in\calF_2$ and $B\in\calW_1$ also implies the existence of a short exact
sequence $0\to A\p\to B\p\to X\to 0$ with $A\p\in\calF_2$ and $B\p\in\calW_1$. Hence the two definitions of $\calW$ agree.

For the thickness and the last claim, the argument goes as follows: As $(\calW_1,\calF_1)$ is a complete cotorsion pair, for any
$X\in\scA$ there exists an exact sequence $0\to X\to A\to B\to 0$ with $A\in\calF_1$ and $B\in\calW_1$. The assignment
$X\mapsto A$ defines an additive functor $\scA\to \calF_1/\calF_1\cap\calW_1 = \calF_1/\inj$ (it is a short check that
any morphism between objects of $\calF_1$ factoring through an object in $\calW_1$ actually factors through some object in
$\calF_1\cap\calW_1$; see also Proposition \ref{prop_cotorsiontorsion}) and in particular the object $A$ from above is unique
up to canonical isomorphism in $\calF_1/\inj$. Next, form the full subcategory $\calF_2/\inj$ of $\calF_1/\inj$
consisting of objects $\calF_2$ (recall that passing to the stable category does not change objects). It is isomorphism
closed by Lemma \ref{lem_stablyisomorphic}, and using this we see that $\calW$ equals the preimage of $\calF_2/\inj$ under 
$\scA\to\calF_1/\inj$. With this description at hand, we can now prove the thickness of $\calW$. As the functor
$\scA\to\calF_1/\inj$ from above is additive and $\calF_2/\inj$ is closed under direct summands in $\calF_1/\inj$,
$\calW$ is closed under direct summands, too. For the 2-out-of-3 property, note that $\calF_2/\inj$ is a triangulated
subcategory of $\calF_1/\inj$, so it suffices to show that $\scA\to\calF_1/\inj$ turns short exact sequences into
distinguished triangles, which follows from Lemma \ref{lem_stablyisomorphic}(2) and Lemma \ref{lem_exact}.

It remains to show that $\calM_1/\calM_2$ is hereditary and that $(\calC\cap\calW,\calF)$ and $(\calC,\calW\cap\calF)$
are complete cotorsion pairs. The former is true since $\calF = \calF_1$ is closed under cokernels of monomorphisms by
assumption and $\calC=\calW_2$ even satisfies the $2$-out-of-$3$ property; the latter will 
follow once we showed that $(\calC\cap\calW,\calF)=(\calW_1,\calF_1)$ and $(\calC,\calW\cap\calF)=(\calW_2,\calF_2)$,
as these are complete cotorsion pairs by assumption.

$\calW\cap\calF = \calF_2$: Suppose $X\in\calW\cap\calF=\calW\cap\calF_1$ and let $0\to X\to A\to B\to 0$ be a short
exact sequence with $A\in\calF_2$ and $B\in\calW_1$. By definition, $\ext^1(\calW_1,X)=0$, so the sequence splits and
$X\in\calF_2$ as $\calF_2$ is thick. This shows that $\calF_1\cap\calW\subset\calF_2$, and the reverse inclusion
$\calF_2\subset\calF_1\cap\calW$ is clear.

$\calC\cap\calW = \calW_1$: Suppose $X\in\calC\cap\calW=\calW_2\cap\calW$ and let $0\to A\to B\to X\to 0$ be a short exact
sequence with $A\in\calF_2$ and $B\in\calW_1$. Again, this sequence is split since $X\in{^{\perp}}\calF_2$, so
$X\in\calW_1$. Hence $\calW_2\cap\calW\subset\calW_1$, and the reverse inclusion is clear. 
\end{proof}
\begin{cor}\label{cor_loc} In the situation of Proposition \ref{prop_localization} the sequence
$$\Ho(\calM_2)\xrightarrow{\ \ \bfR\id\ \ }\Ho(\calM_1)\xrightarrow{\ \ \bfR\id\ \ }\Ho(\calM_1/\calM_2)$$
is a colocalization sequence \cite[Definition 3.1]{Krause_StableDerived} of triangulated categories.
\end{cor}
\begin{proof}
Consider the following commutative diagram
\begin{equation*}\begin{tikzpicture}[description/.style={fill=white,inner sep=2pt}]
    \matrix (m) [matrix of math nodes, row sep=2em,
                 column sep=5.5em, text height=1.5ex, text depth=0.25ex,
                 inner sep=0pt, nodes={inner xsep=0.3333em, inner ysep=0.3333em}]
    {
       \Ho(\calM_1/\calM_2) & \Ho(\calM_1) & \Ho(\calM_2)\\
       \calF_1\cap\calW_2/\calI & \calF_1/\inj & \calF_2/\inj.\\
    };
    \draw[->] ($(m-1-1.east) + (0,0.6mm)$) -- node[above,scale=0.75]{$\bfL\id$} ($(m-1-2.west) + (0,0.6mm)$);
    \draw[->] ($(m-1-2.west) - (0,0.6mm)$) -- node[below,scale=0.75]{$\bfR\id$} ($(m-1-1.east) - (0,0.6mm)$);
    \draw[->] ($(m-1-2.east) + (0,0.6mm)$) -- node[above,scale=0.75]{$\bfL\id$} ($(m-1-3.west) + (0,0.6mm)$);
    \draw[->] ($(m-1-3.west) - (0,0.6mm)$) -- node[below,scale=0.75]{$\bfR\id$} ($(m-1-2.east) - (0,0.6mm)$);
    \draw[->] ($(m-2-1.east) + (0,0.6mm)$) -- node[above,scale=0.75]{$\inc$} ($(m-2-2.west) + (0,0.6mm)$);
    \draw[->] ($(m-2-2.west) - (0,0.6mm)$) -- ($(m-2-1.east) - (0,0.6mm)$);
    \draw[->] ($(m-2-2.east) + (0,0.6mm)$) -- ($(m-2-3.west) + (0,0.6mm)$);
    \draw[->] ($(m-2-3.west) - (0,0.6mm)$) -- node[below,scale=0.75]{$\inc$} ($(m-2-2.east) - (0,0.6mm)$);
    \draw[->] (m-2-1) -- node[left,scale=0.75]{$\cong$} (m-1-1);
    \draw[->] (m-2-2) -- node[left,scale=0.75]{$\cong$} (m-1-2);
    \draw[->] (m-2-3) -- node[left,scale=0.75]{$\cong$} (m-1-3);
\end{tikzpicture}\end{equation*}
By Proposition \ref{prop_localization} the kernel of $\Ho(\calM_1)\to\Ho(\calM_1/\calM_2)$ equals the essential image of
$\calF_2/\inj\to\Ho(\calM_1)$, i.e. the essential image of $\bfR\id:\Ho(\calM_2)\to\Ho(\calM_1)$. It
remains to be shown that the derived functors $\bfR\id:\Ho(\calM_2)\to\Ho(\calM_1)$ and $\bfL\id:
\Ho(\calM_1/\calM_2)\to\Ho(\calM_1)$ are fully faithful, which follows from the commutativity of the diagram and the
fully faithfulness of $\calF_2/\inj\to\calF_1/\inj$ and $\calF_1\cap\calW_2/\inj\to\calF_1/\inj$.
\end{proof}

Dually, we have the following localization result for projective model structures:

\begin{prop}\label{prop_localizationdual}
Let $\scA$ be a bicomplete abelian category and $\calM_1 = (\calC_1,\calW_1)$ and $\calM_2 = (\calC_2,\calW_2)$ be
projective, abelian model structures on $\scA$ with $\calC_2\subset\calC_1$. Then there exists
a hereditary abelian model structure on $\scA$, called the \textit{left localization} of $\calM_1$ with respect to $\calM_2$
and denoted $\calM_2\backslash\calM_1$, with $\calC = \calC_1$, $\calF = \calW_2$ and
\begin{align*}
\calW & \makebox[8mm]{$:=$} \{X\in\scA\ |\ \exists \text{ ex. seq. } 0\to X\to A\to B\to 0\text{ with }A\in\calW_1,
B\in\calC_2\}\\
& \makebox[8mm]{$=$} \{X\in\scA\ |\ \exists \text{ ex. seq. } 0\to A\to B\to X\to 0\text{ with }A\in\calW_1, B\in\calC_2\}.
\end{align*}
Moreover, $X\in\calW$ if and only if it belongs to the essential image of $\calC_2\to\Ho(\calM_1)$, and there is a
localization sequence of triangulated categories  
$$\Ho(\calM_2)\xrightarrow{\ \ \bfL\id\ \ }\Ho(\calM_1)\xrightarrow{\ \ \bfL\id\ \ }\Ho(\calM_2\backslash\calM_1).$$
\end{prop}

\begin{ex}
We consider a simple example, anticipating the more general results that will be discussed later in Sgection
\ref{section_relsing}. Let $R$ be a ring considered as a dg ring concentrated in cohomological degree zero. From Propositions
\ref{prop_projectiveinjectivemodel} and \ref{prop_cocontraderivedmodel} we get the standard projective 
model structure $({^{\perp}}\Acyc(R),\Acyc(R),\Ch(R))$ and the contraderived model structure
$(\Ch(\Proj(R)),\calW\uctr(R),\Ch(R))$ on $\Ch(R)$. Since $\calC\uproj(R)\subseteq\calC\uctr(R)$ by
Corollary \ref{cor_comparison}, we can apply Proposition \ref{prop_localizationdual} and get as the left
localization $\calM\uproj(R)\backslash\calM\uctr(R)$ the model structure $(\Ch(\Proj(R)),?,\Acyc(R))$ on
$\Ch(R)$, with homotopy category $\bfK\lac(\Proj(R))$. Similarly, applying Proposition \ref{prop_localization}
we can form the right localization $\calM\uco(R)/\calM\uinj(R)$, i.e. the abelian model structure corresponding
to the triple $(\Acyc(R),?,\Ch(\Inj(R)))$, with homotopy category $\bfK\lac(\Inj(R))$. In particular, we
deduce that there is a colocalization sequence $\bfK\lac(\Inj(R))\to\bfK(\Inj(R))\to\bfD(R)$ and a localization sequence
$\bfK\lac(\Proj(R))\to\bfK(\Proj(R))\to\bfD(R)$. 
\end{ex}

\label{subsection_localization}
\subsection{Right Bousfield localization}\label{subsection_bousfield}

In this section, we again go back to the classical language of model categories and rewrite Proposition
\ref{prop_localization} as a statement about existence of certain right Bousfield localizations. The results of this
section are not needed anywhere else and are included solely for the purpose of connecting and making explicit
well-established notions and results on model categories in the case of abelian model categories.

\begin{definition}[\protect{\cite[Definition 3.3.1(2)]{Hirschhorn_LocalizationOfModelCategories}}]
Let $\calM$ be a model category and $\locM$ be a class of maps in $\calM$. The \textit{right Bousfield localization} of
$\calM$ with respect to $\locM$ is, if it exists, the model structure $\bfR_{\locM}\calM$ on the category underlying
$\calM$ such that
\begin{enumerate}
\item the class of weak equivalences of $\bfR_{\locM}\calM$ is the class of $\locM$-colocal equivalences,
\item the class of fibrations of $\bfR_{\locM}\calM$ is the class of fibrations of $\calM$, and
\item the class of cofibrations of $\bfR_{\locM}\calM$ is determined by the left lifting property with respect to trivial
  fibrations. 
\end{enumerate}
\end{definition}

\begin{definition}
Let $\calM$ be a model category, $K$ a class of objects and $\locM$ a class of morphisms in $\calM$. 
\begin{enumerate}
\item A morphism $f:A\to B$ is called a \textit{$K$-colocal equivalence} if for all $X\in K$ and $k\geq 0$ the induced
  map $\Ho(\calM)(X,\Omega^k A)\to\Ho(\calM)(X,\Omega^k B)$ is a bijection.
\item An object $X\in\calM$ is called \textit{$\locM$-colocal} if for all $f: A\to B$ in $\locM$ and $k\geq 0$ the
  induced map $\Ho(\calM)(X,\Omega^k A)\to\Ho(\calM)(X,\Omega^k B)$ is a bijection.
\item A morphism is called a \textit{$\locM$-colocal equivalence} if it is a colocal equivalence with respect to the
  class of $\locM$-colocal objects.
\end{enumerate}
\end{definition}

\begin{prop}\label{prop_bousfield}
Let $\scA$ be a bicomplete abelian category and $\calM_1= (\calW_1,\calF_1)$ and $\calM_2= (\calW_2,\calF_2)$ be
 injective model structures on $\scA$ satisfying $\calF_2\subset\calF_1$. Then the model structure
$\calM_1/\calM_2$ described in Theorem \ref{prop_localization} is the right Bousfield localization of $\calM_1$ with
respect to $\locM := \{0\to X\ |\ X\in\calF_2\}\subset\Mor(\scA)$. 
\end{prop}
\begin{proof}
Since domain and codomain of each morphism in $\locM$ are fibrant in $\calM_1$, Proposition
\ref{prop_descriptionhomotopycategory} reveals that the class of $\locM$-colocal objects equals
${^{\perp}}(\calF_2/\inj)$ in $\scA/\inj$, which is $\calW_2/\inj$ by Proposition \ref{prop_cotorsiontorsion} applied to the cotorsion
pair $(\calW_2,\calF_2)$.

It remains to show that the weak equivalences in $\calM_1/\calM_2$ are precisely the $\calW_2$-colocal equivalences. For
this, note the following: 
\begin{enumerate}
\item In $\Ho(\calM_1)$ any morphism is isomorphic to a morphism between objects in
$\calF_1$: This follows from the fact that in $\Ho(\calM_1)$ any object is isomorphic to an object in $\calF_1$ (see
Proposition \ref{prop_descriptionhomotopycategory}).
\item In $\Ho(\calM_1)$, any morphism between objects in $\calF_1$ is isomorphic to an epimorphism between objects in
  $\calF_1$ with kernel again in $\calF_1$: If $f: A\to B$ is (a representative of) the given morphism with
  $A,B\in\calF_1$, and $0\to B\p\to I\stackrel{p}{\to} B\to 0$ is exact with $I\in\inj$ and $B\p\in\calF_1$, then $f$ is isomorphic in
  $\Ho(\calM_1)$ to $(f,-p): A\oplus I\to B$. Moreover, $K := \ker(f,-p)\in\calF_1$ since it fits into the commutative diagram with exact rows 
\begin{equation*}\begin{tikzpicture}[description/.style={fill=white,inner sep=2pt}]
    \matrix (m) [matrix of math nodes, row sep=1.5em,
                 column sep=1.5em, text height=1.5ex, text depth=0.25ex,
                 inner sep=0pt, nodes={inner xsep=0.3333em, inner ysep=0.3333em}]
    {
       0 & B\p & K & A & 0\\
       0 & B\p & I & B & 0\\
    };
    \draw[->] (m-1-1) -- (m-1-2);
    \draw[->] (m-1-2) -- (m-1-3);
    \draw[->] (m-1-3) -- (m-1-4);
    \draw[->] (m-1-4) -- (m-1-5);
    \draw[->] (m-2-1) -- (m-2-2);
    \draw[->] (m-2-2) -- (m-2-3);
    \draw[->] (m-2-3) -- (m-2-4);
    \draw[->] (m-2-4) -- (m-2-5);
    \draw[double distance=0.5mm] (m-1-2) -- (m-2-2);
    \draw[->] (m-1-3) -- (m-2-3);
    \draw[->] (m-1-4) -- (m-2-4);
\end{tikzpicture}\end{equation*}
and $\calF_1$ is closed under extensions.
\item If $f: A\to B$ is an epimorphism of objects in $\calF_1$ and kernel $K\in\calF_1$ as in (2), then $X\in\scA$
  is $f$-colocal if and only if $(\scA/\inj)(X,\Omega^k K)=0$ for all $k\geq 0$: To begin, the short exact sequence $0\to K\to
  A\to B\to 0$ gives rise to a triangle in $\Ho(\calM)$. Now the functor $\Ho(\calM)(X,-)$ is cohomological, i.e. turns
  exact triangles into long exact sequences, and hence $\Ho(\calM)(X,\Omega^k(f))$ is bijective for all $k\geq 
  0$ if and only if $\Ho(\calM)(X,\Omega^k K)=0$ for all $k\geq 0$. By Proposition
  \ref{prop_descriptionhomotopycategory} the latter is equivalent to $(\scA/\inj)(X,\Omega^k K)=0$ for all $k\geq 0$.
\end{enumerate}
As $(\calW_2/\inj)^{\perp}=\calF_2/\inj$ in $\scA/\inj$, steps $(1)-(3)$ show that the $\calW_2$-colocal equivalences are
precisely those morphisms which are isomorphic in $\Ho(\calM_1)$ to epimorphism of objects in $\calF_1$ with kernel in
$\calF_2$.  

We will show that the same description applies to the weak equivalences in $\calM_1/\calM_2$. By Fact
\ref{fact_weakequivalences}, any weak equivalence in $\calM_1/\calM_2$ is the composition of a monomorphism with
cokernel in $\calC\cap\calW={^{\perp}}\calF_1 = \calW_1$ and an epimorphism with kernel in $\calW\cap\calF =
\calW_2^{\perp} = \calF_2$. The former is already a weak equivalence in $\calM_1$, hence any weak equivalence in
$\calM_1/\calM_2$ is isomorphic to an epimorphism with kernel in $\calF_2$ in $\Ho(\calM_1)$. Let $f: B\to A$ be such an
epimorphism and pick a short exact sequence $0\to B\stackrel{\alpha}{\to} F\to W\to 0$ with $F\in\calF_1$. Taking the pushout of
$F\stackrel{\alpha}{\leftarrow} B\stackrel{f}{\to} A$, we get the following commutative diagram (note that the right square is also pullback):
\begin{equation*}\begin{tikzpicture}[description/.style={fill=white,inner sep=2pt}]
    \matrix (m) [matrix of math nodes, row sep=1.5em,
                 column sep=1.5em, text height=1.5ex, text depth=0.25ex,
                 inner sep=0pt, nodes={inner xsep=0.3333em, inner ysep=0.3333em}]
    {
&& 0 & 0\\
       0 & K & B & A & 0\\
       0 & K & F & F\p & 0\\
 && W & W &\\
&& 0 & 0\\
    };
    \draw[double distance=0.8mm] (m-2-2) -- (m-3-2);
    \draw[double distance=0.8mm] (m-4-3) -- (m-4-4);
    \draw[->] (m-2-3) -- node[description,scale=0.75]{$\alpha$} (m-3-3);
    \draw[->] (m-2-4) -- node[description,scale=0.75]{$\beta$} (m-3-4);
    \draw[->] (m-4-3) -- (m-5-3);
    \draw[->] (m-4-4) -- (m-5-4);
    \draw[->] (m-2-1) -- (m-2-2);
    \draw[->] (m-2-2) -- (m-2-3);
    \draw[->] (m-2-3) -- node[scale=0.75,above]{$f$} (m-2-4);
    \draw[->] (m-2-4) -- (m-2-5);
    \draw[->] (m-3-1) -- (m-3-2);
    \draw[->] (m-3-2) -- (m-3-3);
    \draw[->] (m-3-3) -- node[scale=0.75,below]{$g$} (m-3-4);
    \draw[->] (m-3-4) -- (m-3-5);
    \draw[->] (m-1-3) -- (m-2-3);
    \draw[->] (m-3-3) -- (m-4-3);
    \draw[->] (m-1-4) -- (m-2-4);
    \draw[->] (m-3-4) -- (m-4-4);
\end{tikzpicture}\end{equation*}
As $\alpha,\beta$ are weak equivalences in $\calM_1$, $f$ is isomorphic to $g$ in $\Ho(\calM_1)$. Moreover, as $\calF_1$
is closed under cokernels of monomorphisms, $F\p\in\calF_1$. This shows that $f$ is isomorphic in $\Ho(\calM_1)$ to an
epimorphism of objects in $\calF_1$ with kernel in $\calF_2$. Conversely, since any weak equivalence in $\calM_1$ is
also a weak equivalence in $\calM_1/\calM_2$, it is clear that any such morphism is a weak equivalence in
$\calM_1/\calM_2$. 
\end{proof}

\section{The singular model structures}\label{section_relsing}

In this section we attach to each morphism of dg rings $\varphi: R\to A$ two ``relative singular'' model structures on
$A\Mod$, a contraderived and a coderived one. Roughly, the contraderived (resp. coderived) singular model structure
is obtained by pulling back the contraderived (resp. coderived) model $\calM\uctr(R)$  (resp. $\calM\uco(R)$) on $R\Mod$
to $A\Mod$ along the right (resp. left) adjoint $U_\varphi: A\Mod\to R\Mod$, and afterwards taking the left
(resp. right) localization of $\calM\uctr(A)$ (resp. $\calM\uco(A)$) with respect to this pullback model structure. If
$R$ is an ordinary ring of finite left-global dimension, we will see that the relative singular contraderived and coderived model
structures only depend on $A$, and we will call them the ``absolute singular'' model structures attached to $A$.

In general, pulling back model structures along adjoints is a nontrivial problem, so we need to justify that the above
pullbacks are again abelian model structures. In our situation, the connection between abelian
model structures and deconstructible classes makes this problem tractable and we give ad-hoc arguments to establish the
desired pullbacks.   

Recall that right (resp. left) localization of two projective (resp. injective) model structures produces abelian model
structures which are neither projective nor injective in general. In particular, the (relative or absolute) singular
model structures are neither projective nor injective. We will be able, however, to establish a concrete projective
(resp. injective) abelian model structure on $A\Mod$ Quillen equivalent to the singular contraderived (resp. coderived)
one. This alternative description is useful for example in proving that the absolute contraderived 
(resp. coderived) singular model structure on $\Ch(R)$, for $R$ Gorenstein, is Quillen equivalent to Hovey's
Gorenstein projective (resp. Gorenstein injective) model structure on $R\Mod$, as well as in the construction of recollements later.

\subsection{General definitions}

Let $U:\calD\longrightarrow\calC$ be a functor between two categories $\calC,\calD$, and suppose that $\calC$
carries a model structure $\calM$. The \textit{right pullback} of $\calM$ along $U$ is, if it exists, the
model structure on $\calD$ in which a morphism is a weak equivalence (resp. fibration) if and only if its image under
$U$ is a weak equivalence (resp. fibration) in $\calM$, and where the cofibrations are determined by the left lifting
property with respect to all trivial fibrations. Similarly, the \textit{left pullback} of $\calM$ along $U$ is, if it
exists, the model structure on $\calD$ where the cofibrations (resp. weak equivalences) are the morphisms which become
cofibrations (resp. weak equivalences) in $\calM$ after application of $U$, and where the fibrations are determined by
the right lifting property with respect to all trivial cofibrations. 

\begin{prop}\label{prop_pullback}
Let $\varphi: R\to A$ be a morphism of dg rings.
\begin{enumerate}
\item The right-pullback $\varphi\ua\calM\uctr(R)$ of $\calM\uctr(R)$ along $U_\varphi$ exists. 
\item The left-pullback $\varphi\ua\calM\uco(R)$ of $\calM\uco(R)$ along $U_\varphi$ exists.
\end{enumerate}
Moreover, both $\varphi\ua\calM\uctr(R)$ and $\varphi\ua\calM\uco(R)$ are cofibrantly generated.
\end{prop}
\begin{proof}
(1) It suffices to show that firstly $U\ua_\varphi(\calW\uctr(R))$ is of the form $\calS^{\perp}$ for a set $\calS\subset
  A\Mod$, and secondly that $U\ua_\varphi(\calW\uctr(R))\cap{^{\perp}}U\ua_\varphi(\calW\uctr(R))=\proj(A\Mod)$. By Proposition
  \ref{prop_cocontraderivedmodel} $\calC\uctr(R)$ is deconstructible, so we may choose a set $\calT$ such that
  $\calC\uctr(R)=\filt\calT$. Denoting the left adjoint $A\otimes_R -$ to  
  $U_\varphi$ by $F$ for a moment, we claim that $U\ua_\varphi(\calW\uctr(R))=F(\calT)^{\perp}$.
  In fact, we will even show that $\ext^1_{A}(F(T),-)\cong\ext^1_R(T,U_\varphi(-))$ for all
  $T\in\calT$. Having done this, the claim follows via
  $F(\calT)^{\perp}=U_\varphi\ua(\calT^{\perp})=U_\varphi\ua(\calW\uctr(R))$. 
  Let $Y\in A\Mod$ be arbitrary and $0\to Y\to W\stackrel{f}{\to}C\to 0$ be an exact sequence with
  $W\in\calW\uctr(A)$ and $C\in\calC\uctr(A)$. Since $F(\calT)\subseteq\calC\uctr(A)$ (Proposition
  \ref{prop_fouradjunctions}), we get $\ext^1_{A}(F(T),Y)\cong\coker\left[\Hom_A(F(T),f)\right]$. Moreover, since
  $U_\varphi$ is exact and $U_\varphi(\calW\uctr(A))\subseteq\calW\uctr(R)$ (Proposition \ref{prop_fouradjunctions}), computing
  $\ext^1_{A}(T,U_\varphi(Y))$ using the exact sequence $0\to U_\varphi(Y)\to
  U_\varphi(W)\stackrel{U_\varphi(f)}{\longrightarrow} U_\varphi(C)\to 0$ gives
  $\ext^1_R(T,U_\varphi(Y))\cong\coker\left[\Hom_R(T,U_\varphi(f))\right]$. Now, the adjunction $F\dashv U_\varphi$ gives
  $\coker\left[\Hom_R(T,U_\varphi(f))\right]\cong\coker\left[\Hom_A(F(T),f)\right]$, and hence
  $\ext^1_{A}(F(T),Y)\cong\ext^1_R(T,U_\varphi(Y))$ for all $T\in\calT$ and $Y\in A\Mod$.  The remaining part
  $U\ua_\varphi(\calW\uctr(R))\cap{^{\perp}}U\ua_\varphi(\calW\uctr(R))=\proj(A\Mod)$ follows from 
  Lemma \ref{lem_doubleorthogonal} since $\calW\uctr(A)\subseteq U\ua_\varphi(\calW\uctr(R))$ and hence
  ${^{\perp}}U\ua_\varphi(\calW\uctr(R))\subseteq\calC\uctr(A)=A\Mod\lproj$. 

(2) We have to show that $\calK := U_\varphi\ua(\calW\uco(R))$ is deconstructible and
  $\calK\cap\calK^{\perp}=\inj(A\Mod)$. The deconstructibility of $\calK$ follows from Proposition
  \ref{prop_corpullbackmonadic} together with the deconstructibility of $\calW\uco(R)$ established in Proposition
  \ref{prop_cocontraderivedmodel}. Hence $(\calK,\calK^{\perp})$ is a complete cotorsion pair cogenerated by a set. For
  $\calK\cap\calK^{\perp}=\inj(A\Mod)$, first note that since $U_\varphi:\calM\uco(A)\to\calM\uco(R)$ is left Quillen
  (Proposition \ref{prop_fouradjunctions}), we have $\calK\supseteq\calW\uco(A)$, and hence
  $\calK^{\perp}\subseteq\calF\uco(A)=A\Mod\linj$. Applying Lemma \ref{lem_doubleorthogonal} now gives
  $\calK\cap\calK^{\perp}=\inj(A\Mod)$ as required. 
\end{proof}

Note that if $R$ is an ordinary ring of finite left-global dimension, then $\calM\uctr(R)=\calM\uproj(R)$ and
$\calM\uco(R)=\calM\uinj(R)$ (Proposition \ref{prop_finglobdim}), and hence for any morphism $\varphi: R\to A$ of
dg rings $\varphi\ua\calM\uctr(R) = \calM\uproj(A)$ and  $\varphi\ua\calM\uco(R)=\calM\uinj(A)$.

\begin{definition}\label{def_relativesingmodel}
Let $\varphi: R\to A$ be a morphism of dg rings. 
\begin{enumerate}
\item The \textit{relative singular coderived model structure} on $A\Mod$ is defined as the right localization
  $\calM\uco(A)/\varphi\ua\calM\uco(R)$ in the sense of Proposition \ref{prop_localization} and denoted
  $\calM\uco\lsing(A/R)$. 
\item The \textit{relative singular contraderived  model structure} on $A\Mod$ is defined as the left localization
  $\varphi\ua\calM\uctr(R)\backslash\calM\uctr(A)$ in the sense of Proposition \ref{prop_localizationdual} and denoted
  $\calM\uctr\lsing(A/R)$. 
\end{enumerate}
If $R$ is a ring of finite left-global dimension (e.g. $R=\ZZ$ or $R=k$ is a field), then $\calM^{\ctr/\co}\lsing(A) := 
\calM^{\ctr/\co}\lsing(A/R)$ does not depend on $R$ and is called the \textit{absolute singular contraderived
  resp. coderived model structure}.  
\end{definition}

\begin{prop}\label{prop_abssingmodelcontra}
Let $\varphi: R\to A$ be a morphism of dg rings. The relative singular contraderived model structure
$\calM\uctr\lsing(A/R)$ can be described as follows: 
\begin{enumerate}
\item[--] The class $\calC$ of cofibrant objects equals $A\Mod\lproj$. 
\item[--] The class $\calF$ of fibrant objects is the class of $A$-modules whose underlying $R$-modules are contraacyclic.
\item[--] The class $\calW$ of weakly trivial objects is determined by Fact
  \ref{fact_weakequivalences}. 
\end{enumerate}
In particular, the fibrant objects in $\calM\uctr\lsing(A)$ are the acyclic $A$-modules.
\end{prop}

A similar description holds for the relative singular coderived model:

\begin{prop}\label{prop_abssingmodelco}
Let $\varphi: R\to A$ be a morphism of dg rings. The relative singular coderived model structure
$\calM\uco\lsing(A/R)$ can be described as follows: 
\begin{enumerate}
\item[--] The class $\calC$ of cofibrant objects is the class of $A$-modules whose underlying $R$-modules are coacyclic.
\item[--] The class $\calF$ of fibrant objects equals $A\Mod\linj$. 
\item[--] The class $\calW$ of weakly trivial objects is determined by Fact
  \ref{fact_weakequivalences}. 
\end{enumerate}
In particular, the cofibrant objects in $\calM\uctr\lsing(A)$ are the acyclic $A$-modules.
\end{prop}
\label{subsection_gendefs}

\begin{rem}\label{rem_gen2}
The construction of the relative and absolute singular coderived model structures carries over to the setting discussed
in Remark \ref{rem_gen}.
\end{rem}

\subsection{Constructing recollements}

From Proposition \ref{prop_abssingmodelcontra} (resp. \ref{prop_abssingmodelco}) it is clear that $\calM\lsing\uctr(A)$
(resp. $\calM\lsing\uco(A)$) is almost never projective (resp. injective). However, there is a canonical projective
(resp. injective) abelian model structure which is Quillen equivalent to the absolute singular contraderived
(resp. coderived) model, which we describe in this section. 

\begin{prop}\label{prop_singmodelcoctr}
For a dg ring $A$, the following hold:
\begin{enumerate}
\item There exists a projective abelian model structure ${^{p}}\calM\uctr\lsing(A)$ on $A\Mod$ satisfying $\calC =
  A\Mod\lproj\cap\Acyc(A)$. 
\item There exists an injective abelian model structure ${^{i}}\calM\uco\lsing(A)$ on $A\Mod$ satisfying $\calF =
A\Mod\linj\cap\Acyc(A)$. 
\end{enumerate}
${^{p}}\calM\uctr\lsing(A)$ and ${^{i}}\calM\uco\lsing(A)$ are cofibrantly generated and the identity is
a left resp. right Quillen equivalence ${^{p}}\calM\uctr\lsing(A)\to\calM\uctr\lsing(A)$
resp. ${^{i}}\calM\uco\lsing(A)\to\calM\uco\lsing(A)$.    
\end{prop}
\begin{proof}
(1) As usual it suffices to that ${^{p}}\calC\uctr\lsing(A)=A\Mod\lproj\cap\Acyc(A)$ is deconstructible,
 ${^{p}}\calC\uctr\lsing(A)\cap{^{p}}\calC\uctr\lsing(A)^{\perp}=\proj(A\Mod)$ and that
${^{p}}\calC\uctr\lsing(A)^{\perp}$ has the $2$-out-of-$3$ property. Since both $A\Mod\lproj$ and $\Acyc(A)$
are deconstructible by Propositions \ref{prop_cocontraderivedmodel} and \ref{prop_projectiveinjectivemodel}, the deconstructibility
of $A\Mod\lproj\cap\Acyc(A)$ follows from the stability of deconstructible classes under intersections \cite[Proposition
2.9]{StovicekHillGrothendieck}. The equality ${^{p}}\calC\uctr\lsing(A)\cap{^{p}}\calC\uctr\lsing(A)^{\perp}=\proj(A\Mod)$ follows
from Lemma \ref{lem_doubleorthogonal}, and Lemma \ref{lem_twooutofthree} ensures the $2$-out-of-$3$ property since
${^{p}}\calC\uctr\lsing(A)$ is closed under kernels of epimorphisms. Finally, it is clear that the identity is a left Quillen functor
${^{p}}\calM\uctr\lsing(A)\to\calM\uctr\lsing(A)$; moreover, Proposition \ref{prop_descriptionhomotopycategory} implies that
it induces an equivalence on homotopy categories, hence is a Quillen equivalence.

(2) Note that ${^{i}}\calM\uco\lsing(A)=A\Mod\linj\cap\Acyc(A)$ is of the form $\calS^{\perp}$ for some set
$\calS$ as this is true both for $A\Mod\linj$ (Proposition \ref{prop_cocontraderivedmodel}) and $\Acyc(A)$
(Proposition \ref{prop_projectiveinjectivemodel}). Hence $\left({^{\perp}}({^{i}}\calM\uco\lsing(A)),
  {^{i}}\calM\uco\lsing(A)\right)$ is a complete cotorsion pair. By Lemma \ref{lem_doubleorthogonal}, we have
${^{i}}\calM\uco\lsing(A)\cap{^{\perp}}({^{i}}\calM\uco\lsing(A))=A\Mod\linj$, and Lemma \ref{lem_twooutofthree} again
provides the $2$-out-of-$3$ property since ${^{i}}\calM\uco\lsing(A)$ is closed under cokernels of monomorphisms. That
the identity is a right Quillen equivalence ${^{i}}\calM\uco\lsing(A)\to\calM\uco\lsing(A)$ again follows from Proposition
\ref{prop_descriptionhomotopycategory}. 
\end{proof}
We do not expect a variant of Proposition \ref{prop_singmodelcoctr} to hold for the relative singular models attached to
a morphism $\varphi: R\to A$ since we see no reason for $\calW\uctr(R)$ and $U_\varphi\ua\calW\uctr(R)$ to be
deconstructible (resp. for $\calW\uco(R)$ and $U_\varphi\ua\calW\uco(R)$ to be of the form $\calS^{\perp}$ for a set of
objects $\calS$). For the absolute singular models, this is different, because luckily $\Acyc(A)$ arises both  as the
cotorsionfree class in $(\calC\uproj(A),\Acyc(A))$ and as the cotorsion class in $(\Acyc(A),\calF\uinj(A))$. 

Let us pause for a moment to see what model structures are currently around, restricting to the injective case. We
started with the identity right Quillen functor $\calM\uinj(A)\to\calM\uco(A)$ and applied Proposition
\ref{prop_localization} to get the right localization $\calM\uco\lsing(A):=\calM\uinj(A)/\calM\uco(A)$, fitting into a
colocalization sequence $\Ho(\calM\uinj(A))\to\Ho(\calM\uco(A))\to\Ho(\calM\uco\lsing(A))$. Now, however, we have also
constructed a model ${^{i}}\calM\uco\lsing(A)$ for which the identity is \textit{right} Quillen
${^{i}}\calM\uco\lsing(A)\to\calM\uco(A)$, and on the level of homotopy categories we have the following commutative diagram:
\begin{equation*}\begin{tikzpicture}[description/.style={fill=white,inner sep=2pt}]
    \matrix (m) [matrix of math nodes, row sep=3em,
                 column sep=2.5em, text height=1.5ex, text depth=0.25ex,
                 inner sep=0pt, nodes={inner xsep=0.3333em, inner ysep=0.3333em}]
    {
      \Ho(\calM\uco\lsing(A)) &&& \Ho(\calM\uco(A))\\
      & \bfK\lac(A\Mod\linj) &\bfK(A\Mod\linj) & \\
      \Ho({^{i}}\calM\uco\lsing(A)) &&& \Ho(\calM\uco(A))\\
    };
    \draw[->] (m-1-1) -- node[above,scale=0.75]{$\bfL\id$} (m-1-4);
    \draw[->] (m-3-1) -- node[below,scale=0.75]{$\bfR\id$} (m-3-4);
    \draw[->] ($(m-1-1.south) + (1mm,0)$) -- node[right,scale=0.75]{$\bfL\id$} ($(m-3-1.north) + (1mm,0)$);
    \draw[->] ($(m-3-1.north) - (1mm,0)$) -- node[left,scale=0.75]{$\bfR\id$} ($(m-1-1.south) - (1mm,0)$);
    \draw[double, double distance=0.8mm] (m-1-4) -- (m-3-4);
    \draw[->] (m-2-2) -- node[above,scale=0.75]{$\inc$} (m-2-3);
    \draw[->] ($(m-2-2.north west) + (2mm,0)$) -- node[above right,scale=0.75]{$\cong$} ($(m-1-1.south east) - (2mm,0)$);
    \draw[->] ($(m-2-2.south west) + (2mm,0)$) -- node[below right,scale=0.75]{$\cong$} ($(m-3-1.north east) - (2mm,0)$);
    \draw[->] ($(m-2-3.north east) - (2mm,0)$) -- node[above left,scale=0.75]{$\cong$} ($(m-1-4.south west) + (2mm,0)$);
    \draw[->] ($(m-2-3.south east) - (2mm,0)$) -- node[below left,scale=0.75]{$\cong$} ($(m-3-4.north west) + (2mm,0)$);
\end{tikzpicture}\end{equation*}
Note that the diagonal functors are equivalences since they are the canonical functors from the
homotopy category of cofibrant and fibrant objects into the homotopy category. From this diagram we see that
$\bfL\id:\calM\uco\lsing(A)\to\calM\uco(A)$ and $\bfR\id:{^{i}}\calM\uco\lsing(A)\to\calM\uco(A)$ are equivalent, and
hence $\bfL\id:\calM\uco\lsing(A)\to\calM\uco(A)$ has a \textit{left} adjoint while
$\bfR\id:{^{i}}\calM\uco\lsing(A)\to\calM\uco(A)$ has a \textit{right} adjoint. Thus:
\begin{cor}\label{cor_injrec}
For any dg ring $A$, there is a recollement
\begin{align*}
\bfK\lac(A\Mod\linj)\bigrec\bfK(A\Mod\linj)\bigrec\bfD(A).
\end{align*}
\end{cor}
\begin{proof}
$\bfK\lac(A\Mod\linj)\to\bfK(A\Mod\linj)\to\bfD(A)$ is a colocalization sequence by Corollary \ref{cor_loc}, and by the above
$\bfK\lac(A\Mod\linj)\to\bfK(A\Mod\linj)$ also has a left adjoint. This is all we need for a recollement.
\end{proof}
In case $A$ is a Noetherian ring (considered as a dg ring concentrated in degree $0$) the recollement from Corollary
\ref{cor_injrec} was constructed by Krause \cite[Corollary 4.3]{Krause_StableDerived} in the more general framework of
complexes over a locally Noetherian Grothendieck category with compactly generated derived category. 

Dually, in the
projective/contraderived situation we have the following recollement, which again is already known for ordinary rings by
\cite[Theorem 5.15]{Murfet_MockHomotopyCategoryOfProjectives}: 
\begin{cor}\label{cor_projrec}
For any dg ring $A$, there is a recollement
\begin{align*}
\bfK\lac(A\Mod\lproj)\bigrec\bfK(A\Mod\lproj)\bigrec\bfD(A).
\end{align*}
\end{cor}

Back in the injective situation we also want to give a model categorical construction of the left adjoint of
$\bfK(A\Mod\linj)\to\bfD(A)$. For this, note that the injective version ${^{i}}\calM\uco\lsing(A)$ of the singular
coderived model structure has ${^{i}}\calF\uco\lsing(A)\subseteq\calF\uco(A)$; we can therefore apply Proposition
\ref{prop_localization} to form the right localization
${^{m}}\calM\uinj(A):=\calM\uco(A)/{^{i}}\calM\uco\lsing(A)$. This is the abelian model structure 
determined by ${^{m}}\calC\uinj(A)={^{\perp}}\left(\Acyc(A)\cap A\Mod\linj\right)$ and
${^{m}}\calF\uinj(A)=A\Mod\linj$, and the identity is a left Quillen functor ${^{m}}\calM\uinj(A)\to\calM\uinj(A)$. All
in all, we get the following butterfly of abelian model structures and Quillen functors on $A\Mod$, where $L$ denotes 
left Quillen functors and $R$ denotes right Quillen functors:
\begin{align}\label{eq:butterfly}\tag{$\infty$}
\begin{tikzpicture}[baseline,description/.style={fill=white,inner sep=2pt}]
    \matrix (m) [matrix of math nodes, row sep=1.8em,
                 column sep=4em, text height=1.5ex, text depth=0.25ex,
                 inner sep=0pt, nodes={inner xsep=0.3333em, inner ysep=0.3333em}]
    {
       \calM\uco\lsing(A) \pgfmatrixnextcell\pgfmatrixnextcell \calM\uinj(A)\\
       \pgfmatrixnextcell \calM\uco(A) \pgfmatrixnextcell \\
       {^{i}}\calM\uco\lsing(A) \pgfmatrixnextcell\pgfmatrixnextcell {^{m}}\calM\uinj(A)\\
    };
    \draw[->] ($(m-1-1.south) + (1mm,0)$) -- node[scale=0.75,right]{L} ($(m-3-1.north) + (1mm,0)$);
    \draw[->] ($(m-3-1.north) - (+1mm,0)$) -- node[scale=0.75,left]{R} ($(m-1-1.south) - (1mm,0)$);
    \draw[->] ($(m-1-3.south) + (1mm,0)$) -- node[scale=0.75,right]{R} ($(m-3-3.north) + (1mm,0)$);
    \draw[->] ($(m-3-3.north) + (-1mm,0)$) -- node[scale=0.75,left]{L} ($(m-1-3.south) - (1mm,0)$);
    \draw[->] ($(m-1-1.south east) - (4mm,0)$) -- node[scale=0.75,below]{L} ($(m-2-2.north west) + (0mm,0)$);
    \draw[->] ($(m-2-2.north west) + (4mm,0)$) -- node[scale=0.75,above]{R} ($(m-1-1.south east) - (0mm,0)$);
    \draw[->] ($(m-3-1.north east) + (0mm,0)$) -- node[scale=0.75,below]{R} ($(m-2-2.south west) + (4mm,0)$);
    \draw[->] ($(m-2-2.south west) + (0mm,0)$) -- node[scale=0.75,above]{L} ($(m-3-1.north east) - (4mm,0)$);
    \draw[->] ($(m-2-2.north east) + (0mm,0)$) -- node[scale=0.75,below]{L} ($(m-1-3.south west) + (4mm,0)$);
    \draw[->] ($(m-1-3.south west) + (0mm,0)$) -- node[scale=0.75,above]{R} ($(m-2-2.north east) - (4mm,0)$);
    \draw[->] ($(m-2-2.south east) - (4mm,0)$) -- node[scale=0.75,below]{R} ($(m-3-3.north west) + (0mm,0)$);
    \draw[->] ($(m-3-3.north west) + (4mm,0)$) -- node[scale=0.75,above]{L} ($(m-2-2.south east) + (0mm,0)$);
\end{tikzpicture}
\end{align}
The properties of this diagram are summarized in the following proposition:
\begin{prop}\label{prop_butterfly}
Let $A$ be a dg ring and consider the butterfly \eqref{eq:butterfly}.
\begin{enumerate}
\item\label{eq:butterflyrows}  $\calM\uinj(A)\to\calM\uco(A)\to\calM\uco\lsing(A)$ and
  ${^{i}}\calM\uco\lsing(A)\to\calM\uco(A)\to{^{m}}\calM\uinj(A)$ are right localizations in the sense of Proposition \ref{prop_localization}.
\item\label{eq:butterflyquill} $\calM\uco\lsing(A)\rightleftarrows{^{i}}\calM\uco\lsing(A)$ and
  ${^{m}}\calM\uinj(A)\rightleftarrows\calM\uinj(A)$ are Quillen equivalences. More precisely, the classes of
  simultaneously cofibrant and fibrant objects in $\calM\uco\lsing(A)$ and ${^{i}}\calM\uco\lsing(A)$ coincide, and
  the classes of weak equivalences in $\calM\uinj(A)$ and ${^{m}}\calM\uinj(A)$ coincide. 
\item\label{eq:butterflywings} The two wings in the following following diagram commute:
\begin{equation*}
\begin{tikzpicture}[description/.style={fill=white,inner sep=2pt}]
    \matrix (m) [matrix of math nodes, row sep=1.8em,
                 column sep=4em, text height=1.5ex, text depth=0.25ex,
                 inner sep=0pt, nodes={inner xsep=0.3333em, inner ysep=0.3333em}]
    {
       \Ho(\calM\uco\lsing(A)) \pgfmatrixnextcell\pgfmatrixnextcell \Ho(\calM\uinj(A))\\
       \pgfmatrixnextcell \calM\uco(A) \pgfmatrixnextcell \\
       \Ho\left({^{i}}\calM\uco\lsing(A)\right) \pgfmatrixnextcell\pgfmatrixnextcell \Ho\left({^{m}}\calM\uinj(A)\right)\\
    };
    \draw[->] (m-1-1) -- node[scale=0.75,description]{$\bfL\id$} (m-3-1);
    \draw[->] (m-1-3) -- node[scale=0.75,description]{$\bfR\id$} (m-3-3);
    \draw[->] (m-1-1) -- node[scale=0.75,description]{$\bfL\id$} (m-2-2);
    \draw[->] (m-2-2) -- node[scale=0.75,description]{$\bfR\id$} (m-3-3);
    \draw[->] (m-3-1) -- node[scale=0.75,description]{$\bfR\id$} (m-2-2);
    \draw[->] (m-2-2) -- node[scale=0.75,description]{$\bfL\id$} (m-1-3);
\end{tikzpicture}\end{equation*}
\end{enumerate}
\end{prop}
\begin{proof}
\eqref{eq:butterflyrows} and the part of \eqref{eq:butterflyquill} concerning
$\calM\uco\lsing(A)\rightleftarrows{^{i}}\calM\uco\lsing(A)$ hold by definition. Consider now
${^{m}}\calM\uinj(A)\rightleftarrows\calM\uinj(A)$: By Fact \ref{fact_weakequivalences} the weak  
equivalences in ${^{m}}\calM\uinj(A)$ are compositions of monomorphisms with cokernel in
${^{\perp}}\left({^{m}}\calF\uinj(A)\right)=\calW\uco(A)$ and epimorphisms with kernel in $\Acyc(A)\cap A\Mod\linj$. In
particular, any weak equivalence in ${^{m}}\calM\uinj(A)$ is a quasi-isomorphism. Conversely, suppose $f: A\to B$ is a
quasi-isomorphism and $f = g\circ h$ is a factorization of $f$ into a trivial cofibration $h: A\to C$ followed by a 
fibration $g: C\to B$, both with respect to ${^{m}}\calM\uinj(A)$. Then $h$ is a monomorphism with cokernel in
$\calW\uco(A)$, so in particular it is a quasi-isomorphism. Consequently, $g: C\to B$ is both an epimorphism with kernel
in $A\Mod\linj$ and a quasi-isomorphism, hence a trivial fibration in ${^{m}}\calM\uinj(A)$. As the composition of $g$ and $h$, we
conclude that $f$ is a weak equivalence in ${^{m}}\calM\uinj(A)$, too, as claimed. Finally, \eqref{eq:butterflywings}
follows from \eqref{eq:butterflyquill}.
\end{proof}

Proposition \ref{prop_butterfly} shows that when trying to lift a recollement $\calT\p\rec\calT\rec\calT\pp$ of
triangulated categories to the world of model categories, it is likely to happen that it unfolds to a butterfly of model
categories and Quillen functors between them. The two adjoints both for $\calT\p\to\calT$ and
$\calT\to\calT\pp$ are then explained by the presence of two different model structures for $\calT\p$ and $\calT\pp$,
compensating the fact that a functor between model categories is usually either left or right Quillen, but rarely both.

\begin{rem}\label{rem_gen3}
When trying to generalize the previous results to the setting of Remark \ref{rem_gen}, we run into a
problem: we need to know that $A\Mod\linj\cap\Acyc(A)$ is of the form $\calS^{\perp}$ for some set
of objects $\calS$. If $\scA$ has enough projectives, then $\Acyc(A)=\{\Sigma^k A\otimes P\ |\ k\in\Gamma\}^{\perp}$ for
a projective generator $P$ of $\scA$ and hence $A\Mod\linj\cap\Acyc(A)=\calS^{\perp}$ for $\calS$ being the union of a representative set
of isomorphism classes in $\{\Sigma^k A\otimes P\ |\ k\in\Gamma\}$, and $G^+(\calT)$, for a set $\calT\subset A\us\Mod$
such that $A\us\Mod=\filt\calT$. However, without existence of enough projectives, we don't know whether
$A\Mod\linj\cap\Acyc(A)$ is of the form $\calS^{\perp}$ 
for some set $\calS\subset A\Mod$. Note that since $\ext^1_A(X,Y)\cong [X,\Sigma Y]$ for $Y\in A\Mod\linj$, the problem
can also be formulated in the triangulated setting as the question whether there exists a set
$\calS\subset\bfK(A\Mod)$ such that $\bfK\lac(A\Mod\linj) = \{X\in A\Mod\ |\ [S,X]=0\text{ for all
}S\in\calS\}$. Hence the following statements are equivalent:
\begin{enumerate}
\item[(i)] There exists a set $\calS\subset A\Mod$ such that $\Acyc(A)\cap A\Mod\linj=\calS^{\perp}$.
\item[(ii)] There exists a set $\calS\subset\bfK(A\Mod)$ such that $\bfK\lac(A\Mod\linj)=\calS^{\perp}$.
\item[(iii)] The sequence $\bfK\lac(A\Mod\linj)\to\bfK(A\Mod\linj)\to\bfD(A)$ is a recollement.
\item[(iv)] The butterfly from Proposition \ref{prop_butterfly} exists.
\end{enumerate}
It would be nice to have methods at hand for checking these conditions, as well as to see examples where they fail. Note
that by \cite{Krause_StableDerived} the conditions are indeed satisfied for the sequence
$\bfK\lac(\inj(\scA))\to\bfK(\inj(\scA))\to\bfD(\scA)$ if $\scA$ is a locally Noetherian Grothendieck category such that
$\bfD(\scA)$ is compactly generated. 
\end{rem}

\label{subsection_butterfly}
\section{Examples}\label{sec_examples}
\subsection{Gorenstein rings}\label{subsection_gorenstein}
Let $R$ be a Gorenstein ring, i.e. $R$ is Noetherian and of finite injective dimension both as a left and as a right
module over itself. Considering $R$ as a dg ring concentrated in degree $0$, we 
can form the absolute singular contraderived and coderived models $\calM\uctr\lsing(R)$ and $\calM\uco\lsing(R)$ on
$\Ch(R)$, see Definition \ref{def_relativesingmodel}. The goal of this section is to see that they can be
connected through a zig-zag of Quillen equivalences to Hovey's Gorenstein projective and injective models on $R\Mod$
(see Proposition \ref{prop_gorensteinmodels}). The ``intermediate'' model structures we meet along that zig-zag are
the projective and injective versions ${^{p}}\calM\uctr\lsing(R)$ and ${^{i}}\calM\uco\lsing(R)$ of the relative
singular models introduced in Proposition \ref{prop_singmodelcoctr}.

We begin with two examples of weakly trivial objects in ${^{p}}\calM\uctr\lsing(R)$.
\begin{prop}\label{prop_weaklytrivialgproj}
Let $R$ be a Gorenstein ring and $X\in\Ch(R)$. Then we have
$X\in{^{p}}\calW\uctr\lsing(R)=(\Acyc(R)\cap\Ch(\Proj(R)))^\perp$ if either of the following holds:
\begin{enumerate}
\item $X\in\Ch^+(\Proj(R))$.
\item $X\in\Ch^-(R)\cap\Acyc(R)$.
\end{enumerate}
\end{prop}
\begin{proof}
For any $P\in{^{p}}\calC\uctr\lsing(R)=\Acyc(R)\cap\Ch(\Proj(R))$ we have $\ext^1_{\Ch(R)}(P,X)\cong
[P,\Sigma X]$. If $X\in\Ch^+(\Proj(R))$, $[P,\Sigma X]=0$ because $P$ is acyclic, has Gorenstein projective
syzygies and $X$ consists of projective modules, which are injective relative to injections with Gorenstein projective
cokernels. If $X\in\Ch^-(R)\cap\Acyc(R)$, $[P,\Sigma X]=0$ by the fundamental lemma of homological algebra.
\end{proof}

We can now describe the promised Quillen adjunction ${^{p}}\calM\uctr\lsing(R)\rightleftarrows\calM\ugproj(R)$. In the
following, we denote $\sigma\la$ resp. $\tau\la$ the brutal and soft truncation functors on categories of
complexes of $R$-modules. Given such a complex $(X,\partial)$, its $k$-th syzygy $\ker(\delta^k)$ is denoted $Z^k(X)$, and its
$k$-th cosyzygy $\coker(\delta^{k-1})$ is denoted $Q^k(X)$. Given an $R$-module $M$, we denote $\iota^k(M)$ the stalk
complex which has $M$ sitting in degree $k$ and vanishes otherwise.

\begin{lem}\label{lem_syzygyadjunction}
For any ring $R$, there is an adjunction $Q^0: \Ch(R)\rightleftarrows R\Mod: \iota^0$.
\end{lem}

\begin{prop}
Let $R$ be Gorenstein. Then the adjunction $Q^0\dashv\iota^0$ from Lemma \ref{lem_syzygyadjunction} is a Quillen
equivalence ${^{p}}\calM\uctr\lsing(R)\rightleftarrows\calM\ugproj(R)$.
\end{prop}
\begin{proof}
We show first that $Q^0\dashv\iota^0$ is a Quillen adjunction
${^{p}}\calM\uctr\lsing(R)\rightleftarrows\calM\ugproj(R)$, i.e. that $Q^0$ preserves cofibrations and trivial
cofibrations. By Proposition \ref{prop_singmodelcoctr}, a cofibration in ${^{p}}\calM\uctr\lsing(R)$ is a monomorphism
of complexes $f: X\to Y$ such that $P := \coker(f)$ is an 
acyclic complex of projective $R$-modules. Given such an $f$, the long exact sequence in cohomology associated to the exact sequence
of brutal truncations $0\to \sigma_{\leq 0} X\to\sigma_{\leq 0} Y\to \sigma_{\leq 0} P\to 0$ together with the acyclicity
of $P$ show that the sequence $0\to Q^0(X)\to Q^0(Y)\to Q^0(P)\to 0$ is exact. Moreover, $Q^0(P)\in\gproj(R)$ by
definition of Gorenstein projective modules, so $Q^0(f)$ is a monomorphism with Gorenstein projective cokernel,
i.e. a cofibration in $\calM\ugproj(R)$. Next, $Q^0$ preserves trivial cofibrations since these are monomorphisms
with projective cokernel, and $Q^0$ preserves projective objects as the left adjoint to the exact functor $\iota^0$.

To prove that $Q^0\dashv\iota^0$ is a Quillen equivalence, we have to show the following:
\begin{enumerate}
\item For each $X\in\Acyc(R)\cap\Ch(\Proj(R))$ the composition $X\to\iota^0(Q^0(X))$ is a weak equivalence in
  ${^{p}}\calM\uctr\lsing(R)$. 
\item For each $M\in R\Mod$ and some (hence any) cofibrant replacement $P\to\iota^0(M)$ in
  ${^{p}}\calM\uctr\lsing(R)$, the resulting composition $Q^0(P)\to Q^0(\iota^0(M))=M$ is a weak equivalence in
  $\calM\ugproj(R)$. 
\end{enumerate}
(1): We have $\ker(X\to(\iota^0\circ Q^0)(X))=\tau_{\leq 0}(X)\oplus\sigma_{>0}(X)$, and both summands are weakly
trivial by Proposition \ref{prop_weaklytrivialgproj}. (2): Pick a cofibrant replacement $p: K\to M$ in $\calM\ugproj(R)$,
i.e. $p$ is a trivial fibration with $K$ Gorenstein projective. As $\iota^0$ is right Quillen, $\iota^0(p):
\iota^0(K)\to\iota^0(M)$ is a trivial fibration, too, and hence for a cofibrant replacement of $\iota^0(M)$ we may take
any cofibrant replacement of $\iota^0(K)$. As $Z^0\circ\iota^0\cong\id$, we may therefore assume $M$ being Gorenstein
projective right from the beginning. If in that case $P$ is a complete projective resolution of $M$, we know from
(1) that $P\to \iota^0(M)$ is a cofibrant replacement, and applying $Q^0$ gives the identity on $M$, which is a weak
equivalence. 
\end{proof}

\begin{prop}\label{prop_zigzagproj}
Let $R$ be a Gorenstein ring. Then there is a zig-zag of left Quillen equivalences
$\calM\uctr\lsing(R)\stackrel{\id}{\longleftarrow}{^{p}}\calM\uctr\lsing(R)\stackrel{Q^0}{\longrightarrow}\calM\ugproj(R)$. 
\end{prop}

The corresponding statement about injective model structures also holds. The arguments are completely analogous, so we
omit the proof.

\begin{prop}\label{prop_zigzaginj}
Let $R$ be a Gorenstein ring. Then there is a zig-zag of right Quillen equivalences
$\calM\uco\lsing(R)\stackrel{\id}{\longleftarrow}{^{i}}\calM\uco\lsing(R)\stackrel{Z^0}{\longrightarrow}\calM\uginj(R)$. 
\end{prop}

\subsection{Curved mixed complexes}\label{subsection_mf}

In this section we study the relative singular contraderived model structure on the category of curved mixed complexes
over a ring and show that it is Quillen equivalent to the contraderived model structure on the corresponding category of
duplexes. 

\begin{definition}
Let $S$ be a ring and $w\in Z(S)$.
\begin{enumerate}
\item We denote $K_{S,w}$ the \textit{Koszul-algebra} of $(S,w)$, i.e. the $\ZZ$-graded algebra $S[s]/(s^2)$ with
  $\deg(s)=-1$ and differential $\diff$ given by $\diff(s) = w$.   
\item We denote $S_w$ the curved $\ZZ/2\ZZ$-graded dg ring with $(S_w)^{\ol{0}} = S$, $(S_w)^{\ol{1}} = 0$, trivial
  differential and curvature $w\in S = (S_w)^{\ol{2}}$. 
\end{enumerate}
\end{definition}

\begin{fact}\label{fact_kswmods} Let $S$ be a ring and $w\in Z(S)$.
\begin{enumerate}
\item A dg module over $K_{S,w}$ is a complex of $S$-modules together with a square-zero nullhomotopy for
  the multiplication by $w$, i.e. a curved mixed complex with curvature $w$.
\item A curved dg module over $S_w$ is an \textit{$(S,w)$-duplex}, i.e. a sequence $f: M^0\to M^1$, $g:
  M^1\to M^0$ of $S$-modules such that $fg=w\cdot\id_{M^1}$ and $gf=w\cdot\id_{M^0}$. Sometimes we abbreviate such a
  sequence by $f: M^0\rightleftarrows M^1: g$.
\end{enumerate}
\end{fact}

Viewing $K_{S,w}$-modules as curved mixed complexes, the cofibrant and fibrant objects in
$\calM\uctr\lsing(K_{S,w}/S)$ are easy to describe in terms of the two differentials of the mixed complex:  

\begin{prop}\label{prop_coffibincmix}
Let $X = (X,\diff,s)$ be a $K_{S,w}$-module. Then the following hold:
\begin{enumerate}
\item\label{prop_coffibincmix:cofrel} $X$ is cofibrant in $\calM\uctr\lsing(K_{S,w}/S)$ (or, equivalently, $\calM\uctr(K_{S,w})$) if and only if $(X,s)$
  is contractible and $S$-projective. 
\item\label{prop_coffibincmix:fibrel} $X$ is fibrant in $\calM\uctr\lsing(K_{S,w}/S)$ if and only if $(X,\diff)$ is $S$-contraacyclic.
\item\label{prop_coffibincmix:fibabs} $X$ is fibrant in $\calM\uctr\lsing(K_{S,w})$ if and only if $(X,\diff)$ is acyclic.
\end{enumerate}
In particular, if $S$ is semisimple, then $X$ is cofibrant (resp. fibrant) in $\calM\uctr\lsing(K_{S,w}/S)$ if and only
if $(X,\diff)$ (resp. $(X,s)$) is acyclic.
\end{prop}
\begin{proof}
\eqref{prop_coffibincmix:fibrel} and \eqref{prop_coffibincmix:fibabs} hold by
definition. \eqref{prop_coffibincmix:cofrel} is true by Lemma \ref{lem_projective}, since, by definition, $X$ is 
cofibrant in $\calM\uctr\lsing(K_{S,w}/S)$ or $\calM\uctr\lsing(K_{S,w})$ if and only if $(X,s)$ is projective in
$K_{S,w}\us\Mod\cong\Ch(S)$. 
\end{proof}

Curved mixed complexes with curvature $w$ are connected to $(S,w)$-duplexes via the operations of folding and stabilization:

\begin{definition}
Let $S$ be a ring and $w\in Z(S)$. Further, let $(X,\diff,s)$ be a $K_{S,w}$-module and $f: M^0\rightleftarrows M^1: g$ be an $(S,w)$-duplex.
\begin{enumerate}
\item The \textit{folding via products} $\fold\uprod(X)$ of $X$ is the $(S,w)$-duplex given by 
$$\fold\uprod(X)\ :=\ \prod\limits_{n\in\ZZ} X^{2n}\xrightarrow{\ \diff+s\ }\prod\limits_{n\in\ZZ} X^{2n+1}\xrightarrow{\ \diff+s\
}\prod\limits_{n\in\ZZ} X^{2n}.$$
\item The \textit{folding via sums} $\fold^\oplus(X)$ of $X$ is the $(S,w)$-duplex given by 
$$\fold^\oplus(X)\ :=\ \bigoplus\limits_{n\in\ZZ} X^{2n}\xrightarrow{\ \diff+s\ }\bigoplus\limits_{n\in\ZZ} X^{2n+1}\xrightarrow{\ \diff+s\
}\bigoplus\limits_{n\in\ZZ} X^{2n}.$$
\item The \textit{stable bar resolution} $\sbar(M)$ is the $K_{S,w}$-module given by
$$\begin{tikzpicture}[description/.style={fill=white,inner sep=2pt}]
    \matrix (m) [matrix of math nodes, row sep=3em,
                 column sep=4em, text height=1.5ex, text depth=0.25ex,
                 inner sep=0pt, nodes={inner xsep=0.3333em, inner ysep=0.3333em}]
    {
       ... & M^1\oplus M^0 & M^0\oplus M^1 & M^1\oplus M^0 & ...,\\
    };
    \draw[->] ($(m-1-1.east) + (0,0.6mm)$) -- node[scale=0.6,above]{$\begin{pmatrix}f & w \\ -\id &
        -g\end{pmatrix}$} ($(m-1-2.west) + (0,0.6mm)$);     
    \draw[->] ($(m-1-2.west) - (0,0.6mm)$) -- node[scale=0.6,below]{$\begin{pmatrix}0 & 0 \\ \id &
        0\end{pmatrix}$}($(m-1-1.east) - (0,0.6mm)$);  
    \draw[->] ($(m-1-2.east) + (0,0.6mm)$) --  node[scale=0.6,above]{$\begin{pmatrix}g & w \\ -\id &
        -f\end{pmatrix}$} ($(m-1-3.west) + (0,0.6mm)$);     
    \draw[->] ($(m-1-3.west) - (0,0.6mm)$) --  node[scale=0.6,below]{$\begin{pmatrix} 0 & 0 \\ \id &
        0\end{pmatrix}$} ($(m-1-2.east) - (0,0.6mm)$);  
    \draw[->] ($(m-1-3.east) + (0,0.6mm)$) -- node[scale=0.6,above]{$\begin{pmatrix}f & w \\ -\id &
        -g\end{pmatrix}$} ($(m-1-4.west) + (0,0.6mm)$);     
    \draw[->] ($(m-1-4.west) - (0,0.6mm)$) -- node[scale=0.6,below]{$\begin{pmatrix}0 & 0 \\ \id &
        0\end{pmatrix}$}($(m-1-3.east) - (0,0.6mm)$);   
    \draw[->] ($(m-1-4.east) + (0,0.6mm)$) -- node[scale=0.6,above]{$\begin{pmatrix}g & w \\ -\id &
        -f\end{pmatrix}$}($(m-1-5.west) + (0,0.6mm)$);      
    \draw[->] ($(m-1-5.west) - (0,0.6mm)$) -- node[scale=0.6,below]{$\begin{pmatrix} 0 & 0 \\ \id &
        0\end{pmatrix}$} ($(m-1-4.east) - (0,0.6mm)$); 
\end{tikzpicture}$$
where the terms $M^0\oplus M^1$ live in cohomologically even degrees. 
\end{enumerate}
\end{definition}

\begin{prop}\label{prop_sbaradjunctions}
There are canonical adjunctions $\sbar\dashv\fold\uprod$, $\fold^\oplus\dashv\sbar\circ\Sigma$.
\end{prop}
\begin{proof}
Let $g: M^1\rightleftarrows M^0: f$ be an $(S,w)$-duplex and $(X,\diff,s)\in K_{S,w}\Mod$. A morphism $\sbar(M)\to X$ is given by a diagram
\begin{equation*}\begin{tikzpicture}[description/.style={fill=white,inner sep=2pt}]
    \matrix (m) [matrix of math nodes, row sep=3em,
                 column sep=3em, text height=1.5ex, text depth=0.25ex,
                 inner sep=0pt, nodes={inner xsep=0.3333em, inner ysep=0.3333em}]
    {
       \cdots & M^1\oplus M^0 & M^0\oplus M^1 & M^1\oplus  M^0 & \cdots\\
       \cdots & X^{-1} & X^0 & X^1 & \cdots\\
    };
    \draw[->] ($(m-1-1.east) + (0,0.6mm)$) -- node[above,scale=0.6]{$$} ($(m-1-2.west) + (0,0.6mm)$);
    \draw[->] ($(m-1-2.west) - (0,0.6mm)$) -- node[below,scale=0.6]{$$} ($(m-1-1.east) - (0,0.6mm)$);
    \draw[->] ($(m-1-2.east) + (0,0.6mm)$) -- node[above,scale=0.6]{$\begin{pmatrix} g & w \\ -\id & -f\end{pmatrix}$} ($(m-1-3.west) + (0,0.6mm)$);
    \draw[->] ($(m-1-3.west) - (0,0.6mm)$) -- node[below,scale=0.6]{$\begin{pmatrix} 0 & 0 \\ \id & 0\end{pmatrix}$} ($(m-1-2.east) - (0,0.6mm)$);
    \draw[->] ($(m-1-3.east) + (0,0.6mm)$) -- node[above,scale=0.6]{$\begin{pmatrix} f & w \\ -\id & -g\end{pmatrix}$} ($(m-1-4.west) + (0,0.6mm)$);
    \draw[->] ($(m-1-4.west) - (0,0.6mm)$) -- node[below,scale=0.6]{$\begin{pmatrix} 0 & 0 \\ \id & 0\end{pmatrix}$} ($(m-1-3.east) - (0,0.6mm)$);
    \draw[->] ($(m-1-4.east) + (0,0.6mm)$) -- node[above,scale=0.6]{$$} ($(m-1-5.west) + (0,0.6mm)$);
    \draw[->] ($(m-1-5.west) - (0,0.6mm)$) -- node[below,scale=0.6]{$$} ($(m-1-4.east) - (0,0.6mm)$);
    \draw[->] ($(m-2-1.east) + (0,0.6mm)$) -- node[above,scale=0.6]{$\diff$} ($(m-2-2.west) + (0,0.6mm)$);
    \draw[->] ($(m-2-2.west) - (0,0.6mm)$) -- node[below,scale=0.6]{$s$} ($(m-2-1.east) - (0,0.6mm)$);
    \draw[->] ($(m-2-2.east) + (0,0.6mm)$) -- node[above,scale=0.6]{$\diff$} ($(m-2-3.west) + (0,0.6mm)$);
    \draw[->] ($(m-2-3.west) - (0,0.6mm)$) -- node[below,scale=0.6]{$s$} ($(m-2-2.east) - (0,0.6mm)$);
    \draw[->] ($(m-2-3.east) + (0,0.6mm)$) -- node[above,scale=0.6]{$\diff$} ($(m-2-4.west) + (0,0.6mm)$);
    \draw[->] ($(m-2-4.west) - (0,0.6mm)$) -- node[below,scale=0.6]{$s$} ($(m-2-3.east) - (0,0.6mm)$);
    \draw[->] ($(m-2-4.east) + (0,0.6mm)$) -- node[above,scale=0.6]{$\diff$} ($(m-2-5.west) + (0,0.6mm)$);
    \draw[->] ($(m-2-5.west) - (0,0.6mm)$) -- node[below,scale=0.6]{$s$} ($(m-2-4.east) - (0,0.6mm)$);
    \draw[->] (m-1-2) -- node[description,scale=0.6]{$\begin{pmatrix}\alpha_{-1} & \alpha_{-1}\p\end{pmatrix}$} (m-2-2);
    \draw[->] (m-1-3) -- node[description,scale=0.6]{$\begin{pmatrix}\alpha_{0} & \alpha_{0}\p\end{pmatrix}$} (m-2-3);
    \draw[->] (m-1-4) -- node[description,scale=0.6]{$\begin{pmatrix}\alpha_{1} & \alpha_{1}\p\end{pmatrix}$} (m-2-4);
\end{tikzpicture}\end{equation*}
such that each square commutes both with respect to the maps pointing to the right and the ones pointing to the
left. The latter is equivalent to $\alpha_n\p = s\alpha_{n+1}$ for all $n\in\ZZ$, so assume this from now on. Writing
$\partial$ in place of $f$ and $g$ (to avoid distinction of cases), the other commutativity constraint then writes as
follows: 
\begin{enumerate}
\item $\alpha_{n} \partial - s\alpha_{n+1} = \diff\alpha_{n-1}$.
\item $\diff s\alpha_n = w\alpha_n - s\alpha_{n+1}\partial$.
\end{enumerate}
The second condition follows from the first by applying $s\circ-$. Thus, the constraint on the family
$\{\alpha_n\}_{n\in\ZZ}$ to yield a morphism of $K_{S,w}$-modules $\sbar(M)\to X$ is $\alpha\partial = (\diff+s)\alpha$,
in this in turn is equivalent to saying that $\prod\alpha_{2n}$ and $\prod\alpha_{2n+1}$ yield a morphism of duplexes
$M\to\fold\uprod(X)$.  

Similarly, a morphism $X\to\sbar(M)\circ\Sigma$ is given by a diagram
\begin{equation*}\begin{tikzpicture}[description/.style={fill=white,inner sep=2pt}]
    \matrix (m) [matrix of math nodes, row sep=3em,
                 column sep=2.5em, text height=1.5ex, text depth=0.25ex,
                 inner sep=0pt, nodes={inner xsep=0.3333em, inner ysep=0.3333em}]
    {
       \cdots & X^{-1} & X^0 & X^1 & \cdots\\
       \cdots & M^0\oplus M^1 & M^1\oplus M^0 & M^0\oplus  M^1 & \cdots\\
    };
    \draw[->] ($(m-2-1.east) + (0,0.6mm)$) -- node[above,scale=0.6]{$$} ($(m-2-2.west) + (0,0.6mm)$);
    \draw[->] ($(m-2-2.west) - (0,0.6mm)$) -- node[below,scale=0.6]{$$} ($(m-2-1.east) - (0,0.6mm)$);
    \draw[->] ($(m-2-2.east) + (0,0.6mm)$) -- node[above,scale=0.6]{$\begin{pmatrix} -f & w \\ -\id & g\end{pmatrix}$} ($(m-2-3.west) + (0,0.6mm)$);
    \draw[->] ($(m-2-3.west) - (0,0.6mm)$) -- node[below,scale=0.6]{$\begin{pmatrix} 0 & 0 \\ \id & 0\end{pmatrix}$} ($(m-2-2.east) - (0,0.6mm)$);
    \draw[->] ($(m-2-3.east) + (0,0.6mm)$) -- node[above,scale=0.6]{$\begin{pmatrix} -g & w \\ -\id & f\end{pmatrix}$} ($(m-2-4.west) + (0,0.6mm)$);
    \draw[->] ($(m-2-4.west) - (0,0.6mm)$) -- node[below,scale=0.6]{$\begin{pmatrix} 0 & 0 \\ \id & 0\end{pmatrix}$} ($(m-2-3.east) - (0,0.6mm)$);
    \draw[->] ($(m-2-4.east) + (0,0.6mm)$) -- node[above,scale=0.6]{$$} ($(m-2-5.west) + (0,0.6mm)$);
    \draw[->] ($(m-2-5.west) - (0,0.6mm)$) -- node[below,scale=0.6]{$$} ($(m-2-4.east) - (0,0.6mm)$);
    \draw[->] ($(m-1-1.east) + (0,0.6mm)$) -- node[above,scale=0.6]{$\diff$} ($(m-1-2.west) + (0,0.6mm)$);
    \draw[->] ($(m-1-2.west) - (0,0.6mm)$) -- node[below,scale=0.6]{$s$} ($(m-1-1.east) - (0,0.6mm)$);
    \draw[->] ($(m-1-2.east) + (0,0.6mm)$) -- node[above,scale=0.6]{$\diff$} ($(m-1-3.west) + (0,0.6mm)$);
    \draw[->] ($(m-1-3.west) - (0,0.6mm)$) -- node[below,scale=0.6]{$s$} ($(m-1-2.east) - (0,0.6mm)$);
    \draw[->] ($(m-1-3.east) + (0,0.6mm)$) -- node[above,scale=0.6]{$\diff$} ($(m-1-4.west) + (0,0.6mm)$);
    \draw[->] ($(m-1-4.west) - (0,0.6mm)$) -- node[below,scale=0.6]{$s$} ($(m-1-3.east) - (0,0.6mm)$);
    \draw[->] ($(m-1-4.east) + (0,0.6mm)$) -- node[above,scale=0.6]{$\diff$} ($(m-1-5.west) + (0,0.6mm)$);
    \draw[->] ($(m-1-5.west) - (0,0.6mm)$) -- node[below,scale=0.6]{$s$} ($(m-1-4.east) - (0,0.6mm)$);
    \draw[->] (m-1-2) -- node[description,scale=0.6]{$\begin{pmatrix}\alpha_{-1}\p \\ \alpha_{-1}\end{pmatrix}$} (m-2-2);
    \draw[->] (m-1-3) -- node[description,scale=0.6]{$\begin{pmatrix}\alpha_{0}\p\\ \alpha_{0}\end{pmatrix}$} (m-2-3);
    \draw[->] (m-1-4) -- node[description,scale=0.6]{$\begin{pmatrix}\alpha_{1}\p \\ \alpha_{1}\end{pmatrix}$} (m-2-4);
\end{tikzpicture}\end{equation*}
such that each square commutes both with respect to the maps pointing to the right and the ones pointing to the
left. The latter is equivalent to $\alpha\p_n = \alpha_{n-1}s$, and we assume this from now. Then, again writing $\partial$
for $f$ and $g$, the other commutativity constraint writes as
\begin{enumerate}
\item $w\alpha_n-\partial\alpha_{n-1} s = \alpha_n s \diff$
\item $\partial\alpha_n - \alpha_{n-1} s = \alpha_{n+1}\diff$.
\end{enumerate}
The first condition follows from the second by applying $-\circ s$, and the second is equivalent to saying that
$\bigoplus\alpha_{2n}$ and $\bigoplus\alpha_{2n+1}$ yield a morphism of $S_w$-modules $\fold^\oplus(X)\to M$.
\end{proof}

\begin{prop}\label{prop_quillenadjunctions}
Let $S$ be a ring and $w\in Z(S)$. Then the following are Quillen adjunctions:
\begin{enumerate}
\item\label{item:quill1} $\sbar: \calM\uctr(S_w)\rightleftarrows\calM\uctr\lsing(K_{S,w}):\fold\uprod$
\item\label{item:quill2} $\sbar: \calM\uctr(S_w)\rightleftarrows\calM\uctr\lsing(K_{S,w}/S):\fold\uprod$.
\item\label{item:quill3} $\fold^\oplus: \calM\uctr\lsing(K_{S,w})\rightleftarrows\calM\uctr(S_w):\sbar\circ\Sigma$.
\item\label{item:quill4} $\fold^\oplus: \calM\uctr\lsing(K_{S,w}/S)\rightleftarrows\calM\uctr(S_w):\sbar\circ\Sigma$.
\end{enumerate}
\end{prop}
\begin{proof}
Because of the trivial Quillen adjunction $\id:\calM\uctr\lsing(A)\rightleftarrows\calM\uctr\lsing(A/R):\id$ between
absolute and relative contraderived singularity models, \eqref{item:quill2} follows from \eqref{item:quill1} and
\eqref{item:quill3} follows from \eqref{item:quill4}.

For \eqref{item:quill1}, we have to show that $\sbar$ preserves cofibrations and trivial
cofibrations. By the exactness of $\sbar$ and the definition of an abelian model structure, it suffices to show
$\sbar(\calC)\subset\calC$ and $\sbar(\calC\cap\calW)\subset\calC\cap\calW$. The cofibrants in $\calM\uctr(S_w)$
are those $f: M^0\rightleftarrows M^1: g$ with $M^0, M^1$ projective $S$-modules, and the cofibrants in
$\calM\uctr\lsing(K_{S,w})$ are the $K_{S,w}$-modules with underlying projective $K_{S,w}\us$-modules. By definition of
$\sbar$, the $K_{S,w}\us$-module underlying $\sbar(M)$ is isomorphic to $\bigoplus_{n\in\ZZ} K_{S,w}\us\otimes_S
\Sigma^{2n}M^0\oplus K_{S,w}\us\otimes_S \Sigma^{2n+1}M^1$, and hence is $K_{S,w}\us$-projective if $M^0, M^1$ are $S$-projective. This
proves $\sbar(\calC)\subset\calC$. The assertion $\sbar(\calC\cap\calW)\subset\calC\cap\calW={^{\perp}}\calF$ is clear
because $\calC\cap\calW=\proj$ in $\calM\uctr(S_w)$ and $\sbar$ preserves projectives as the left adjoint to the exact
functor $\fold\uprod$. 

For \eqref{item:quill4}, we have to show that $(\sbar\circ\Sigma)(\calF)\subset\calF$ and
$(\sbar\circ\Sigma)(\calW\cap\calF)\subset\calW\cap\calF$. In $\calM\uctr(S_w)$ everything is fibrant, while in 
$\calM\uctr\lsing(K_{S,w}/S)$ the fibrants are the $S$-contraacyclic $K_{S,w}$-modules, so for
$(\sbar\circ\Sigma)(\calF)\subset\calF$ we have to show that the image of $\sbar$ consists of $S$-contraacyclic
complexes. The stable bar resolutions are even contractible as complexes of 
$S$-modules, so this follows from Proposition \ref{prop_coctracyclic}. The other condition
$(\sbar\circ\Sigma)(\calW\cap\calF)\subset\calW\cap\calF$ means that $\sbar$ maps $S_w$-contraacyclics to
$K_{S,w}$-contraacyclics, i.e. that it maps $S_w\Mod\lproj^{\perp}$ to $K_{S,w}\Mod\lproj^{\perp}$. For this, suppose $X\in
K_{S,w}\Mod$ and $M$ is $S_w$-contraacyclic. Then $\ext^1_{K_{S,w}}(X,(\sbar\circ\Sigma)(M))\cong
\ext^1_{S_w}(\fold^{\oplus}(X),M)$, which is trivial since $\fold^{\oplus}(X)\in S_w\Mod\lproj$.
\end{proof}

Our goal is to show that the adjunctions \ref{prop_quillenadjunctions}\eqref{item:quill2} and
\ref{prop_quillenadjunctions}\eqref{item:quill4} are Quillen equivalences, but before we come to the proof, we
define the completed Bar resolution.

\begin{fact}[\protect{\cite[Proposition 8.6.10]{Weibel_HomologicalAlgebra}}]\label{fact_comonad}
Let $F:\scA\rightleftarrows\scB:U$ be an adjunction between abelian categories and $\perp:=FU:\scB\to\scB$ the
associated comonad. For $X\in\scB$ there is a canonical structure of a simplicial object on ${\perp\ua}X :=
\left\{{\perp^{n+1}}X\right\}_{n\geq 0}$, and $U({\perp\ua}X)$ admits a canonical left contraction. In particular,
if $U$ is exact and faithful, then the normalized augmented chain complex $N({\perp\ua}X)\to X$ is acyclic.
\end{fact}

\begin{cor}\label{cor_barres}
Let $S$ be a ring, $A$ be a dg $S$-algebra and $M$ an $A$-module. Let $\eta: S\to A$ be the structure map and $\ol{A} :=
\coker(\eta)$. Then the following augmented complex of $A$-modules is acyclic:
\begin{align}
\label{eq:barres}\left(...\to A\otimes_S \ol{A}\otimes_S \ol{A}\otimes_S M\to A\otimes_S \ol{A}\otimes_S M\to A\otimes_S
  M\right)\to M.
\end{align}
\end{cor}

\begin{definition}\label{def_completedbarres}
Let $S$ be a ring, $A$ be a dg $S$-algebra and $M$ an $A$-module. The \textit{completed Bar resolution} of $M$ is the
totalization of the augmented complex \eqref{eq:barres} formed by taking products, and is denoted $B\uprod M\to M$. 
\end{definition}

\begin{lem}\label{lem_barresweakequiv}
Let $S$, $A$ and $M$ be as in Definition \ref{def_completedbarres} and let $q: B\uprod M\to M$ be
the completed Bar resolution. Then $\ker(q)$ is contraacyclic. In other words, the completed Bar resolution $B\uprod M\to
M$ is a trivial fibration in $\calM\uctr\lsing(A)$.
\end{lem}
\begin{proof}
The second statement follows from the first since the contraacyclic $A$-modules are precisely the trivially fibrant
objects in $\calM\uctr\lsing(A)$. That $\ker(q)$ is contraacyclic follows from Lemma \ref{lem_totalizations} as it
is the totalization by taking products of a bounded above exact sequence of $A$-modules. 
\end{proof}

The following gives explicit descriptions of the functors $\sbar\circ\fold\uprod$ and $B\uprod$.

\begin{lem}\label{lem_barconcrete}
Let $(X,\diff,s)$ be a $K_{S,w}$-module. There are natural isomorphisms
$$(\sbar\circ\fold\uprod)(X)^n\cong\prod\limits_{k\in\ZZ} X^k\quad\text{and}\quad (B\uprod X)^n\cong\prod\limits_{k\geq n} X^k.$$
Under these isomorphisms, the $K_{S,w}$-module structure can be described as follows:
\begin{enumerate}
\item\label{item:modstr1} $\diff$ acts on $X^k$ as $\diff+s-\id$ for $k\equiv n\pmod{2}$ and as
$w-\diff-s$ otherwise.
\item\label{item:modstr2} $s$ acts on $X^k$ as $\id$ if $k\equiv n\pmod{2}$ and as $0$ otherwise.
\end{enumerate}
In particular, we have the following:
\begin{enumerate}
\item There is a canonical epimorphism of $K_{S,w}$-modules
$$\alpha: (\sbar\circ\fold\uprod)(X)\longrightarrow B\uprod X$$
with $\ker(\alpha)^n \cong\prod\limits_{k<n} X^k$ and $K_{S,w}$-module structure as in \eqref{item:modstr1} and
\eqref{item:modstr2}. 
\item $\ker(\alpha)$ admits a complete decreasing filtration $...\subset F_2\subset F_1\subset F_0 = \ker(\alpha)$ with
  $F_n / F_{n+1}\cong K_{S,w}\otimes_S \Sigma^{-2n-2}X$. 
\end{enumerate}
\end{lem}
\begin{proof}
To compute $B\uprod X$, note that for the unit $\eta: S\to K_{S,w}$ we have $\ol{K_{S,w}}=\coker(\eta)=\Sigma S$. Hence the
$n$-th term in the augmented Bar resolution \eqref{eq:barres} is given by $K_{S,w}\otimes_S \Sigma^n X$, and the differential
$K_{S,w}\otimes_S \Sigma^n X\to K_{S,w}\otimes_S \Sigma^{n-1}X$ maps $a\otimes x$ to $as\otimes x + (-1)^{n} a\otimes sx$. All in
all, the Bar (bi)complex is given as follows:
\begin{equation*}\begin{tikzpicture}[description/.style={fill=white,inner sep=2pt}]
    \matrix (m) [matrix of math nodes, row sep=2.5em,
                 column sep=2.5em, text height=1.5ex, text depth=0.25ex,
                 inner sep=0pt, nodes={inner xsep=0.3333em, inner ysep=0.3333em}]
    {
      & \vdots & \vdots & \vdots && \vdots \\
      \cdots & X^0\oplus X^1 & X^1\oplus X^2 & X^2\oplus X^3 & \cdots &  K_{S,w}\otimes_S \Sigma X\\
      \cdots & X^{-1}\oplus X^0 & X^0\oplus X^1 & X^1\oplus X^2 & \cdots & K_{S,w}\otimes_S X\\
       \cdots & X^{-1} & X^0 & X^1 & \cdots & X\\
    };
    \draw[->] (m-2-6) -- (m-3-6);
    \draw[->] (m-3-6) -- (m-4-6);
    \draw[->] ($(m-4-1.east) + (0,0.6mm)$) -- node[above,scale=0.6]{$\diff$} ($(m-4-2.west) + (0,0.6mm)$);
    \draw[->] ($(m-4-2.west) - (0,0.6mm)$) -- node[below,scale=0.6]{$s$} ($(m-4-1.east) - (0,0.6mm)$);
    \draw[->] ($(m-4-2.east) + (0,0.6mm)$) -- node[above,scale=0.6]{$\diff$} ($(m-4-3.west) + (0,0.6mm)$);
    \draw[->] ($(m-4-3.west) - (0,0.6mm)$) -- node[below,scale=0.6]{$s$} ($(m-4-2.east) - (0,0.6mm)$);
    \draw[->] ($(m-4-3.east) + (0,0.6mm)$) -- node[above,scale=0.6]{$\diff$} ($(m-4-4.west) + (0,0.6mm)$);
    \draw[->] ($(m-4-4.west) - (0,0.6mm)$) -- node[below,scale=0.6]{$s$} ($(m-4-3.east) - (0,0.6mm)$);
    \draw[->] ($(m-4-4.east) + (0,0.6mm)$) -- node[above,scale=0.6]{$\diff$} ($(m-4-5.west) + (0,0.6mm)$);
    \draw[->] ($(m-4-5.west) - (0,0.6mm)$) -- node[below,scale=0.6]{$s$} ($(m-4-4.east) - (0,0.6mm)$);
    \draw[->] ($(m-4-1.east) + (0,0.6mm)$) -- ($(m-4-2.west) + (0,0.6mm)$);
    \draw[->] ($(m-3-1.east) + (0,0.6mm)$) -- ($(m-3-2.west) + (0,0.6mm)$);
    \draw[->] ($(m-3-2.west) - (0,0.6mm)$) -- ($(m-3-1.east) - (0,0.6mm)$);
    \draw[->] ($(m-3-2.east) + (0,0.6mm)$) -- node[above,scale=0.6]{$\begin{pmatrix} \diff & w\\ 0 & -\diff\end{pmatrix}$}($(m-3-3.west) + (0,0.6mm)$);
    \draw[->] ($(m-3-3.west) - (0,0.6mm)$) -- node[below,scale=0.6]{$\begin{pmatrix} 0 & 0 \\ 1 & 0 \end{pmatrix}$} ($(m-3-2.east) - (0,0.6mm)$);
    \draw[->] ($(m-3-3.east) + (0,0.6mm)$) -- node[above,scale=0.6]{$\begin{pmatrix} \diff & w\\ 0 & -\diff\end{pmatrix}$} ($(m-3-4.west) + (0,0.6mm)$);
    \draw[->] ($(m-3-4.west) - (0,0.6mm)$) -- node[below,scale=0.6]{$\begin{pmatrix} 0 & 0 \\ 1 & 0 \end{pmatrix}$} ($(m-3-3.east) - (0,0.6mm)$);
    \draw[->] ($(m-3-4.east) + (0,0.6mm)$) --  ($(m-3-5.west) + (0,0.6mm)$);
    \draw[->] ($(m-3-5.west) - (0,0.6mm)$) -- ($(m-3-4.east) - (0,0.6mm)$);
    \draw[->] ($(m-2-1.east) + (0,0.6mm)$) -- ($(m-2-2.west) + (0,0.6mm)$);
    \draw[->] ($(m-2-2.west) - (0,0.6mm)$) -- ($(m-2-1.east) - (0,0.6mm)$);
    \draw[->] ($(m-2-2.east) + (0,0.6mm)$) -- node[above,scale=0.6]{$\begin{pmatrix} -\diff & w\\ 0 & \diff\end{pmatrix}$} ($(m-2-3.west) + (0,0.6mm)$);
    \draw[->] ($(m-2-3.west) - (0,0.6mm)$) -- node[below,scale=0.6]{$\begin{pmatrix} 0 & 0 \\ 1 & 0 \end{pmatrix}$} ($(m-2-2.east) - (0,0.6mm)$);
    \draw[->] ($(m-2-3.east) + (0,0.6mm)$) -- node[above,scale=0.6]{$\begin{pmatrix} -\diff & w\\ 0 & \diff\end{pmatrix}$} ($(m-2-4.west) + (0,0.6mm)$);
    \draw[->] ($(m-2-4.west) - (0,0.6mm)$) -- node[below,scale=0.6]{$\begin{pmatrix} 0 & 0 \\ 1 & 0 \end{pmatrix}$} ($(m-2-3.east) - (0,0.6mm)$);
    \draw[->] ($(m-2-4.east) + (0,0.6mm)$) -- ($(m-2-5.west) + (0,0.6mm)$);
    \draw[->] ($(m-2-5.west) - (0,0.6mm)$) -- ($(m-2-4.east) - (0,0.6mm)$);
    \draw[->] (m-2-2) -- node[description,scale=0.6]{$\begin{pmatrix} -s & 0\\ \id & -s \end{pmatrix}$} (m-3-2);
    \draw[->] (m-3-2) -- node[description,scale=0.6]{$\begin{pmatrix} 1 & s \end{pmatrix}$} (m-4-2);
    \draw[->] (m-2-3) -- node[description,scale=0.6]{$\begin{pmatrix} -s & 0\\ \id & -s \end{pmatrix}$} (m-3-3);
    \draw[->] (m-3-3) -- node[description,scale=0.6]{$\begin{pmatrix} 1 & s \end{pmatrix}$} (m-4-3);
    \draw[->] (m-2-4) -- node[description,scale=0.6]{$\begin{pmatrix} -s & 0\\ \id & -s \end{pmatrix}$} (m-3-4);
    \draw[->] (m-3-4) -- node[description,scale=0.6]{$\begin{pmatrix} 1 & s \end{pmatrix}$} (m-4-4);
    \draw[->] (m-1-2) -- node[description,scale=0.6]{$\begin{pmatrix} s & 0\\ \id & s \end{pmatrix}$} (m-2-2);
    \draw[->] (m-1-3) -- node[description,scale=0.6]{$\begin{pmatrix} s & 0\\ \id & s \end{pmatrix}$} (m-2-3);
    \draw[->] (m-1-4) -- node[description,scale=0.6]{$\begin{pmatrix} s & 0\\ \id & s \end{pmatrix}$} (m-2-4);
    \draw[->] (m-1-6) -- (m-2-6);
\end{tikzpicture}\end{equation*}
By definition of the totalization, $B\uprod X$ is equal to $\prod_{k\geq 0} \Sigma^k (K_{S,w}\otimes_S \Sigma^k X)$ as a
$K_{S,w}\us$-module, with differential being the sum of the differentials on the $\Sigma^k (K_{S,w}\otimes_S \Sigma^k X)$
and the maps $K_{S,w}\otimes_S \Sigma^k X\to K_{S,w}\otimes\Sigma^{k-1} X$. As $K_{S,w}\us$-modules we have $\Sigma^k
(K_{S,w}\otimes_S \Sigma^k X)\cong K_{S,w}\otimes_S \Sigma^{2k} X$ via $a\otimes x\mapsto (-1)^{k|a|+\frac{k(k+1)}{2}}
a\otimes x$, and the $n$-th term of $\prod_{k\geq 0} (K_{S,w}\otimes_S \Sigma^{2k} X)$ is given by $\prod_{k\geq n} X^n$. Pulling back the
differential on $\prod_{k\geq 0} \Sigma^k (K_{S,w}\otimes_S \Sigma^k X)$ to $\prod_{k\geq 0} (K_{S,w}\otimes_S
\Sigma^{2k} X)$ via the above sign change, the resulting differential is given as $\diff+s-\id$ on factors $X^k$ with
$n\equiv k\pmod{2}$ and as $w-\diff-s$ on those $X^k$ with $k\not\equiv n\pmod{2}$, as claimed.

The statement about the description of $(\sbar\circ\fold\uprod)(X)$ and the canonical epimorphism
$\alpha: (\sbar\circ\fold\uprod)(X)\to B\uprod X$ is clear. For the last statement about the filtration on $\ker(\alpha)$, define
$F_i\subset\ker(\alpha)$ by $(F_i)^n:=\prod_{k<n-2i} X^k$. Clearly this is a complete decreasing filtration, and the
filtration quotient $F_i/F_{i+1}$ is given by $(F_i/F_{i+1})^n=X^{n-2i-1}\oplus X^{n-2i-2}$. Together with the
explicit description of the differential on $\ker(\alpha)$ we conclude that $F_i/F_{i+1}\cong K_{S,w}\otimes_S \Sigma^{-2i-2}X$. 
\end{proof}

\begin{theorem}\label{thm_singularitymodelofkoszulalgebra}
Let $S$ be a ring and $w\in Z(S)$. Then the adjunctions
\begin{equation*}\begin{tikzpicture}[description/.style={fill=white,inner sep=2pt}]
    \matrix (m) [matrix of math nodes, row sep=1em,
                 column sep=2.5em, text height=1.5ex, text depth=0.25ex,
                 inner sep=0pt, nodes={inner xsep=0.3333em, inner ysep=0.3333em}]
    {
       \makebox[1.2cm][r]{$\sbar:\ $}\makebox[2.3cm][c]{$\calM\uctr(S_w)$} &
       \makebox[2.3cm][l]{$\calM\uctr\lsing(K_{S,w}/S)$}\makebox[1.2cm][l]{$\ :\fold\uprod.$}\\
       \makebox[1.2cm][r]{$\fold^\oplus:\ $}\makebox[2.3cm][r]{$\calM\uctr\lsing(K_{S,w}/S)$} &
       \makebox[2.3cm][c]{$\calM\uctr(S_w)$}\makebox[1.2cm][l]{$\ :\sbar\circ\Sigma$}.\\
    };
    \draw[->] ($(m-1-1.east) + (0,0.7mm)$) -- ($(m-1-2.west) + (0,0.7mm)$);
    \draw[->] ($(m-1-2.west) - (0,0.7mm)$) -- ($(m-1-1.east) - (0,0.7mm)$);
    \draw[->] ($(m-2-1.east) + (0,0.7mm)$) -- ($(m-2-2.west) + (0,0.7mm)$);
    \draw[->] ($(m-2-2.west) - (0,0.7mm)$) -- ($(m-2-1.east) - (0,0.7mm)$);
\end{tikzpicture}\end{equation*}
are Quillen equivalences.
\end{theorem}

\begin{proof}
We already know from Proposition \ref{prop_quillenadjunctions} that the adjunctions in question are Quillen adjunctions,
so it remains to check that unit and counit of the derived adjunctions are isomorphisms. 

To show that the derived counit $\bfL\sbar\circ\bfR\fold\uprod\Rightarrow\id$ is an isomorphism, we have to show
  that for fibrant $X\in\calM\uctr\lsing(K_{S,w})$ and a
  cofibrant resolution $Y\to\fold\uprod X$ in $\calM\uctr(S_w)$ the morphism
\begin{align*}
\sbar(Y)\longrightarrow(\sbar\circ\fold\uprod)(X)\longrightarrow X
\end{align*}
 is a weak equivalence in $\calM\uctr\lsing(K_{S,w})$.
 By definition of a cofibrant resolution, the morphism $Y\to\fold\uprod X$ is a trivial fibration, and hence so is
 $\sbar(Y\to\fold\uprod X)$ by Proposition \ref{prop_quillenadjunctions}\eqref{item:quill4}. Moreover, since the fibrants in
 $\calM\uctr\lsing(K_{S,w}/S)$ are the $S$-contraacyclic $K_{S,w}$-modules, we therefore have to show that for some
 $S$-contraacyclic $X\in K_{S,w}\Mod$ the (ordinary) counit $\varepsilon_X: (\sbar\circ\fold\uprod)(X)\to X$ is a weak
 equivalence in $\calM\uctr\lsing(K_{S,w}/S)$. 
For this, recall from Lemma \ref{lem_barconcrete} that $\varepsilon_X$ factors through the completed Bar resolution $q: B\uprod
X\to X$ via a canonical epimorphism $\alpha: (\sbar\circ\fold\uprod)(X)\to B\uprod X$ described there. Since the completed
Bar resolution $B\uprod X\to X$ is a weak equivalence in $\calM\uctr\lsing(K_{S,w}/S)$ (even in $\calM\uctr(K_{S,w})$) by
Lemma \ref{lem_barresweakequiv}, it is therefore sufficient to check that 
$\alpha$ is a weak equivalence in $\calM\uctr\lsing(K_{S,w}/S)$. In fact, we will show that $\alpha$ is even a trivial
fibration, i.e. that $\ker(\alpha)$ is $K_{S,w}$-contraacyclic: First, by Lemma \ref{lem_barconcrete} we know that $\ker(\alpha)$
admits a complete descending filtration with filtration quotients isomorphic to shifts of $K_{S,w}\otimes_S X$. We have
$\Hom_S(K_{S,w},X)\cong\Hom_S(K_{S,w},S)\otimes_S X$, and since $\Hom_S(K_{S,w},S)\cong\Omega K_{S,w}$ as 
$K_{S,w}$-$S$-bimodules, we get $K_{S,w}\otimes_S X\cong\Sigma\Hom_S(K_{S,w},X)$. Since $K_{S,w}\us$ is free over
$S\us$, Proposition \ref{prop_fouradjunctions}\eqref{item_fouradj5} and the assumption that $X$ is
$S$-contraacyclic yield that $K_{S,w}\otimes_S X$ is $K_{S,w}$-contraacyclic, too. We conclude that $\ker(\alpha)$
admits a complete descending filtration with $K_{S,w}$-contraacyclic filtration quotients; Lemma
\ref{lem_totalizations} then shows that $\ker(\alpha)$ is $K_{S,w}$-contraacyclic, as claimed. 

Similarly, the derived unit $\id\Rightarrow\bfR\fold\uprod\circ\bfL\sbar$ being an isomorphism means that for
any cofibrant duplex $f: M^0\rightleftarrows M^1: g$ and a fibrant resolution $\sbar(M)\to X$ in
$\calM\uctr\lsing(K_{S,w}/S)$ the morphism $$M\to(\fold\uprod\circ\sbar)(M)\to \fold\uprod(X)$$ is a weak equivalence
in $\calM\uctr(S_w)$. By Proposition \ref{prop_quillenadjunctions}\eqref{item:quill4} any object in the image of
$\sbar$ is fibrant in $\calM\uctr\lsing(K_{S,w}/S)$, and hence we have to show that for $M\in S_w\Mod$ with $M^0, M^1$
projective over $S$ the unit $M\to(\fold\uprod\circ\sbar)(M)$ is a weak equivalence in $\calM\uctr(S_w)$. In fact, we
will show that this is true for \textit{any} $S_w$-module $M$.

Note that there is a canonical isomorphism $M\cong\fold\uprod(i(M))$ where $i(M)$ is given by $g:
M^1\rightleftarrows M^0: f$ in cohomological degrees $-1$ and $0$, and $0$ otherwise; it follows that the unit
$M\to(\fold\uprod\circ\sbar)(M)$ is split by the composition
$$(\fold\uprod\circ\sbar)(M)\cong\fold\uprod((\sbar\circ\fold\uprod)(i(M)))\xrightarrow{\fold\uprod(\varepsilon_{i(M)})}\fold\uprod(i(M))=M.$$  
Hence, in order to show that $M\to(\fold\uprod\circ\sbar)(M)$ is a weak equivalence in $\calM\uctr(S_w)$ it is therefore sufficient to show that
$\fold\uprod(\varepsilon_{i(M)})$ is a weak equivalence in $\calM\uctr(S_w)$, and we will show that it is even a trivial
fibration. First, recall that $\varepsilon_{i(M)}$ factors through the completed 
Bar resolution $q: B\uprod(i(M))\to i(M)$ via the map $\alpha: \sbar(M)\to B\uprod(i(M))$. Since $q$ is a trivial fibration and the right
Quillen functor $\fold\uprod$ preserves trivial fibrations, this means that we only have to check that $\fold\uprod(\alpha)$ is a trivial
fibration, i.e. that $\fold\uprod(\ker(\alpha))$ is trivially fibrant in $\calM\uctr(S_w)$. For this, recall from Lemma
\ref{lem_barconcrete} that $\fold\uprod(\ker(\alpha))$ admits a complete decreasing filtration with filtration
quotients being shifts of $\fold\uprod(K_{S,w}\otimes_S i(M))$. $K_{S,w}\otimes_S i(M)$ is an extension of $K_{S,w}\otimes_S M^0$ and $K_{S,w}\otimes_S \Sigma M^1$, and hence
$\fold\uprod(K_{S,w}\otimes_S i(M))$ is an extension of $\fold\uprod(K_{S,w}\otimes_S M^0)$ and
$\fold\uprod(K_{S,w}\otimes_S \Sigma M^1)$, both of which are contractible, hence contraacyclic, by Proposition
\ref{prop_coctracyclic}. Applying Lemma \ref{lem_totalizations} shows that $\fold\uprod(\ker(\alpha))$ is
$S_w$-contraacyclic, as claimed. 

The statement that $\fold^\oplus\dashv\sbar\circ\Sigma$ is a Quillen equivalence follows from the first part since
$\bfR(\sbar\circ\Sigma)=\bfR\sbar\circ\Sigma = \bfL\sbar\circ\Sigma$ is invertible and a Quillen adjunction is a
Quillen equivalence if and only if its derived adjunction is an adjoint equivalence \cite[Proposition
1.3.13]{Hovey_ModelCategories}. 
\end{proof}

From Theorem \ref{thm_singularitymodelofkoszulalgebra} we get the following consequence:

\begin{cor}\label{cor_prodsumfolding} There is an isomorphism
$$\Sigma\circ\bfL\fold^\oplus\cong\bfR\fold\uprod$$
of functors $\Ho(\calM\uctr\lsing(K_{S,w}/S))\to\Ho(\calM\uctr(S_w))$.
\end{cor}
\begin{proof}
By Theorem \ref{thm_singularitymodelofkoszulalgebra} we know that $\bfL\sbar=\bfR\sbar$ is invertible, and that we have
canonical adjunctions $\bfL\sbar\dashv\bfR\fold\uprod$ and $\Sigma\circ\bfL\fold^\oplus\dashv\bfR\sbar$. 
\end{proof}
\appendix
\label{appendix_pullback}
\renewcommand{\thesubsection}{\thesection}
\setcounter{prop}{0}
\section{Pulling back deconstructible classes}
Throughout the section we use the notions of $<\kappa$-presentable objects and locally $<\kappa$-presentable
categories as defined in \cite[Definition 1.13]{AdamekRosicky}. Note \cite[Section 1]{StovicekHillGrothendieck}
that by \cite[Remark 1.21]{AdamekRosicky} $<\kappa$-presentability is the same as $\kappa$-accessibility in the sense of
\cite[Definition 9.2.7]{KashiwaraShapira}, so it is legitimate to use results from loc.cit. when studying
$<\kappa$-presentable objects. If $\calF\subset\scA$ is a class of objects in a category
$\scA$, $\calF^{<\kappa}$ denotes the class of $<\kappa$-presentable objects in $\calF$.

We begin by recalling the definition of a monad and its category of algebras.

\begin{definition}
Let $\calC$ be a category.
\begin{enumerate}
\item A \textit{monad} on $\calC$ is a triple $(\monad,\eta,\mu)$ consisting of an endofunctor $\monad:
  \calC\to\calC$ and natural transformations $\eta: \id_\calC\to\monad$, $\mu: \monad^2\to\monad$, such
  that $\mu$ and $\eta$ obey the associativity and unit axioms $\mu\circ \monad\mu = \mu\circ
  \mu\monad$ and $\mu\circ \monad\eta = \id_{\monad} =  \mu\circ \eta\monad$.  
\item An \textit{algebra} over $\monad$ is a pair $(X,\rho)$ consisting of an object $X$ of $\calC$ and a
  morphism $\rho: \monad X\to X$ such that $\rho\circ\eta_X=\id_X$ and $\rho\circ\mu_X =
  \rho\circ\monad\rho$.
\end{enumerate}
The category of $\monad$-algebras is denoted $\monad\Alg$. If $\calF$ is a class of objects in
$\calC$, then $\monad\Alg_\calF$ denotes the class of $\monad$-algebras whose underlying objects
belong to $\calF$. The forgetful functor $\monad\Alg\to\calC$ is denoted $U$.
\end{definition}

\begin{ex}
The standard example of a monad is the following. If $F:\calD\rightleftarrows\calC:U$ is an
adjunction, then $\monad := UF$ together with the unit $\eta: \id\to UF$ and the counit $U\varepsilon F:
\monad^2=U(FU)F\to UF$ is a monad on $\calC$. 

For example, given a dg ring $A$, there is the monad associated to the adjunction $G^+:
A\Mod\rightleftarrows A\us\Mod: (-)\us$ defined in Proposition \ref{prop_adjointstoforget}. Its
category of algebras is canonically equivalent to $A\Mod$ (i.e. $(-)\us$ is a
\textit{monadic functor}). 
\end{ex}

\begin{lem}\label{lem_alggroth}
Let $\monad: \scA\to\scA$ be a right exact monad on an abelian category $\scA$.
\begin{enumerate}
\item[(1)] $\monad\Alg$ is abelian.
\item[(2)] The forgetful functor $\monad\Alg\to\scA$ is faithful and exact.
\end{enumerate}
Suppose that, in addition, $\scA$ is Grothendieck and $\monad$ is cocontinuous.
\begin{enumerate}
\item[(3)] $\monad\Alg$ is a Grothendieck category.
\item[(4)] The forgetful functor $U:\monad\Alg\to\scA$ is bicontinuous.
\end{enumerate}
\end{lem}
\begin{proof}
Since $\monad$ is additive, the sum in $\scA$ of two morphisms of $\monad$-algebras is again a morphism of
$\monad$-algebras. Hence $\monad\Alg$ inherits a unique preadditive structure from $\scA$ such that
$U:\monad\Alg\to\scA$ is preadditive. 

Next, let $D: I\to\monad\Alg$, $D(x) = (M(x),\rho_x)$, be a diagram such that the underlying $\scA$-diagram $M: I\to\scA$ has a
colimit $\projlim M$, and assume that $\monad$ commutes with that colimit, i.e. that $\monad(\projlim M)$ is a colimit for
$\monad M$ with respect to the maps $\monad M(x)\to\monad(\projlim M)$. Then there is a unique structure $\rho$ of a
$\monad$-algebra on $\projlim M$ such that all maps $(M(x),\rho_x))\to (\projlim M,\rho)$ are 
morphisms of $\monad$-algebras: take as $\rho: \monad(\projlim M)\to\projlim M$ the unique map such that
for each $x\in I$ the diagram 
\begin{equation*}\begin{tikzpicture}[description/.style={fill=white,inner sep=2pt}]
    \matrix (m) [matrix of math nodes, row sep=3em,
                 column sep=2.5em, text height=1.5ex, text depth=0.25ex,
                 inner sep=0pt, nodes={inner xsep=0.3333em, inner ysep=0.3333em}]
    {
       \monad M(x) & \monad\projlim M\\
       M(x) & \projlim M\\
    };
    \draw[->] (m-1-1) -- (m-1-2);
    \draw[->] (m-2-1) -- (m-2-2);
    \draw[->] (m-1-1) -- node[scale=0.75,left]{$\rho_x$} (m-2-1);
    \draw[->] (m-1-2) -- node[scale=0.75,right]{$\rho$} (m-2-2);
\end{tikzpicture}\end{equation*}
This is justified by our assumption that $\monad$ commutes with $\projlim M$. Unit and associativity axiom
also follow by using the universal property of the colimit, and hence $(\projlim M,\rho)$ indeed is a
$\monad$-algebra. Moreover, it is straightforward to check that $(\projlim M,\rho)$ together with the maps
$D(x)\to(\projlim M,\rho)$ is a colimit of $D$.

Similarly, if $M$ has a limit $\invlim M$ in $\scA$, then $\invlim M$ admits a unique structure $\rho$ of a
$\monad$-algebra such that all maps $(\invlim M,\rho)\to D(x)$ are morphisms of $\monad$-algebras, and, moreover, $(\invlim
M,\rho)$ is then a limit of $D$ in $\monad\Alg$ with respect to these. Note, however, that we don't have to assume
that $\monad$ commutes with $\invlim M$ here.

The preceding arguments show that for right-exact $\scA$ the category $\monad\Alg$ admits arbitrary finite limits and
colimits and $U:\monad\Alg\to\scA$ commutes with these. In particular, we get that $\monad\Alg$ is additive, and that
any morphism admits a kernel and a cokernel. Finally, since $U:\monad\Alg\to\scA$ reflects isomorphisms, we even get
that $\coimage=\image$ in $\monad\Alg$, and hence $\monad\Alg$ is abelian.

If $\scA$ admits arbitrary colimits and $\monad$ is cocontinuous, $\monad\Alg$ also admits arbitrary colimits
and $U:\monad\Alg\to\scA$ is cocontinuous, and since $U$ reflects isomorphisms, directed colimits are exact in
$\monad\Alg$ provided they are exact in $\scA$. Similarly, if $\scA$ admits arbitrary limits, then so does $\monad\Alg$
and $U$ preserves them. Finally, if $\scA$ is Grothendieck with generator $G$, the free
algebra $\monad G$ on $G$ is a generator for $\monad\Alg$. Indeed, given $(X,\rho)\in\monad\Alg$ we can choose an epimorphism
$G^{\coprod I}\to X$ by \cite[Proposition 5.2.4]{KashiwaraShapira}. Applying $\monad$, we get the
morphism of $\monad$-algebras $\monad (G^{\coprod I})\cong(\monad G)^{\coprod I}\to\monad X\to X$, which is an
epimorphism, too, since $\monad$ is cocontinuous and $\rho: \monad X\to X$ is a split epimorphism in $\scA$. Applying
\cite[Proposition 5.2.4]{KashiwaraShapira} again we conclude that $\monad G$ is a generator for $\monad\Alg$, as claimed.
\end{proof}

\begin{lem}\label{lem_completelatticehom}
Let $\scA$ be a Grothendieck category, $\monad$ be a cocontinuous monad on $\scA$ and $(X,\rho)$ be a
$\monad$-algebra. Then the forgetful functor $U:\monad\Alg\to\scA$ induces an injective complete lattice
homomorphism $$\left(\subobj_{\monad\Alg}(X,\rho),\Sigma,
  \cap\right)\longrightarrow\left(\subobj_\scA(X),\Sigma,\cap\right).$$ Its image consists of (the classes of) those 
monomorphisms $\iota: Y\hookrightarrow X$ such that the composite $\monad
Y\stackrel{\monad\iota}{\longrightarrow}\monad X\stackrel{\mu}{\longrightarrow} X$ factors through $\iota$.  
\end{lem}
\begin{proof}
Given an object $X$ in a Grothendieck category and a family $\{X_i\}$ of subobjects, the intersection $\bigcap X_i$ is
the limit of the diagram consisting of the inclusions $X_i\hookrightarrow X$, and the sum $\sum X_i$ is the image of the
canonical map $\bigoplus X_i\to X$. Hence any bicontinuous functor between Grothendieck categories, in particular
$U:\monad\Alg\to\scA$ (Lemma \ref{lem_alggroth}(3)), induces a complete lattice homomorphism on subobjects. The second
statement is clear.
\end{proof}

\begin{fact}
Let $\monad:\scA\to\scA$ be a cocontinuous monad on an abelian category $\scA$, $(X,\rho)$ be a $\monad$-algebra
and $Z\subseteq X$ a subobject of $X$. Then the poset of $\monad$-subalgebras of $(X,\rho)$ containing
$Z$ has a minimal element $\gen_\monad Z:=\image(\monad Z\to\monad X\to X)$.
\end{fact}
\begin{proof}
If $Z\p\subseteq X$ is a $\monad$-subalgebra of $(X,\rho)$ with $Z\subseteq Z\p$, then $\monad Z\p\to\monad X\to X$
factors through $Z\p$, and hence so does $\monad Z\to\monad X\to X$. Thus $\gen_\monad Z\subseteq Z\p$. 

It remains to show that $\gen_\monad Z$ is a $\monad$-subalgebra of $(X,\rho)$, i.e. that the composition
$\monad\left(\gen_\monad Z\right)\to\monad X\to X$ factors through $\gen_\monad Z$. By definition of
$\gen_\monad Z$ there is a commutative diagram
\begin{equation*}\begin{tikzpicture}[description/.style={fill=white,inner sep=2pt}]
    \matrix (m) [matrix of math nodes, row sep=3em,
                 column sep=2.5em, text height=1.5ex, text depth=0.25ex,
                 inner sep=0pt, nodes={inner xsep=0.3333em, inner ysep=0.3333em}]
    {
       \monad Z & \monad X & X\\
       & \gen_\monad Z & \\
    };
    \draw[->] (m-1-1) -- (m-1-2);
    \draw[->] (m-1-2) -- (m-1-3);
    \draw[->>] (m-1-1) -- (m-2-2);
    \draw[>->] (m-2-2) -- (m-1-3);
\end{tikzpicture}\end{equation*}
and hence it is sufficient to show that $\monad^2 Z\to\monad^2 X\stackrel{\monad\rho}
{\longrightarrow}\monad X\stackrel{\rho}{\longrightarrow} X$ factors through  
$\gen_\monad Z$. By associativity and naturality, this composition equals $\monad^2
Z\stackrel{\mu_Z}{\longrightarrow} \monad Z\to \monad X\stackrel{\rho}{\longrightarrow} X$, which factors through 
$\gen_\monad Z$ by definition.
\end{proof}

We need the following version of the generalized Hill lemma \cite[Theorem 2.1]{StovicekHillGrothendieck} as a tool
for constructing filtrations.
\begin{prop}[Hill Lemma]\label{prop_hill}
Let $\kappa$ be an infinite regular cardinal and let $\scA$ be a locally
$<\kappa$-presentable Grothendieck category. Further, let $\calS$ be a set of
$<\kappa$-presentable objects and $X\in\filt\calS$. Then there exists a
set $\sigma$ together a subset $\calL\subseteq\calP(\sigma)$ and a map
$l:\calL\to\subobj(X)$ such that the following hold:
\begin{enumerate}
\item[(H1)]\label{hill1} For any family $\{S_i\}\subset\calL$, both $\bigcup_i S_i$ and
  $\bigcap_i S_i$ belong to $\calL$ again, and we have $l\left(\bigcup_i S_i\right)
  = \sum_i l(S_i)$ and $l\left(\bigcap_i S_i\right) = \bigcap_i l(S_i)$.
\item[(H2)]\label{hill2} Given $S, T\in\calL$ with $S\subseteq T$, $l(T)/l(S)$ admits an
  $\calS$-filtration of size $|T\setminus S|$. 
\item[(H3)]\label{hill3} For any $<\kappa$-presentable $Z\subseteq X$ there exists some $S\in\calL$
  satisfying $|S|<\kappa$ and $Z\subseteq l(S)$.  
\end{enumerate}
\end{prop}

The Hill Lemma allows for recursive constructions of filtrations on $X$ by first constructing continuous
chains of elements in $\calL\subset\calP(\sigma)$ and then applying $l:\calL\to\subobj(X)$ to these chains. The
continuity of the resulting filtration is guaranteed by (H1), control over filtration quotients is given by (H2), and
finally property (H3) is needed for the recursion step. This principle is illustrated in the proof of the following
proposition, which is the main result of this section:

\begin{prop}\label{prop_pullbackmonadic}
Let $\kappa$ be an uncountable regular cardinal and $\scA$ be a locally $<\kappa$-presentable
Grothendieck category. Assume further that $\calF\subset\scA$ is a class of objects and
$\monad: \scA\to\scA$ a cocontinuous monad such that  
\begin{enumerate}
\item $\calF = \filt\calS$, where $\calS$ is a representative set of $<\kappa$-presentable objects in $\calF$,
\item $\monad$ preserves the class of $<\kappa$-presentable objects in $\scA$.
\end{enumerate}
Then $\monad\Alg_\calF = \filt\left(\monad\Alg_\calS\right)$. In particular, $\monad\Alg_\calF$ is
deconstructible. 
\end{prop}
\begin{lem}[\protect{see \cite[Proposition 9.2.10]{KashiwaraShapira}}]\label{lem_prophill} 
For any Grothendieck category $\scA$ and any infinite cardinal $\kappa$, the class $\scA^{<\kappa}$ of $<\kappa$-presentable
objects is closed under the formation of $\scA$-colimits of diagrams $I\to\scA^{<\kappa}$ with $|\Mor(I)|<\kappa$.
\end{lem}
\begin{proof}[Proof of Proposition \ref{prop_pullbackmonadic}]
Let $(X,\rho)\in\monad\Alg_\calF$. By definition we have $X\in\calF=\filt\calS$, so we may apply Proposition
\ref{prop_hill} to get $l: \calP(\sigma)\supset\calL\to\subobj(X)$ satisfying the properties (H1), (H2), (H3). By
transfinite recursion, we will now define for each ordinal $\lambda$ a subset $T(\lambda)\in\calL$ such that the
following hold: 
\begin{enumerate}
\item\label{item_monadic1} $l(T(\lambda))$ is a $\monad$-subalgebra of $X$.
\item $T(\lambda)\subseteq T(\mu)$ if $\lambda\leq\mu$, and
  $T(\lambda)\subsetneq T(\mu)$ if $\lambda<\mu$ and $l(T(\lambda))\neq X$.
\item\label{item_monadic3} $|T(\lambda+1)\setminus T(\lambda)|<\kappa$.
\item $T(\lambda) = \bigcup_{\mu<\lambda} T(\mu)$ if $\lambda$ is a limit ordinal.
\end{enumerate}
Start with $T(0) := \emptyset$ and assume that we are given an ordinal $\lambda$ such
that we already constructed $T(\mu)$ for all $\mu<\lambda$. If $\lambda$ is a
limit ordinal, we put $T(\lambda) := \bigcup_{\mu<\lambda} T(\mu)$, and if $\lambda =
\mu+1$ with $l(T(\mu))=X$, we put $T(\lambda):=T(\mu)$. In case $\lambda=\mu+1$ with $l(T)\subsetneq X$ for $T:=T(\mu)$,
we proceed as follows: Since $\scA$ is locally $<\kappa$-presentable, there exists some $<\kappa$-presentable $Z\subset X$ with 
$Z\not\subseteq l(T)$, and by (H3) we find $Z\subset l(S_0)$ for some $S_0\in\calL$ with $|S_0|<\kappa$. By Lemma
\ref{lem_prophill}, $l(S_0)$ is $<\kappa$-presentable and hence so is $\gen_\monad l(S_0) = 
\image\left(\monad l(S_0)\to\monad X\to X\right)$. Applying (H3) again, we can find $S_1\in\calL$ with $|S_1|<\kappa$,
$S_0\subseteq S_1$ and $\gen_\monad Z\subseteq l(S_1)$, and again $l(S_1)\in\scA^{<\kappa}$. Continuing this way, we 
find a sequence $S_0\subseteq S_1\subseteq S_2\subseteq ...$ in $\calP(\sigma)$ with $S_i\in\calL$,
$|S_i|<\kappa$ and $\gen_\monad l(S_i)\subseteq l(S_{i+1})$ for all $i\geq 0$. Put $S :=
\bigcup_{i\geq 0} S_i$. We then have $S\in\calL$, $|S|<\kappa$ and $l(S) = \sum_{i\geq 0} l(S_i)$ by (H1).
In particular, as $\monad$ is cocontinuous, $l(S)$ is a $\monad$-subalgebra of $(X,\rho)$. We put $T(\lambda) := T\cup
S$. This finishes the recursion step and the construction of $T$.

Pick $\lambda$ sufficiently large such that $l(T(\lambda))=X$ and consider the filtration
$l\circ T: \{\tau\ |\ \tau\leq\lambda\}\to \subobj(X)$ on $X$. By \eqref{item_monadic1} all its components are
$\monad$-subalgebras of $X$, and its successive quotients are given by $l(T(\mu+1))/l(T(\mu))$, all of which lie in
$\calS$ by \eqref{item_monadic3} and Lemma \ref{lem_prophill}. Finally, since
$\subobj_{\monad\Alg}(X)\hookrightarrow\subobj_\scA(X)$ is a complete lattice homomorphism, $l\circ T$ is also
continuous considered as a filtration of $(X,\rho)$ in $\monad\Alg$. Summing up, $l\circ T$ is the desired $\monad\Alg_\calS$-filtration of $X$. 
\end{proof}

To give a less technical version of Proposition \ref{prop_pullbackmonadic} we need some generalities about
$<\kappa$-presentable objects in Grothendieck categories.

\begin{lem}\label{lem_everythingissmall}
Let $\scA$ be a Grothendieck category.
\begin{enumerate}
\item\label{lem_small_1} For any set $\calS\subset\scA$ there exists some
  cardinal $\kappa$ such that $\calS\subseteq\scA^{<\kappa}$. 
\item\label{lem_small_2} For any cardinal $\kappa$ the category $\scA^{<\kappa}$
  is essentially small.
\end{enumerate}
\end{lem}
\begin{proof}
Part \eqref{lem_small_1} is contained in \cite[Theorem 9.6.1]{KashiwaraShapira}. Part
\eqref{lem_small_2} follows from \cite[Corollary 9.3.5(i)]{KashiwaraShapira} and the fact that
$\scA^{<\kappa}\subseteq\scA^{<\mu}$ for $\kappa\leq\mu$.  
\end{proof}

\begin{lem}\label{lem_preservesmallness}
Let $\scA$, $\scB$ be Grothendieck categories and $F:\scA\to\scB$ be a cocontinuous functor. Then there exist
arbitrarily large regular cardinals $\kappa$ such that $F$ preserves $<\kappa$-presentable objects, i.e.
$F(\scA^{<\kappa})\subseteq \scB^{<\kappa}$. 
\end{lem}
\begin{proof}
Let $G$ be a generator of $\scA$ and pick any cardinal $\kappa$ such that $G\in\scA^{<\kappa}$ and $F(G)\in\scB^{<\kappa}$
hold. This is possible by Lemma \ref{lem_everythingissmall}. Moreover, possibly after enlarging $\kappa$ we get that
$\scA^{<\kappa}=\{X\in\scA\ |\ |\Hom_\scA(G,X)|<\kappa\}$ \cite[Theorem 9.3.4]{KashiwaraShapira} (note, however, that
this characterization doesn't seem to be true for all sufficiently large, but only for a cofinal class of cardinals
$\kappa$). We claim that $F$ preserves $<\kappa$-presentable objects. Indeed, let $X\in\scA^{<\kappa}$ 
is $<\kappa$-presentable. Then the canonical morphism $G^{\coprod\Hom_\scA(G,X)}\to 
X$ is an epimorphism \cite[Proposition 5.2.3(iv)]{KashiwaraShapira}, and hence so is $F(G)^{\coprod\Hom_\scA(G,X)}\to
F(X)$ since $F$ commutes with colimits by assumption. As $F(G)\in\scB^{<\kappa}$ and $|\Hom_\scA(G,X)|<\kappa$ by
assumption, Lemma \ref{lem_prophill} implies $F(X)\in\scB^{<\kappa}$ as claimed.
\end{proof}

\begin{prop}\label{prop_corpullbackmonadic}
Let $U: \scB\to\scA$ be a cocontinuous, monadic functor between Grothendieck categories, and let $\calF\subset\scA$ be a
deconstructible class. Then $U\ua(\calF) := \{X\in\scB\ |\ U(X)\in\calF\}$ is again deconstructible.
\end{prop}
\begin{proof}
By definition of monadic functors, we may assume that $U$ is the forgetful functor $\monad\Alg\to\scA$ for a
cocontinuous monad $\monad$ on $\scA$, and then $U\ua(\calF) = \monad\Alg_\calF$. Since
$\calF=\filt\calF$ by \cite[Lemma 1.6]{StovicekHillGrothendieck}, Lemma \ref{lem_prophill} implies that $\calF =
\filt(\calF\cap\scA^{<\kappa})$ for all sufficiently large cardinals $\kappa$. Here, by slight abuse of notation
$\calF\cap\scA^{<\kappa}$ means a representative set of isomorphism classes of objects in
$\calF\cap\scA^{<\kappa}$ (it is a set by Lemma \ref{lem_everythingissmall}\eqref{lem_small_2}). Moreover, by Lemma
\ref{lem_preservesmallness} we may also assume that $\monad$ preserves $\kappa$-presentable objects, and hence the claim
follows from Proposition \ref{prop_pullbackmonadic}.   
\end{proof}

\bibliography{references}{}
\bibliographystyle{halpha}

\end{document}